\documentclass[11pt]{article}

\usepackage[mathscr]{eucal}
\usepackage[margin=1in]{geometry}
\usepackage[section]{placeins}
\usepackage{amsrefs}
\usepackage{amsmath}
\usepackage{amsthm}
\usepackage{amsfonts}
\usepackage{amssymb}
\usepackage{array}
\usepackage{hyperref}
\usepackage{enumerate} 
\numberwithin{equation}{section}
\usepackage{enumitem}
\usepackage{comment}
\usepackage{tabularray}
\usepackage{algorithm}
\usepackage{algpseudocode}
\usepackage{subcaption}
\usepackage{graphicx}
\usepackage{bm}
\usepackage[labelfont=bf]{caption}
\usepackage{tikz}
\usetikzlibrary{positioning,arrows}

\usepackage{color}

\newtheorem{theorem}{Theorem}[section]
\newtheorem{lemma}[theorem]{Lemma}
\newtheorem{proposition}[theorem]{Proposition}
\newtheorem{corollary}[theorem]{Corollary}

\theoremstyle{definition}
\newtheorem{example}[theorem]{Example}
\newtheorem{definition}[theorem]{Definition}
\newtheorem{assumption}[theorem]{Assumption}

\newtheorem{remark}[theorem]{Remark}

\newcommand{\Emb}{{\mathbb{E}}}
\newcommand{\Fmb}{{\mathbb{F}}}

\newcommand{\Imb}{{\mathbb{I}}}

\newcommand{\Zmb}{{\mathbb{Z}}}

\newcommand{\Bmc}{{\mathcal{B}}}

\newcommand{\Emc}{{\mathcal{E}}}
\newcommand{\Fmc}{{\mathcal{F}}}
\newcommand{\Gmc}{{\mathcal{G}}}
\newcommand{\Hmc}{{\mathcal{H}}}

\newcommand{\Wmc}{{\mathcal{W}}}
\newcommand{\Xmc}{{\mathcal{X}}}
\newcommand{\Ymc}{{\mathcal{Y}}}

\newcommand{\Pms}{{\mathscr{P}}}

\def\R{{\mathbb R}}

\def\N{{\mathbb N}}
\def\PP{{\mathbb P}}

\def\P{{\mathcal P}}

\newcommand{\ps}{\mathscr{P}}

\newcommand{\te}{\text}
\newcommand{\indp}{\perp\!\!\!\perp}
\newcommand{\indic}[1]{\Imb_{\left\{#1\right\}}}
\newcommand{\ex}[1]{\Emb\left[#1\right]}
		
\newcommand{\exmu}[2]{\Emb_{#1}\left[#2\right]}	
\newcommand{\deq}{\overset{\text{(d)}}{=}}		
\newcommand{\defeq}{:=}

\newcommand{\borel}{\Bmc}
\newcommand{\wh}[1]{\widehat{#1}}
\newcommand{\ind}{\hspace{24pt}}
\newcommand{\alt}[1]{\widetilde{#1}}
\newcommand{\law}{\te{Law}}

\newcommand{\ep}{\epsilon}
\newcommand{\filtm}{\Fmc}
\newcommand{\filt}{\Fmb}
\newcommand{\pcar}[1]{^{(#1)}}

\newcommand{\wsp}{\mathcal{W}}

\newcommand{\mf}{\mathcal{M}}
\newcommand{\incs}{\mathcal{I}}

\allowdisplaybreaks
\title{Mean-Field Analysis of Latent Variable Process Models on Dynamically Evolving Graphs with Feedback Effects}

\author{Ankan Ganguly\footnote{Corresponding Author}, Konstantinos Spiliopoulos\thanks{KS was partially supported by  NSF SES-2120115 and NSF-DMS 2311500.} and Daniel Sussman\\
Department of Mathematics \& Statistics, Boston University\\
665 Commonwealth Avenue, Boston, 02215, MA, United States\footnote{\emph{Email Addresses}: ankang@bu.edu (AG), kspiliop@bu.edu (KS) and sussman@bu.edu (DS)}}
\date{}
\begin{document}
\maketitle
\begin{abstract}
We study the mean-field limit of a generic class of dynamic co-evolving latent space networks motivated by the social and opinion dynamics literature. Such models include $n$ agents, whose opinions are given by latent stochastic processes, and a dynamic network process describing agent interactions. Models in this class incorporate (a) bi-directional feedback between the latent processes and the network process, (b) persistence effects, meaning that the network structure at the current time depends on the value of the latent processes at the current time but also on the network structure at the previous time instance and (c) localized interactions, meaning that individual agents do not have global information. We characterize the distributional limit of a random sample taken from the latent space network as the number of nodes in the network diverges. We describe the rich conditional probabilistic structure of the resulting limiting model which we use to establish the limiting behavior of the following quantities: (i) the empirical measure of the latent process, (ii) a conditional empirical measure relating the latent process to the network process and (iii) the network process graphon. In proving our main results, we derive a general conditional propagation of chaos result, which is of independent interest. Our novel approach to studying the limiting behavior of random samples proves to be a very useful methodology for fully grasping the asymptotic behavior of co-evolving particle systems. Numerical results are included to illustrate the theoretical findings.
\end{abstract}

\sloppy Keywords: Interacting particle systems, co-evolving networks, opinion dynamics, mean-field, graphons, propagation of chaos


\section{Introduction}
\label{sec::intro}

Interacting particle systems are a class of mathematical models used to describe a group of ``particles" or ``agents" in which interactions influence agents' behavior. Such models can be found in a large variety of fields in both the social and physical sciences, including opinion dynamics~\cite{friedkin1990social,sznajd2000opinion}, voter behavior~\cite{clifford1973model,holley1975ergodic},
herding~\cite{banerjee1992simple} or flocking~\cite{douglis1948social,frey2018cognitive}, polarization~\cite{DelVicario2017,Matakos2017}, interacting particle systems in applied mathematics and statistical physics~\cite{GartnerIPS,Garnier2, MFK2008,SirignanoSpiliopoulos2018NN_LLN}, ecology and theoretical biology ~\cite{IssacsonMS,
SirignanoSpiliopoulos2018NN_LLN}, and economics and game theory~\cite{AllenGale2000, FrickeLux2015,GSS,EisenbergNoe2001, Spiliopoulos2015} to name just a few. In such models, interactions are typically modeled using networks. However, such networks are often large and complex. This often renders these models both numerically and analytically challenging, so \emph{mean-field} approximations are often used to study such models. This paper aims to develop a mean-field theory for a class of interacting particle systems on dynamic random networks in which the particles and networks co-evolve. We also include a characterization of the network limit. 

\ind Mean-field approximation theory (see  \cite{Kur70,Kur71,Szn91,Oel84} for related early works) is a standard tool for approximating networked models (see \cite{Kol10,ChaDie22A,ChaDie22B} and the references therein). Given an exchangeable, weakly interacting particle system with $n$ agents (all pairs of agents interact, and the strength of each pairwise interaction is inversely proportional to $n$), one first establishes ``propagation of chaos," which implies that the agents are asymptotically independent of one another \cite{Szn91}. This assumption can then be used to establish a McKean-Vlasov equation (first introduced in \cite{McK66}) characterizing the limiting dynamics of each particle as $n\to\infty$. During the last few decades, there has been considerable interest in the problem of extending mean-field approximations to interacting particle systems on more general graphs and understanding the resulting limiting equations \cite{CruTan24,CopCrePha25,CredeFPha25,DelGiaLuc16,OliRei19,CopDieGia19,BudMukWu19,BayraktarChakrabortyWu2023}. 

 \ind The theory of graphons as limiting objects of dense graphs was introduced by Lov\'{a}sz and Szegedy \cite{LovSze06} and further developed by Borgs et al. in \cite{Boretal08, Boretal12} to describe the convergence of large, dense (possibly random) networks in the limit as network size increases to infinity. A given sequence of graphs $G_1, G_2, \dots$ is called to be \emph{left-convergent} if, for all finite graphs $H$, the homomorphism density $H$ with respect to $G_n$ also converges to some limit. One of the fundamental results of graphon theory states that left convergence is equivalent to convergence in the graphon space and that this limiting graphon characterizes the limiting homomorphism densities. The graphon convergence of a sequence of graphs can be used to establish (scaling) limits involving many global graph statistics of interest, such as edge-density, clustering coefficient and eigenvalues of the adjacency or laplacian matrices of the graphs \cite{Lov12}. It is worth noting that in many mean-field approximations of interacting particle systems on dense networks, graphon representations of the limiting network appear in the limiting model \cite{DelGiaLuc16, OliRei19,BayraktarChakrabortyWu2023,CaiHua21}. These connections are an area of active research. Lastly, just as graphons are useful for the study of graph asymptotics, \emph{probability-graphons} can be used to study the asymptotics of graphs with decorated edges \cite{LovSze10,AbrDelWei23,Zuc24A,Zuc24B}. One type of probability-graphon of note is the \emph{multiplexon} \cite{BhaGan25}, which can be used to describe the limiting behavior of multiplexes: collections of graphs sharing the same vertex.


\ind Our work is motivated by recent developments in the literature on social networks and polarization. Interacting particle systems are a natural choice of model for such studies. In such models, the evolution of agent beliefs (coded as latent variables) depends on their interactions (modeled as a dynamic network). In this context, it is natural to assume that agent interactions are heavily influenced by their opinions; agents tend to favor interactions with similar agents and avoid interactions with agents holding different opinions. As a result, it is natural to study dynamic co-evolving systems in which networks evolve over time, and the evolution of both the network and agent opinions depends on one another. 

\ind In the literature on social networks, this co-evolving interaction phenomenon is modeled via so-called co-evolving networks in which the links between the nodes and certain attributes of the nodes evolve over time in ways that affect each other. A non-exhaustive list of dynamic network models within the statistics and social network analysis literature includes dynamic Erd{o}s-R\`{e}nyi graph models \cite{BraHolMan22,BraHolMan23,BhaBudWu19}, dynamic stochastic block models (e.g., \cite{yang2011detecting,xu2014dynamic,corneli2016exact,matias2017statistical,zhang2017finding}), dynamic latent space network models (e.g., \cite{sarkar2006dynamic, sewell2015latent, sewell2015analysis, sewell2016latent, sewell2017latent,Loyal2023}), temporal exponential family random graph models (e.g., \cite{hanneke2010discrete,krivitsky2014separable}), stochastic actor-oriented models (SAOM,  e.g., \cite{snijders1997simulation,snijders2010introduction,snijders2017modeling}), opinion dynamics with bounded confidence models (BCOD, \cite{HegKra02,Deffuant2000,PerKerIni22}), and the recently proposed dynamic co-evolving latent space network with attractor models (CLSNA, \cite{Zhuetal23,Panetal24}).  The SAOM, BCOD, and CLSNA models are inherently co-evolving. See Examples \ref{ex::BCOD} and \ref{ex::CLSNA} for more details.

\ind The class of models we study makes a couple of structural assumptions, which add to the technical difficulty of its analysis. First, it is assumed that two agents that interact at time $t$ are (all else equal) more likely to interact at time $t+1$ than a pair of agents that don't interact at time $t$. Motivated by this phenomenon, the model includes persistence effects, meaning that the graph structure at the next time instance does not depend only on the corresponding latent process values but also on the graph structure at the current time instance. Second, it is assumed that individual agents have no global information, including the size of the population. In particular, this implies that the strength of the influence of one agent on another is inversely proportional to the second agent's degree in the network rather than the size of the entire network. 

\ind Motivated by applications in social and opinion dynamics literature, we study a generic class of models satisfying the above criteria. Our work shows that such models possess a rich mathematical structure. We also remark that the mathematical tools developed in this paper are general, and we expect them to apply to a much larger class of models than described in this article.

\ind To the best of our knowledge, our paper is among the first rigorous studies of the mean-field limit of an interacting particle system with dynamic, co-evolving networks. Indeed, models with co-evolving particles and networks are difficult to work with as the network itself is endogenous to the model. This can make it difficult to grasp properties of the underlying network. In addition, standard mean-field models are no longer sufficient, as it is also necessary to (in some way) capture the dependencies between agent beliefs and their interactions and to understand the limiting network structure. To resolve this problem, we introduce the \textit{sample perspective} for interacting particle system limits. We examine the asymptotics of a random sample of fixed size $k$ of agents in the limit as the population $n$ converges to infinity. Within this random sample, we characterize the asymptotic joint distribution of agent opinion dynamics and the subnetwork trajectory induced by the random sample. We show that this limiting random sample (which we refer to as the mean-field limit) has a rich conditional structure which we use to establish propagation of chaos (asymptotic independence of beliefs), conditional propagation of chaos (asymptotic conditional independence of beliefs between two agents given their interaction history), related hydrodynamic limits and even a graphon limit of the underlying network. Including persistence effects also greatly complicates the conditional structure of the model. Lastly, as a result of our assumption that individual agents only have access to local information, interactions between particles are non-linear functions of the local empirical measure. This induces some significant technical challenges in establishing sufficient uniform integrability. To resolve these challenges, we must establish lower bounds on the number of interactions involving any individual agent. Lastly, to properly capture the heterogeneity in the joint structure of the graph at different times, we treat the graph trajectories as multiplexes. The term multiplex refers to a collection of graphs sharing a vertex set, with each graph being a layer of the multiplex network. This point of view allows us to study and identify the asymptotic behavior of the underlying multiplexons, see Section \ref{ssec::Gconv}. 

\ind 
One relevant article introduced by Bayraktar and Wu \cite{BayWu21} studies a general continuous-time interacting particle system with state space $\Zmb$ on a dynamic, multi-colored graph which may co-evolve with its particles and which exhibits persistence, see also \cite{BAYRAKTAR_WU_2024,BayraktarChakrabortyWu2023} for related works. The authors in \cite{BayWu21} study law of large numbers, propagation of chaos and central limit theorem results given various assumptions on the dynamics of the model. 
The model we study is different, and in addition to the law of large numbers and propagation of chaos (albeit not central limit theorem), we include a graphon-level analysis of the limiting network. Furthermore, at a high level, \cite{BayWu21} takes a stochastic differential equation approach to the model, while our approach provides a stronger emphasis on the conditional structure of our model. Another artice of note by Crucianelli and Tangpi \cite{CruTan24} takes a different approach in their analysis of the convergence of interacting diffusions on (exogeneous) sparse $W$-random graphs to the limiting graphon particle system. Likewise, Coppini et al. \cite{CopCrePha25} examine the mean-field dynamics of a system of interacting diffusions with an exogenous sequence of dense graphs and for which agents only have local information. This introduces many of the same challenges we faced while establishing uniform integrability of our sampled process. Both articles use Fubini extensions to address some significant measurability problems that arise in the continuous time setting. This allows them to ``sample" an $N$ particle system from the infinite-dimensional limiting system and establish quantitative bounds on the distance between the $N$ particle interacting diffusion and the $N$ particle system sampled from the limiting graphon SDE. This differs from our sampling method as we sample from both the $N$ particle system and the limiting mean-field process. There have been other recent works examining specific examples of interacting particle systems on large, co-evolving networks on discrete state spaces, including the voter model \cite{BasSly17,Aveetal25,Baletal24,Kraetal23}, the evoSIR model \cite{DurYao22} and the Hawkes process \cite{Mac24}.


\ind Let us next present in greater detail the model we study. Consider a model for which at each time $t \in \N_0$, we model $n$ agents holding latent opinions represented by the set of vectors $\{Z^n_i(t)\}_{i=1}^n \subset \R^d$. Interactions between agents at time $t$ are encoded by a dynamic (undirected) network with adjacency matrix $A^n(t) \defeq (A^n_{ij}(t))_{i,j=1}^n$. At time $t+1$, agent $i$ updates their opinion by considering a convex combination of their own opinion and the average opinion of their neighbors at time $t$, then perturbing that opinion with i.i.d. additive noise $\{\xi_i(t)\}_{i\in \N,t\in\N_0}$:
\begin{align}
\label{eqn::Znevolve}
Z^n_{i}(t+1) &= (1 - \gamma)Z^n_{i}(t) + \gamma L^n_i(t) + \xi_i(t),\\
\label{eqn::Lnidef}
L^n_i(t) &= \frac{1}{d_{A^n(t)}(i)}\sum_{j=1}^n Z^n_j(t)A^n_{ij}(t),\\
\label{eqn::dAnt}
d_{A^n(t)}(i) &= \sum_{j=1}^n A^n_{ij}(t),
\end{align}
where $\gamma \in (0,1)$ is some fixed constant.

\ind The network $A^n$ also evolves with time. Fix any $t \in \N_0$ and $1\leq i < j \leq n$. At time $t+1$, there exists an edge between agents $i$ and $j$ with probability depending on \textbf{(a)} their latent positions $(Z^n_i(t+1),Z^n_j(t+1))$ and \textbf{(b)} the interaction between the agents $A^n_{ij}(t)$ at time $t$. This probability is determined by an a.e. continuous \emph{interaction function} $B: \{0,1\}\times\R^d\times \R^d \to [0,1]$:
\begin{equation}
\label{eqn::Anevolve}
\PP\left(A^n_{ij}(t+1) = 1\middle|\Fmc^{A,n}_t\right) = \PP\left(A^n_{ji}(t+1) = 1\middle|\Fmc^{A,n}_t\right) = B\left(A^n_{ij}(t),Z^n_i(t+1),Z^n_j(t+1)\right),
\end{equation}
where $\Fmc^{A,n}_t = \sigma(Z^n[t+1],A^n[t])$. It is also assumed that for any $i \in [1:n]$, $A^n_{ii}(t) = 1$. The assumption that all particles have self-loops ensures that the denominator in \eqref{eqn::Lnidef} is always strictly positive, and it does not impact the subgraph densities of $A^n_{ij}$ in the limit as $n\to\infty$. Note the persistence effect, in that the behavior of $A^n_{ij}(t+1)$ does not only depend on the latent processes $(Z^n_i(t+1),Z^n_j(t+1))$, but also depends on $A^n_{ij}(t)$. It can be noted that \eqref{eqn::Anevolve} implies that the interaction between two agents at time $t+1$ is determined by the latent positions of the agents at time $t+1$ given by $Z^n(t+1) \defeq (Z^n_i(t+1))_{i=1}^n$ and the network $A^n(t)$ at time $t$. Indeed, it is assumed that $\{A^n_{ij}(t+1)\}_{1\leq i < j\leq n}$ are mutually conditionally independent given $\Fmc^{A,n}_t$.
\begin{remark}
The model \eqref{eqn::Znevolve}-\eqref{eqn::Anevolve} captures the persistence, co-evolving and local interaction features we are interested in. We seek to understand how these features interact in the large network limit of this class of models. As demonstrated in this work, despite the perhaps simple formulation of the model, the conditional structure and limiting analysis of this class of models is rich. In Examples \ref{ex::BCOD}-\ref{ex::CLSNA} we provide two relevant examples of concrete applied network models from the social and opinion dynamics literature that are special cases of the model \eqref{eqn::Znevolve}-\eqref{eqn::Anevolve}.
\end{remark}
\begin{example}[BCOD Model]
\label{ex::BCOD}
When the interaction kernel $B$ takes the form
\[ B(a,z_1,z_2) = \indic{\|z_1-z_2\|_2 \leq \delta},\]
then our model becomes an \emph{opinion dynamics with bounded confidence} (BCOD) model (e.g. \cite{HegKra02,Deffuant2000,PerKerIni22} to name a few representative papers here). Specifically, in the model of \cite{HegKra02}, the authors take $\gamma=1$ and the noise component $\xi_i(t)=0$ for all agents $i$. In  \cite{Deffuant2000,PerKerIni22}, $\gamma$ is allowed to vary, $\xi_i(t)=0$  and $Z^{n}_{i}$ is restricted to be in a compact state space, say $[0,1]$, for all agents $i$.
\end{example}
 
\begin{example}[CLSNA Model]
\label{ex::CLSNA}
When the interaction kernel $B$ takes the form
\[ \te{logit}(B(a,z_1,z_2)) = \alpha + \delta a - |z_1-z_2|,\]
for constants $\alpha,\delta > 0$, then the model becomes a co-evolving latent space network with attractor model (CLSNA, \cite{Zhuetal23,Panetal24}).
 The original CLSNA model as introduced in \cite{Zhuetal23} assumes the additive noise $\xi_i(t)$ is multivariate normal with no bias and a covariance matrix of $\sigma^2I_d$ for some $\sigma > 0$. In addition, the original CLSNA model is slightly more general than our own model in that it introduces multiple particle types with both attraction and repulsion forces among different types present. However, it coincides with our model when the particles are of a single type. In forthcoming work, we leverage the mathematical machinery developed in this paper to tackle the more involved multitype particle case with dynamical interactions. We provide a numerical analysis of the CLSNA model and its mean-field limit in Section \ref{sec::Num}.
\end{example}

\ind 
Our main contributions are summarized below.
\begin{itemize}
\item We rigorously study a generic class of models that features persistence, co-evolution and local interaction.
\item We introduce the sample perspective proof methodology, which turns out to be both a natural and a powerful vehicle to comprehensively study mean-field limits of interacting particle systems featuring dynamic co-evolution of particle dynamics and the underlying network (Theorem \ref{thm::Sampleconv}). 
\item We provide an analysis of the rich conditional structure of the sample limit generated from the above methodology (Definition \ref{defn::AZlim}, Proposition \ref{prop::welldef} and Corollary \ref{coro::morestruct}). In particular, we show that the mean-field limit no longer possesses a co-evolving structure: agents' beliefs do not depend on the underlying sample subgraph, though the sample subgraph still depends on agent beliefs.
\item We establish limits of empirical distributions of the latent processes and of conditional empirical distributions describing the interaction between latent processes and the network process (Theorem \ref{thm::hydro}).
\item We derive graphon and multiplexon limits of the interaction network trajectory (Theorems \ref{thm::Antconv} and \ref{thm::Antmultconv} and Corollary \ref{coro::Wdefthm}).
\item We provide a general statement and proof of a conditional propagation of chaos result, Proposition \ref{prop::condconvres}, generalizing its classical analogue \cite[Proposition 2.2(i)]{Szn91}.
\end{itemize}

\ind 
The conditional propagation of chaos result essentially states that certain empirical measures provide consistent estimators of the conditional distribution of two agents' latent beliefs conditioned on their interaction history. A similar result is shown in \cite{BayWu21} in the context of the class of interacting particle systems studied there. To the best of our knowledge, our article is the first to establish simple, general conditions under which conditional propagation of chaos holds for generic models.

\ind The rest of the paper is organized as follows. We conclude this introduction with a short subsection on the notation that is used throughout the paper, Section \ref{ssec::notation}. In Section \ref{ssec::assures} we state our assumptions (Section \ref{sssec::assu}) and our results regarding the limiting behavior of the latent opinions model (Section \ref{sssec::samplelimit}). These results include Theorem \ref{thm::Sampleconv}, which describes the joint limit of a random sample of agents and \ref{thm::hydro}, which establishes hydrodynamic limits describing the global behavior of the model. In Section \ref{ssec::Gconv} we present our main convergence results on the latent particle network, Theorems \ref{thm::Antconv}, \ref{thm::Antmultconv} and Corollary \ref{coro::Wdefthm}. Section \ref{sec::Num} includes our numerical studies, which we use to demonstrate our theoretical results: we examine several relevant statistics for comparing the $n$-particle system to the mean-field process. In Section \ref{sec::Num} we also present a mean-field verification algorithm which approximates the limiting mean-field model that may be of independent interest. 

\ind The proofs of the main theorems are in the subsequent sections. In Section \ref{sec::condpropchaos} we present and prove our conditional propagation of chaos result, which is of independent interest. In Section \ref{sec::proppf}, we rigorously establish multiple equivalent characterizations of the conditional structure of the limiting model. In Section \ref{sec::pfsmple} we prove Theorem \ref{thm::Sampleconv}, which establishes the limiting behavior of a random sample of an $n$-agent process. This then leads to the propagation of chaos result, Theorem \ref{thm::hydro}, proven in Section \ref{ssec::hydropf}. Section \ref{sec::graphon} proves our network convergence results, Theorem \ref{thm::Antmultconv} and Corollary \ref{coro::Wdefthm}. Appendix \ref{sec::useful} contains several useful technical lemmas (including several conditional convergence lemmas, which, despite our best efforts, we were unable to find in the literature) that are used in various aspects of the proofs of the main results. We include a small schematic of the major results of the paper in Fig. \ref{fig::thmlist}.

\begin{figure}[H]
\begin{center}
\begin{tikzpicture}
\tikzset{node distance = 2cm}
\node (a) at (0,0) {Theorem \ref{thm::Sampleconv}};
\node (b) [below right = of a] {Theorem \ref{thm::hydro}};
\node (c) [above right = of a] {Theorem \ref{thm::Antmultconv}};
\node (d) [right = of c] {Theorem \ref{thm::Antconv}};
\node (e) [left = of a] {Proposition \ref{prop::condconvres}};

\draw [line width=1pt, double distance=1pt, arrows = {-stealth}] (e) -- (a);
\draw [line width=1pt, double distance=1pt, arrows = {-stealth}] (a) -- (b);
\draw [line width=1pt, double distance=1pt, arrows = {-stealth}] (a) -- (c);
\draw [line width=1pt, double distance=1pt, arrows = {-stealth}] (c) -- (d);
\end{tikzpicture}
\end{center}
\caption{Above is a list of the major results in this paper and their dependencies.}
\label{fig::thmlist}
\end{figure}

\subsection{Notation}
\label{ssec::notation}

Throughout, the sets $\Xmc$ and $\Ymc$ are always assumed to be arbitrary Polish spaces (such as Euclidean space). Same for $\Xmc_i$ or $\Ymc_i$ for $i \in I$, where $I$ is any index set. $\P(\Xmc)$ is the space of Borel probability measures on $\Xmc$ equipped with the topology of weak convergence. $C_b(\Xmc)$ is the space of bounded, continuous functions from $\Xmc$ to $\R$. $C_0(\Xmc)$ is the space of continuous functions vanishing at infinity (the uniform closure of the set of continuous functions with compact support). For any $U \subseteq \R$, $C_b(\Xmc,U)$ is the space of bounded, continuous functions from $\Xmc$ to $U$. Both $C_0(\cdot)$ and $C_b(\cdot)$ are equipped with the uniform topology. Lastly, we use $\ps(\Xmc)$ to denote the power set of $\Xmc$; this is only used for certain finite $\Xmc$ for which any measurability concerns are trivially satisfied.

\ind For integers $a < b$, let $[a:b] = \{a,a+1,\dots,b\}$ be the sequence of integers from $a$ to $b$. If $a = 0$, then we simply write $[b]\defeq [0:b]$. This notation is also used in vertex/matrix indices. For example, $x_{1:n} = (x_1,\dots,x_n)$ and $M_{1:k,1:k} = (M_{ij})_{i,j \in 1:k}$. Let $\N = \{1,2,\dots\}$ denote the natural numbers and let $\N_0 = \{0\}\cup \N$ be the set of whole numbers. For $a,b \in \R$, $a\wedge b = \min\{a,b\}$ and $a \vee b = \max\{a,b\}$. Given $x \in \R^d$, $|x| = \sqrt{x\cdot x}$. Given a function $f: \Xmc \to \R^d$, $\|f\|_{\infty} = \sup_{x \in \Xmc}|f(x)|$ is the uniform norm of $f$. $S_n$ is the set of permutations of the set $[1:n]$.

\ind Let $\mf(\Xmc)$ be the space of functions from $\N_0$ to $\Xmc$ and let $\mf_t(\Xmc)$ be the space of functions from $[t]$ to $\Xmc$. We write $x(t)$ for the value of $x$ at time $t$ and $x[a:b] \defeq (x(a),\dots,x(b))$ for the trajectory of $x$ in the time interval $[a:b]$. For $t \in \N_0$, we shorten $x[0:t]$ to $x[t]\in \mf_t(\Xmc)$. 

\ind Lastly, for $k \in \N$, we use the notation $\incs_k$ for the subset $\{(i,j): 1\leq i < j \leq k\}\subset [1:k]^2$.

\ind A summary of notation introduced later in the paper is included in Table \ref{tbl::notation}.

\begin{table}[H]
\begin{center}
\begin{tabular}{|m{2cm}|m{0.65\textwidth}|m{3.5cm}|}
\hline
Notation & Definition & Location\\\hline
$B_0$ & The conditional probability that $i$ and $j$ interact given their latent opinions. & Assumption \ref{assu::init}(b)\\\hline
$B$ & The probability that $i$ and $j$ interact given their latent opinions and the existence/non-existence of an edge between $i$ and $j$ in the previous time.& \eqref{eqn::Anevolve}\\\hline
$\alt{B}$ & The transition kernel induced by $B$.& Lemma \ref{lem::condpBtild}\vspace{2pt}\\\hline
$B_t$ & The limiting conditional probability of there being an edge between agents $i$ and $j$ at time $t$ given the latent opinions of agent $i$.& \eqref{eqn::Bsdef}\\\hline
$L^n_i$ & The averaging term in the $n$-agent evolution equation \eqref{eqn::Lnidef}.& \eqref{eqn::Lnidef}\\\hline
$L\pcar{i}$ & The limit as $n\to\infty$ of $L^n_i$. & \eqref{eqn::Lidef}\\\hline
$\mf_t(\Xmc)$ & The space of functions from $[t]$ to $\Xmc$. & Section \ref{ssec::notation}\\\hline
$\incs_k$ & the set of $1\leq i <j \leq k$. & Section \ref{ssec::notation}\\\hline
$\Fmc^n_s$ & $\sigma(Z^n[s],A^n[s])$ & Section \ref{sssec::assu}\\\hline
$\Fmc^{A,n}_s$ & $\sigma(Z^n[s],A^n[s-1])$ if $s > 0$ and $\sigma(Z^n[0])$ otherwise. & Section \ref{sssec::assu}\\\hline
$\Fmc^k_s$& $\sigma(Z\pcar{1:k}[s],A\pcar{1:k,1:k}[s])$ & Section \ref{ssssec::limit}\\\hline
$\Gmc^k_s$& $\sigma(Z\pcar{1:k}[s])$ & Section \ref{ssssec::limit}\\\hline
$\Gmc^{A,k}_s$ & $\sigma(Z\pcar{1:k}[s],A\pcar{1:k,1:k}[s-1])$ if $s > 0$ and $\sigma(Z\pcar{1:k}(0))$ otherwise. &\eqref{eqn::filtrations}\\\hline
$\Hmc^k_s$& $\sigma(Z\pcar{1:k}(s),A\pcar{1:k,1:k}(s-1))$ if $s > 0$ and $\sigma(Z\pcar{1:k}(0))$ otherwise. &\eqref{eqn::Hmc}\\\hline
\end{tabular}
\caption{Selected notation from the paper.}
\label{tbl::notation}
\end{center}
\end{table}

\subsection{Statements and Declarations}
We declare that the authors have no competing interests as defined by Springer, or other interests that might be perceived to influence the results and/or discussion reported in this paper.
\subsection{Acknowledgement}

The authors thank Giulio Zucal for bringing our attention to the references \cite{LovSze10, AbrDelWei23,Zuc24A,Zuc24B}. We also thank the reviewer of the paper for the very insightful and constructive review of the initial submitted version of this article.

\section{Assumptions and Limiting Behavior of the Latent Particle System}
\label{ssec::assures}

In this section, we outline the assumptions we make about the model as well as the asymptotic results we derive on the convergence of the latent process. 

\subsection{Assumptions}
\label{sssec::assu}

We begin by explicitly stating a few standard assumptions about the conditional structure of the process described in \eqref{eqn::Znevolve}-\eqref{eqn::Anevolve}. For each $n \in \N$ and $t \in \N_0$, recall that $\Fmc^{A,n}_t \defeq \sigma(Z^n[t],A^n[t-1])$ for $t > 0$ and $\Fmc^{A,n}_0 = \sigma(Z^n(0))$. Moreover, define $\Fmc^n_t \defeq \sigma(Z^n[t],A^n[t])$.

\begin{assumption}[Conditional Structure and Absolute Continuity]
\label{assu::cond}
The process $(Z^n,A^n,\xi)$ possesses the following properties:
\begin{enumerate}[label = (\alph*)]
\item The collection $\{\xi_i(t)\}_{i \in \N,t\in \N_0}$ is i.i.d., absolutely continuous and for each $t \in \N$, $(\xi_i(t))_{i\in\N}$ is independent of $\bigvee_{n \in \N}\filtm^n_t$.
\item For each $n\in \N$ and $t \in \N_0$, $A^n(t)$ is a symmetric $\{0,1\}$-random matrix where $\{A^n_{ij}(t)\}_{(i,j) \in \incs_n}$ is a conditionally independent collection of Bernoulli random variables given $\filtm^{A,n}_{t}$. In addition, for any $i\leq n$, $A^n_{ii}(t) = 1$.
\end{enumerate}
\end{assumption}

We assume the initial distribution of the process satisfies the following condition:
\begin{assumption}[Initial Conditions]
\label{assu::init}
The initial states $Z^n_{1:n}(0)$ and network adjacency matrix $A^n_{1:n,1:n}(0)$ satisfy the following conditions:
\begin{enumerate}[label = (\alph*)]
\item For all $n \in \N$, $Z^n_{1:n}(0)$ are exchangeable, weakly dependent and weakly convergent in the sense that there exists an absolutely continuous probability measure $\mu_0 \in \P(\R^d)$ such that $\mu^n_0 \defeq \frac{1}{n}\sum_{i=1}^n \delta_{Z^n_i(0)} \to \mu_0$ in probability;
\item There exists an a.e. continuous function $B_0: \R^d\times \R^d \to [0,1]$ such that
\[\ex{A^n_{ij}(0)} = B_0(Z^n_i(0),Z^n_j(0))\te{  for any $(i,j) \in \incs_n$}.\]
\end{enumerate}
\end{assumption}
Finally, our limiting results require the random variables $\{Z^n_i(t)\}_{n \in \N, t \leq T, i\in [1:n]}$ to be uniformly integrable for all $T < \infty$. Below, we state an assumption that ensures that the process remains uniformly integrable.
\begin{assumption}[Uniform Integrability Bounds]
\label{assu::bds}
The following assumptions hold:
\begin{enumerate}[label = (\alph*)]
\item There exists a non-decreasing mapping $M_Z: \R \to \R_+$ such that
\begin{equation}
\label{eq::MGFZ0}
\limsup_{n\to\infty} \ex{\exp\left(C|Z^n_1(0)|\right)} \defeq M_Z(C) < \infty \te{ for all }C \in \R;
\end{equation}
\item the noise terms satisfy a similar bound: there exists a non-decreasing mapping $M_{\xi}: \R\to\R_+$ such that
\begin{equation}
\label{eq::MGFxi0}
 \ex{\exp\left(C|\xi_1(0)|\right)} \defeq M_{\xi}(C) < \infty \te{ for all }C \in \R;
\end{equation}
\item the functions $B$ and $B_0$ satisfy exactly one of the two assumptions below:
\begin{enumerate}[label = (\roman*)]
\item (At most exponentially decaying interactions) for every $(a,z_1,z_2) \in \{0,1\}\times (\R^d)^2$, there exist constants $\bar{C}$ and $C_b >0$ such that
\begin{align}
\inf_{(z_1,z_2)\in\R^d\times\R^d} B_0(&z_1,z_2)\exp\left(C_b|z_1-z_2|\right) \geq\bar{C},\label{eq::B0expbd}\\
\inf_{(a,z_1,z_2)\in\{0,1\}\times\R^d\times\R^d} B(a,&z_1,z_2)\exp\left(C_b|z_1-z_2|\right) \geq\bar{C};\label{eq::Bexpbd}
\end{align} 

\item (Finite range interactions) there exists a constant $C_b > 0$ such that for all $(a,z_1,z_2) \in \{0,1\}\times\R^d\times\R^d$ such that $|z_1-z_2| > C_b$,
\begin{equation}
\label{eq::B0Bcutoff}
B_0(z_1,z_2) = B(a,z_1,z_2) = 0.
\end{equation}
Moreover, for any $(a,z) \in \{0,1\}\times\R^d$, $\min\{B_0(z,z),B(a,z,z)\} > 0$ and $(z,z)$ and $(a,z,z)$ are continuity points of $B_0$ and $B$ respectively.
\end{enumerate}
\end{enumerate}
\end{assumption}

\begin{remark}[Application of Assumption \ref{assu::bds}(c)]
\label{rem:2c}
Assumption \ref{assu::bds}(c) covers two distinct cases. In Assumption \ref{assu::bds}(c)(i), all agents may interact with positive probability regardless of their opinions. In such a case, the probability of interaction between two agents cannot decay more than exponentially fast in the difference in agent opinions, see Example \ref{ex::CLSNA}. In Assumption \ref{assu::bds}(c)(ii), we allow for agents to have no interactions if their opinions are sufficiently different (Example \ref{ex::BCOD}), but we do require interaction probabilities to be continuous and strictly positive at any point where two agents have identical opinions. In Lemma \ref{lem::posdenom}, the assumption is used to show that the denominator of the averaging term $L^n_i(t)$ is of order $\theta(n)$. In Section \ref{sssec::Abpf}, uniform integrability of $Z^n_i[t]$ is established using lower bounds in the denominator given Assumption \ref{assu::bds}(c)(i) or bounds on the averaging term $L^n_i(t)$ established via Assumption \ref{assu::bds}(c)(ii).
\end{remark}

\subsection{Limiting Model over a Randomly Sampled Subnetwork}
\label{sssec::samplelimit}

In this section, we describe the behavior as $n\to\infty$ of the latent opinions and interaction subnetwork generated by a random sample (with or without replacement) of $k$ agents, which we state rigorously in Theorem \ref{thm::Sampleconv}. Theorem \ref{thm::Sampleconv} is used to establish the limiting behavior of the empirical measure of the latent processes and a conditional empirical measure relating the latent positions of pairs of agents to the interaction history between them (Theorem \ref{thm::hydro}). We also apply Theorem \ref{thm::Sampleconv} to study limits of the network at a graphon level (Theorems \ref{thm::Antconv} and \ref{thm::Antmultconv}).

\ind In the limit, the agents of a random sample of size $k$ have i.i.d. latent opinions $Z\defeq Z\pcar{1:k}$ satisfying equations \eqref{eqn::Zevolve}-\eqref{eqn::mu2def} below. The limiting subnetwork connecting the agents given by the Bernoulli random variables $A\defeq A\pcar{1:k,1:k}$ with a rich conditional structure detailed in Definition \ref{defn::AZlim} of Section \ref{ssssec::limit}. In Section \ref{ssssec::intuition}, we provide simple heuristic arguments that motivate our description of the limiting behavior of the $k$-agent sample. In Section \ref{ssssec::limit}, we describe the limiting model $(Z\pcar{1:k}[t],A\pcar{1:k,1:k}[t])$ in detail. Then, in Section \ref{ssssec::mainres}, we rigorously state our main convergence result: that the limiting distribution of the latent opinions and subnetwork generated by a sample of size $k$ is given by $(Z\pcar{1:k}[t],A\pcar{1:k,1:k}[t])$.

\subsubsection{ Building Intuition}
\label{ssssec::intuition}

Suppose that $n$ is large, and we select a sample of $k$ agents uniformly at random with or without replacement. By a simple exchangeability argument, it can be shown that for all $n \in \N$, the joint opinion dynamics and subnetwork connecting nearly any $k$ agents have the same distribution (with deviations only when agents are chosen multiple times within the same sample). We may, therefore, assume without loss of generality that our sample consists of agents $1,\dots,k$. To keep things simple, let's assume the latent opinions and the subnetwork converge weakly:
\[(Z^n_{1:k}[t],A^n_{1:k,1:k}[t]) \Rightarrow (Z[t],A[t]).\]
What can we say about the limiting quantities?

\ind First, our assumptions ensure that for large $n$, each agent interacts with at least $cn$ other agents, where $c > 0$ is some (possibly random) constant. This suggests that interactions between any finite set of agents vanish in the limit as $n\to\infty$. In addition, we can infer that the limiting opinions $\{Z\pcar{i}[t]\}_{i \in [1:k]}$ are independent and identically distributed (by exchangability).

\ind The dynamic subnetwork described by $A[t]$ is much more interesting. Just as in the $n$ particle case, at time $t > 0$, we can construct a new subnetwork $A(t)$ at each time $t$ by setting $A\pcar{ij}(t) = 1$ with probability $B(A\pcar{ij}(t-1),Z\pcar{i}(t),Z\pcar{j}(t))$ for each $(i,j) \in \incs_k$. Moreover, just as in the $n$-particle case, the edges $\{A\pcar{ij}(t)\}_{(i,j)\in \incs_k}$ are conditionally independent given $\sigma(A[t-1],Z[t])$. However, in the $n$-particle case, $Z_i[t]$ and $Z_j[t]$ are not independent, while $Z\pcar{i}[t]$ and $Z\pcar{j}[t]$ are independent. Then for $i\neq j$ and $0 < s \leq t$,
\begin{itemize}
\item the random variables $\{A\pcar{ij}(s)\}_{(i,j)\in \incs_k}$ are conditionally independent given $\sigma(A(s-1),Z(s))$.
\item The random elements $\{A\pcar{ij}[t]\}_{(i,j)\in \incs_k}$ are conditionally independent given $\sigma(Z[t])$.
\item $A\pcar{ij}[t]$ is a conditional Markov chain given $\sigma(Z[t])$ with Markov kernel $B$.
\item The \emph{edge-process} $A\pcar{ij}[s]$ is conditionally independent of $\sigma(Z[s+1:t])$ given $\sigma(Z[s])$.
\end{itemize}

We can make use of the conditional Markov chain formulation to find the conditional probability that $A\pcar{ij}(s) = 1$ given $Z\pcar{i,j}[t]$. First, consider the function $\wh{B}: [0,1]\times \R^d\times\R^d \to \R$ defined by
\begin{equation}
\label{eqn::Bhatdef}
\wh{B}(p,x,y) \defeq  pB(1,x,y) + (1-p)B(0,x,y).
\end{equation}

Then, we actually have that  $\wh{B}(p,x,y) = \ex{B(A,x,y)}$ where $A \sim \te{Ber}(p)$. This motivates the following definition:
\begin{align}
\label{eqn::Bsdef}
B_s(Z\pcar{i}[s],Z\pcar{j}[s]) &\defeq \ex{A\pcar{ij}(s)\middle|Z\pcar{i,j}[s]} = \wh{B}\left(B_{s-1}(Z\pcar{i}[s-1],Z\pcar{j}[s-1]),Z\pcar{i}(s),Z\pcar{j}(s)\right),
\end{align}
for $0 < s\leq t$, where the definition of $B_0$ is given in Assumption \ref{assu::init}(b). This conditional structure and the recursively defined functions $\{B_s\}_{s \in \N_0}$ can be leveraged to derive the dynamics of the limiting opinions $Z[t]$. To do this, we show that for any $i\in \N$ and as $n \to \infty$, the sequence of paired random elements $\{(Z^n_j(t),A^n_{ij}(t))\}_{j=1}^n$ becomes approximately conditionally independent given $Z^n_i[t]$. Therefore, we can prove the following conditional law of large numbers results:
\begin{equation}
\label{eqn::clln}
\frac{1}{n}\sum_{j=1}^n Z^n_j(t)A^n_{ij}(t) \Rightarrow \ex{Z'(t)A'(t)\middle|Z\pcar{i}(t)} \te{ and } \frac{1}{n}\sum_{j=1}^n A^n_{ij}(t) \Rightarrow \ex{A'(t)\middle|Z\pcar{i}(t)},
\end{equation}
where $(Z\pcar{i}[t],Z'[t],A'[t])\deq (Z\pcar{1}[t],Z\pcar{2}[t],A\pcar{12}[t])$. With this, we can derive the limit of the term $L^n_i(t)$ from \eqref{eqn::Lnidef}:
\begin{align}
L^n_i(t) &= \frac{1}{d_{A^n(t)}(i)}\sum_{j=1}^n Z^n_j(t)A^n_{ij}(t) = \frac{\frac{1}{n}\sum_{j=1}^n Z^n_j(t)A^n_{ij}(t)}{\frac{1}{n}\sum_{j=1}^n A^n_{ij}(t)}\nonumber\\
&\Rightarrow \frac{\ex{Z'(t)A'(t)\middle|Z\pcar{i}[t]}}{\ex{A'(t)\middle|Z\pcar{i}[t]}} = \frac{\ex{Z'(t)B_t(Z\pcar{i}[t],Z'[t])\middle|Z\pcar{i}[t]}}{\ex{B_t(Z\pcar{i}[t],Z'[t])\middle|Z\pcar{i}[t]}} \defeq L\pcar{i}(t).\label{eqn::Liintuit}
\end{align}

From here, we can then derive the following limiting dynamics for $Z\pcar{i}$:
\begin{equation}
\label{eqn::Zintuit}
Z\pcar{i}(t+1) = (1-\gamma)Z\pcar{i}(t) + \gamma L\pcar{i}(t) + \xi_i(t).
\end{equation}

\subsubsection{A Full Characterization of the Limiting Model}
\label{ssssec::limit}

In this section, we provide a complete description of the distribution of the limiting latent opinions and subnetwork of a random sample up to some arbitrarily given time $t \in \N_0$: $(Z\pcar{1:k}[t],A\pcar{1:k,1:k}[t])$. Importantly, we provide multiple characterizations of the conditional structure of the model. For this, we make use of the following filtrations. For each $0\leq s \leq t$ define 
\begin{equation}
\label{eqn::filtrations}
\Fmc^k_s \defeq \sigma\left(Z\pcar{1:k}[s],A\pcar{1:k,1:k}[s]\right),\quad \Gmc^k_s \defeq \sigma\left(Z\pcar{1:k}[s]\right),
\nonumber
\end{equation}
\begin{equation}
\Gmc^{A,k}_t \defeq \begin{cases}
\sigma\left(Z\pcar{1:k}[s],A\pcar{1:k,1:k}[s-1]\right)&\te{ if } s > 0,\\
\sigma\left(Z\pcar{1:k}(0)\right)&\te{ if } s = 0.
\end{cases}
\end{equation}

For $0\leq s \leq t$, we also define the $\sigma$-algebra $\Hmc^k_s$ by
\begin{equation}
\label{eqn::Hmc}
\Hmc^k_s = \begin{cases}
\sigma\left(Z\pcar{1:k}(s),A\pcar{1:k,1:k}(s-1)\right) &\te{ if } s > 0,\\
\sigma\left(Z\pcar{1:k}(0)\right) &\te{ if } s= 0.
\end{cases}
\end{equation}

Combining \eqref{eqn::Bsdef}, \eqref{eqn::Liintuit} and \eqref{eqn::Zintuit}, we define $Z\pcar{i}$ as the solution to the following equations:
\begin{align}
\label{eqn::Zevolve}
Z\pcar{i}(s+1) &= (1-\gamma)Z\pcar{i}(s) + \gamma L\pcar{i}(s) + \xi_i(s),\\
\label{eqn::Lidef}
L\pcar{i}(s)&=\frac{\exmu{\mu_s^2}{Z'(s)B_s(Z\pcar{i}[s],Z'[s])\middle|Z\pcar{i}[s]}}{\exmu{\mu_s^2}{B_s(Z\pcar{i}[s],Z'[s])\middle|Z\pcar{i}[s]}}\\
\label{eqn::Bevolve}
B_{s}(z_1[s],z_2[s]) &= \begin{cases}
\hat{B}\left(B_s(z_1[s-1],z_2[s-1]),z_1(s),z_2(s)\right) &\te{ if } s > 0,\\
B_0\left(z_1(0),z_2(0)\right) &\te{ if } s = 0,
\end{cases}\\
\label{eqn::mu2def}
\mu_s^2 &\defeq \law(Z\pcar{i}[s],Z'[s]) \defeq \mu_s\otimes \mu_s,
\end{align}
where for each $s$, $\xi_{1:k}(s)$ is independent of $\Fmc^k_s$, $\{\xi_i(s)\}_{i\in [1:k],s\in[t]}$ are i.i.d., and $\mu_s = \law(Z\pcar{1}[s])$.

\ind For each $i\neq j \in [1:k]$ and $s \in [t]$, we can also define $A\pcar{ij}(s)$ to be a Bernoulli random variable defined by
\begin{equation}
\label{eqn::Aevolve}
\ex{A\pcar{ij}(s)\middle|\Gmc^{A,k}_s} =\begin{cases}
B\left(A\pcar{ij}(s-1),Z\pcar{i}(s),Z\pcar{j}(s)\right) &\te{ if } s > 0,\\
B_0\left(Z\pcar{i}(0),Z\pcar{j}(0)\right) &\te{ if } s = 0.
\end{cases}
\end{equation}

We now completely define the joint distribution of the limiting mean-field process $(Z\pcar{1:k}[t],A\pcar{1:k,1:k}[t])$.

\begin{definition}
\label{defn::AZlim} (Limiting mean-field process) We say that $(Z\pcar{1:k}[t],A\pcar{1:k,1:k}[t])$ is a limiting mean-field process (pair) if the following hold. 
Given i.i.d. initial conditions $(Z\pcar{i}(0))_{i \in [1:k]}$, $Z\pcar{i}[t]$ is the unique solution to \eqref{eqn::Zevolve}-\eqref{eqn::mu2def} for each $i \in [1:k]$. Likewise, for each $s \in [t]$ and $i\neq j \in [1:k]$, $A\pcar{ij}(s)$ is a Bernoulli random variable satisfying \eqref{eqn::Aevolve}. For any $(i,j) \in [1:k]^2$, $A\pcar{ij}[t] = A\pcar{ji}[t]$ and $A\pcar{ii}(s) = 1$ for all $s \in [t]$. Lastly, the joint distribution of $(Z\pcar{1:k}[t],A\pcar{1:k,1:k}[t])$ possesses either of the following (equivalent, as proven in Proposition \ref{prop::welldef}) conditional structures:
\begin{enumerate}[label = (\alph*)]
\item For each $s \in [t]$, the Bernoulli random variables in $\{A\pcar{ij}(s)\}_{(i,j) \in \incs_k}$ are mutually conditionally independent given $\Gmc^{A,k}_s$. 
\item We equip $\{0,1\}^{\binom{k}{2}\times (t+1)}$ with the product $\sigma$-algebra and, for $1\leq i < j \leq k$, view $A\pcar{ij}[t]$ as a $\mf_t(\{0,1\})$-random element. The random elements $\{A\pcar{ij}[t]\}_{(i,j) \in \incs_k}$ are mutually conditionally independent given $\Gmc^k_t$. In addition, for $(i,j) \in \incs_k$, $A\pcar{ij}[t]$ is a conditional Markov chain given $\Gmc^k_t$ with initial distribution and transition kernel defined by
\begin{align}
\PP\left(A\pcar{ij}(s)=1\middle|\Gmc^k_t,A\pcar{ij}(s-1)=a\right) &= B\left(a,Z\pcar{i}(s),Z\pcar{j}(s)\right)\te{ for }s > 0, a \in \{0,1\},\label{eqn::Atrans}\\
\PP\left(A\pcar{ij}(0)=1\middle|\Gmc^k_t\right) &= B_0\left(Z\pcar{i}(0),Z\pcar{j}(0)\right).\label{eqn::Ainit}
\end{align}
\item For any $s \in [t]$, the $\mf_s(\{0,1\})$-random elements in $\{A\pcar{ij}[s]\}_{(i,j) \in \incs_k}$ are mutually conditionally independent given $\Gmc^k_s$.
\end{enumerate}
\end{definition}

An explicit definition of the transition kernel given by $\PP(A\pcar{(i,j)} = a|\Gmc^{A,k}_s,A\pcar{(ij)}(s-1)=a')$ can be found in Lemma \ref{lem::condpBtild}.

\ind Throughout the article, when we reference Definition \ref{defn::AZlim}(a) (or (b) or (c)), we include the first part of the definition in the statement regarding equations \eqref{eqn::Zevolve}-\eqref{eqn::Aevolve}. Proposition \ref{prop::welldef} shows that Definition \ref{defn::AZlim} consistently and completely characterizes the joint distribution of $(Z\pcar{1:k}[t],A\pcar{1:k,1:k}[t])$ and that Definition \ref{defn::AZlim}(a),  \ref{defn::AZlim}(b)  and  \ref{defn::AZlim}(c) are indeed equivalent.

\begin{proposition}[Conditional Structure of the Limiting Model]
\label{prop::welldef}
Definitions \ref{defn::AZlim}(a), \ref{defn::AZlim}(b), and \ref{defn::AZlim}(c) individually characterize the distribution of $(Z\pcar{1:k}[t],A\pcar{1:k,1:k}[t])$. In addition, the distributions described in all three parts of Definition \ref{defn::AZlim} are equivalent.
\end{proposition}

Definition \ref{defn::AZlim} implies the following additional structure:

\begin{corollary}
\label{coro::morestruct}
The $\mf_t(\R^d)$-random elements $Z\pcar{i}[t]$, $i \in [1:k]$ are i.i.d., and for every $s \in [t]$, the random elements $A\pcar{ij}(s)$, $(i,j) \in \incs_k$, are conditionally independent given $\Hmc^k_s$ with
\begin{equation}
\label{eqn::exAH}
\ex{A\pcar{ij}(s)\middle|\Hmc^k_s} = \ex{A\pcar{ij}(s)\middle|\Gmc^{A,k}_s}.
\end{equation}
Lastly, for any $s \in [t]$,
\begin{equation}
\label{eqn::exAG}
\ex{A\pcar{ij}(s)\middle|\Gmc^k_s} = B_s\left(Z\pcar{i}[s],Z\pcar{j}[s]\right).
\end{equation}
\end{corollary}

We prove Proposition \ref{prop::welldef} and Corollary \ref{coro::morestruct} in Section \ref{sec::proppf}.

\subsubsection{Convergence Results}
\label{ssssec::mainres}

We now rigorously state the results described in Section \ref{ssssec::intuition} in terms of the limiting mean-field process described in Section \ref{ssssec::limit}. We additionally describe some hydrodynamic limits which follow as a consequence of our main result below. 

\ind We begin by stating our main result: the limiting distribution of the latent opinions of a sample of size $k$ and the resulting subnetwork. We prove this result in Section \ref{sec::pfsmple}.

\begin{theorem}[Asymptotic Distribution of a Random Sample]
\label{thm::Sampleconv}
Suppose Assumptions \ref{assu::cond}, \ref{assu::init} and \ref{assu::bds} hold. Fix any $t \in \N_0$ and for each $n \in \N$, let $M^n_k = \{m^n_1,\dots,m^n_k\}$ be a uniform random sample (with or without replacement) of size $k$ from the set $[1:n]$ so that $Z^n_{m^n_{1:k}}$ is a sample of size $k$ from the population of latent opinions in our $n$-agent model. Then, the sample latent positions and the associated network trajectory converge in distribution:
\[\left(Z^n_{m^n_{1:k}}[t],A^n_{m^n_{1:k},m^n_{1:k}}[t]\right) \Rightarrow \left(Z\pcar{1:k}[t],A\pcar{1:k,1:k}[t]\right),\]
where $\left(Z\pcar{1:k}[t],A\pcar{1:k,1:k}[t]\right)$ is defined via \eqref{eqn::Zevolve}-\eqref{eqn::Aevolve}.
\end{theorem}

We now provide a few hydrodynamic and conditional hydrodynamic limits. Using the notation from Definition \ref{defn::AZlim}, recall that for any $t \in \N$, $\mu_t = \law(Z\pcar{1}[t])$. Additionally, define
\begin{equation}
\label{eqn::muatdef}
\mu\pcar{i}_{A,t} \defeq \law\left(Z\pcar{i}[t],Z'[t],A'[t]\middle|Z\pcar{i}[t]\right),
\end{equation}
where $(Z\pcar{i},Z',A') \deq (Z\pcar{1},Z\pcar{2},A\pcar{12})$. In Section \ref{ssec::hydropf}, we prove that Theorem \ref{thm::Sampleconv} implies the following hydrodynamic limits:
\begin{theorem}[Hydrodynamic Limits]
\label{thm::hydro}
Suppose Assumptions \ref{assu::cond}, \ref{assu::init} and \ref{assu::bds} hold. Then for any $t \in \N_0$, the following hydrodynamic convergences hold:
\begin{equation}
\label{eqn::muntdef}
\mu^n_t \defeq \frac{1}{n}\sum_{i=1}^n \delta_{Z^n_i[t]} \to \mu_t\te{ in probability},
\end{equation}
and for any $i \in \N$,
\begin{equation}
\label{eqn::munaitdef}
(Z^n_i[t],\mu^n_{A,i,t}) \defeq \left(Z^n_i[t],\frac{1}{n}\sum_{j=1}^n \delta_{Z^n_i[t],Z^n_j[t],A^n_{ij}[t]}\right) \Rightarrow (Z\pcar{i}[t],\mu\pcar{i}_{A,t}).
\end{equation}
\end{theorem}

\begin{remark}
\label{rmk::munaitdefcpropchaos}
Note that \eqref{eqn::muntdef} of Theorem \ref{thm::hydro} is a standard propagation of chaos result that follows from the exchangeability of $\{Z^n_i[t]\}_{i=1}^n$ \cite{Szn91}. Likewise, we may regard \eqref{eqn::munaitdef} as a \emph{conditional} propagation of chaos result as it describes a hydrodynamic limit in which the limiting measure is a conditional law. The proof of \eqref{eqn::munaitdef} relies on the more general conditional propagation of chaos result of Proposition \ref{prop::condconvres}, we make use of the fact that the sequence $\{(Z^n_j[t],A^n_{ij}[t])\}_{j=1}^n$ is exchangeable \emph{excluding} $i$ in the sense of Definition \ref{def::exchexcli}. We discuss the concept of conditional propagation of chaos in more detail in Section \ref{sec::condpropchaos}.
\end{remark}

\section{Graphon and Multiplexon Convergence of the Latent Particle Network Process}
\label{ssec::Gconv}

In this section, we examine the interaction networks generated by the model (with adjacency matrix $A^n(t)$) and establish appropriate convergence theorems regarding these graphs. We show that the networks converge in a graphon sense and that the network trajectories (collection of networks at all times) converge in a multiplexon sense described in Section \ref{sssec::trajconv} below. We also show that the graphon limit of the interaction network between agents is determined by the function $B_t$ defined in \eqref{eqn::Bsdef}, and the multiplexon limit of the interaction network trajectory is described by the conditional Markov chain outlined in Definition \ref{defn::AZlim}(b).

\subsection{Left-Convergence of Graphs and Graphon Convergence}
\label{sssec::graphon}

There are many ways in which graphs can be said to converge. One natural method of studying graph convergence is to examine the homomorphism densities of subgraphs. Such homomorphism densities can be used to derive common graph statistics such as edge density, clustering coefficients and large/small eigenvalues of the adjacency matrix (up to normalization). If such statistics converge for a sequence of large graphs, then we may intuitively understand that the graphs possess a similar limiting structure. This notion of convergence is called \emph{left-convergence} \cite[Chapter 11]{Lov12}.

\ind We now provide a precise definition of left-convergence. Let $H = (V_H,E_H)$ and $G = (V_G,E_G)$ be two graphs. Then a \emph{homomorphism} from $H$ to $G$ is a map $\phi: V_H\to V_G$ such that if $\{u,v\} \in E_H$, then $\{\phi(u),\phi(v)\} \in E_G$. If $\te{Hom}(H,G)$ is the set of homomorphisms from $H$ to $G$, then the \emph{homomorphism density} of $H$ in $G$ \cite[Section 2]{LovSze06} is the function
\[t(H,G) \defeq \frac{\left|\te{Hom}(H,G)\right|}{|V_G|^{|V_H|}}.\]

\ind For each $n \in \N$, let $G_n$ be a graph with $n$ vertices. Then the sequence $\{G_n\}_{n \in \N}$ is said to be \emph{left-convergent} if for all simple, finite graphs $H$, $\lim_{n\to\infty} t(H,G_n)$ exists. If the sequence $\{G_n\}_{n \in \N}$ is left-convergent, then its limit can be described in terms of a \emph{graphon}, which is defined below.

\begin{definition}[Graphon Space]
\label{defn::graphon}
Let $\wsp_*$ be the space of bounded, symmetric and measurable functions $W: [0,1]^2 \to \R$ equipped with the \emph{cut norm} \cite[Section 8.2]{Lov12}:
\[\|W\|_{\square} \defeq \sup_{S,T\in\borel([0,1])} \left|\int_{S\times T} W(x,y)\,dx\,dy\right|.\]
Any $W \in \wsp_*$ such that $0\leq W\leq 1$ is called a \emph{graphon}, and the $\wsp \defeq \{W \in \wsp_*: 0\leq W \leq 1\}$ is the space of graphons.
\end{definition}

Any finite graph $G = (V,E)$ can be associated with its \emph{empirical graphon} $W^G \in \wsp$ given by
\[W^G(x,y) = \indic{\{\lceil xn\rceil,\lceil yn\rceil\} \in E},\]
where $n = |V|$. Furthermore, the function $t(H,\cdot)$ can be expanded to the space of graphons \cite[Section 2.5]{LovSze06}:
\[t(H,W) = \int_{[0,1]^{|V_H|}} \prod_{\{i,j\} \in E_H} W(x_i,x_j)\,dx_1,\dots,\,dx_{|V_H|}.\]
It's easily verified that for any finite graph $G$, $t(H,G) = t(H,W^G)$. 

\ind We say a map $\psi$ is a measure-preserving transformation (m.p.t.) if for any Borel-measurable $A \subset [0,1]$, $A$ and $\psi^{-1}(A)$ have the same Lebesgue measure. Two graphons $W_1$ and $W_2$ are said to be \emph{equivalent} ($W_1\sim W_2$) if $W_2$ lies in the closure of the set
\[\{W: W(x,y) = W_1^{\psi}(x,y)\defeq W_1(\psi(x),\psi(y))\te{ for an m.p.t. } \psi:[0,1]\to[0,1]\}.\]
This is analogous to describing two graphs as equivalent if they are isomorphic \cite[Section 8.2.2]{Lov12}. 

\begin{definition}[Graphon Classes]
\label{defn::graphonclasses}
We define the space of graphon classes $\alt{\wsp}\defeq \wsp/\sim$. If $W$ is a representative graphon of the equivalence class $\alt{W}$, and we equip $\alt{\wsp}$ with the metric
\[\delta_{\square}(\alt{W}_1,\alt{W}_2) = \delta_{\square}(W_1,W_2) \defeq \inf_{\substack{\psi: [0,1]\to [0,1]\\\te{measure-preserving}}} \|W_1 - W_2^{\psi}\|_{\square},\]
where the above equation holds independently of the choice of representative graphons $W_i\sim \alt{W}_i$, $i = 1,2$ (see Section 8.2.2 and Theorem 8.13 of \cite{Lov12}). 
\end{definition}

It is known that $\alt{\wsp}$ is a compact metric space \cite[Theorem 9.23]{Lov12}. In addition, it is an easy consequence of the famous "counting lemma," \cite[Lemma 10.22]{Lov12}, that for $W_1,W_2\sim \alt{W}\in\alt{\wsp}$ and any graph $H$, $t(H,W_1) = t(H,W_2)$, so the homomorphism density $t(H,\alt{W})$ is a well-defined quantity. With this background, we may now state Proposition \ref{prop::graphonleft}, see also \cite{LovSze06} for earlier related results.
\begin{proposition}
\label{prop::graphonleft}
\cite[Theorem 11.5]{Lov12} A sequence of graphs $\{G_n\}_{n\in \N}$ is left-convergent if and only if there exists a graphon class $\alt{W} \in \alt{\wsp}$ such that
\[\lim_{n\to\infty} \alt{W}^{G_n} = \alt{W}.\]
In addition, for all finite, simple graphs $H$,
\[\lim_{n\to\infty} t(H,G_n) = t(H,\alt{W}).\]
\end{proposition}

With this, we can state our first graph convergence result. Let $\lambda$ be the Lebesgue measure.

\begin{theorem}[Graphon Limit]
\label{thm::Antconv}
Suppose that Assumptions \ref{assu::cond}, \ref{assu::init} and \ref{assu::bds} hold, and fix any $t \in \N_0$. For each $n \in \N$, let $G_n$ be the random graph with adjacency matrix $A^n(t)$. For any $\lambda$-$\mu_t$ measure-preserving transformation $\theta_t: [0,1]\to \mf_t(\R^d)$, let $W$ be the graphon defined by
\[W(u_1,u_2) = B_t\left(\theta_t(u_1),\theta_t(u_2)\right)\te{ for all } u_1,u_2 \in [0,1].\]

Then the equivalence class $\alt{W}$ of $W$ does not depend on our choice of $\lambda$-$\mu_t$ measure-preserving transformations $\theta_t$. In addition, $\alt{W}^{G_n} \to \alt{W}$ in probability in $\alt{\wsp}$.
\end{theorem}

\begin{remark}
\label{rmk::specase}
Theorem \ref{thm::Antconv} is a special case of Theorem \ref{thm::Antmultconv} and Corollary \ref{coro::Wdefthm} which we prove in Section \ref{ssec::Antmultconvpfpf}. See Remark \ref{rmk::thmantconv} for details.
\end{remark}

\begin{remark}[Existence of a Measure-Preserving Transformation]
\label{rmk::thetatexists}
In fact, the existence of $\theta_t$ is given by \cite[Theorem 5.1]{Wil77}, which states that for any Polish space $\Xmc$ and any Borel probability measure $\eta \in \P(\Xmc)$, there exists a measurable mapping $\phi: (0,1)\to \Xmc$ such that for any $O \in \borel(\Xmc)$, $\lambda(\phi^{-1}(O)) = \eta(O)$. In particular, this implies that for any $U \sim \te{Unif}(0,1)$, $\phi(U) \sim \eta$. Then $\theta_t$ exists because $\Xmc = \mf_t(\R^d)$ is Polish and $\mu_t \in \P(\Xmc)=\P(\mf_t(\R^d))$ is a Borel probability measure, so the theorem applies. This implies that for any $U \sim \te{Unif}[0,1]$ 
 and $i \in \N$, $\theta_t(U) \deq Z\pcar{i}[t]$.
\end{remark}

\subsection{Convergence of Latent Network Trajectories as Multiplexons}
\label{sssec::trajconv}

Fix $t\in \N_0$. For each $n \in \N$ and $s \in [t]$, define the graph $G_n(s) = ([1:n],E_n(s))$ to be the graph with adjacency matrix $A^n(s)$. Given Theorem \ref{thm::Antconv}, it  follows that
\begin{equation}
\label{eqn::badconv}
\alt{W}^{G_n[t]} \defeq \left(\alt{W}^{G_n(s)}\right)_{s \in [t]} \to \alt{W}[s] \defeq (\alt{W}(s))_{s \in [t]} \te{ in probability,}
\end{equation}
where $\alt{W}(s)$ is the graphon class described in Theorem \ref{thm::Antconv}. However, this result is somewhat unsatisfying. For each $n$, the vertices of $G_n[t]$ all correspond to a fixed particle. However, \eqref{eqn::badconv} only tells us that there is a way to relabel the vertices of each graph $G_n(s)$ to get a different graph $G'_n(s)$ such that
\[W^{G'_n(s)} \to W(s) \te{ for all }s \in [t],\]
where for each $s$, $W(s)$ is some representative graphon of $\alt{W}(s)$. The problem is that the relabeling may differ for each value of $s$. This means that the latent particle associated with the vertex labeled ``1'' may be different in $G'_n(s)$ and $G'_n(s')$ for $s \neq s'$. As a result, $\alt{W}[s]$ from \eqref{eqn::badconv} does not properly capture heterogeneities in the joint structure of $G_n$ at different times.

\ind To handle this, we treat the graph trajectory of the $n$-particle system as a \emph{multiplex}, $\bm{G}_n$. A multiplex is a collection of graphs sharing a vertex set. Each graph is called a \emph{layer} of the multiplex network. For example, $\bm{G}_n$ is a multiplex network on the vertex set $[1:n]$ with $t+1$ layers where for all $s \in [t]$, the $s$th layer of the multiplex consists of the edges in $G_n(s)$. 

\begin{definition}[Cumulative Multiplex Decomposition]
\label{defn::multdecomp}
For any $(t+1)$-layer multiplex $\bm{H} = (V_H,E_H(0),E_H(1),\dots,E_H(t))$ and any $S\subseteq [t]$, let $H^S = (V_H,E_H^S)$ where
\[E_H^S = \cap_{s \in S} E_H(s).\]
So $E_H^S$ contains all edges that lie in $E_H(s)$ for all $s \in S$.
\end{definition}

\ind We can now define multiplexons, recently introduced in \cite{BhaGan25}, which are a special class of probability-graphons \cite{AbrDelWei23}. In particular, we are interested in \emph{cumulatively decomposable} multiplexons.

\begin{definition}[Cumulatively Decomposable Multiplexons]
\label{defn::cdmultiplexon}
A $(t+1)$-layer cumulatively decomposable multiplexon is a vector of graphons indexed by the subsets of $[t]$: $\mathbf{W} = (W^S)_{S\subseteq [t],S\neq\emptyset}$ such that for a.e. $(x,y) \in [0,1]^2$, there exists a measure $\mu_{xy} \in \P(\Pms([t]))$ such that
\[\sum_{S': S'\supseteq S} \mu_{x,y}(S') = W^S(x,y).\]
We define $\Wmc\pcar{t}_{\blacktriangle}$ to be the space of $(t+1)$-layer cumulatively decomposable multiplexons.
\end{definition}

In (3.5) of \cite{BhaGan25}, the pseudo-metric $\delta_{\square}^t$ is defined on $\Wmc\pcar{t}_{\blacktriangle}$ by
\[\delta_{\square}^t(\mathbf{W}_1,\mathbf{W}_2) = \inf_{\substack{\psi: [0,1]\to[0,1]\\\te{measure-preserving}}}\sum_{\substack{S \subseteq [t]\\ S\neq \emptyset}} \|W_1^S - (W_2^S)^{\psi}\|_{\square}.\]
Let $\mathbf{W}_1\sim \mathbf{W}_2$ if and only if $\delta_{\square}^t(\mathbf{W}_1,\mathbf{W}_2) = 0$. Then $\alt{\Wmc}\pcar{t}_{\blacktriangle} \defeq \Wmc\pcar{t}_{\blacktriangle}/\sim$ is a well-defined, compact metric space with respect to $\delta^t_{\square}$ \cite[Theorem 4.2]{BhaGan25}. Under this definition, the vertices of each layer of the multiplexon are labeled consistently.

\ind The empirical multiplexon of a multiplex $\bm{G}$ is given by
\[\mathbf{W}^{\bm{G}} = \left(W^{G^S}\right)_{S\subseteq [t],S\neq \emptyset}.\]
It can be shown that $\mathbf{W}^{\bm{G}} \in \Wmc\pcar{t}_{\blacktriangle}$, \cite[Remark 3.12]{BhaGan25}. We now state the following theorem which extends Theorem \ref{thm::Antconv}:

\begin{theorem}[Multiplexon Convergence]
\label{thm::Antmultconv} Suppose Assumptions \ref{assu::cond}, \ref{assu::init} and \ref{assu::bds} hold at time $0$ and fix any $t \in \N_0$. For each $n \in \N$, let $\bm{G}_n$ be the random multiplex such that for each $s \in [t]$, $G_n(s)$ is the random graph with adjacency matrix $A^n(s)$. Then there exists a multiplexon $\mathbf{W}$ (given explicitly in \eqref{eqn::Wmultdef} below) such that
\[\alt{\mathbf{W}}^{\bm{G}_n} \to \alt{\mathbf{W}} \te{ in probability.}\]
\end{theorem}

We prove Theorem \ref{thm::Antmultconv} in Section \ref{ssec::Antmultconvpfpf}.

\subsection{The Latent Network Trajectory Limit}
\label{sssec::Limit}

In Theorem \ref{thm::Antmultconv}, we show that the trajectory $\bm{G_n}$ has a limit $\mathbf{W}$ and claim that $\mathbf{W}$ is given explicitly. Here, we describe the limit.

\ind Let $\theta_t$ be a $\lambda$-$\mu_t$ measure-preserving transformation as described in the statement of Theorem \ref{thm::Antconv}. For any $(z_1,z_2) \in \mf_t\left(R^d\right)\times\mf_t\left(R^d\right)$ define the time-inhomogeneous $\{0,1\}$-valued Markov chain $\{B_{z_1,z_2}(s): s \in [t]\}$ with following initial conditions and trajectories:
\begin{align}
\PP(B_{z_1,z_2}(0) = 1) &= B_0\left(z_1(0),z_2(0)\right).\label{eq::Bzz0def}\\
\PP\left(B_{z_1,z_2}(s) = 1\middle|B_{z_1,z_2}(s-1) = b\right) &= B\left(b, z_1(s), z_2(s)\right).\label{eq::Bzzdef}
\end{align}

\begin{remark}
\label{rem::Bzzeqdist}
Notice that Definition \ref{defn::AZlim}(b) implies 
\[\law\left(B_{Z\pcar{i}[t],Z\pcar{j}[t]}[t]\middle|Z\pcar{i}[t],Z\pcar{j}[t]\right) \deq\law\left(A\pcar{ij}[t]\middle|Z\pcar{i}[t],Z\pcar{j}[t]\right)\te{ for }i\neq j.\]
\end{remark}

Then we can define the multiplexon $\mathbf{W}$ from Theorem \ref{thm::Antmultconv}:

\begin{corollary}[Multiplexon Limit]
\label{coro::Wdefthm}
Under the conditions of Theorem \ref{thm::Antmultconv}, $\mathbf{W}$ satisfies $\mathbf{W} = (W^S)_{S \subseteq [t], S\neq \emptyset}$, where for each non-empty $S \subseteq [t]$,
\begin{equation}
\label{eqn::Wmultdef}
W^S(u_1,u_2) = \PP\left(B_{\theta_t(u_1),\theta_t(u_2)}(s) = 1 \te{ for all }s \in S\right).
\end{equation}
\end{corollary}

\begin{remark}
\label{rmk::thmantconv}
As a consequence of Corollary \ref{coro::morestruct} and Definition \ref{defn::AZlim}(b), if we set $S = \{t\}$, then 
\begin{align*}
W^{\{t\}}(u_1,u_2) &= \PP\left(B_{\theta_t(u_1),\theta_t(u_2)}(t) = 1\right) = \PP\left(A\pcar{12}(t) = 1\middle|\Gmc^2_t\right)\bigg|_{Z\pcar{1:2}[t]=\theta_t(u_{1:2})}\\
&= B_t\left(\theta_t(u_1),\theta_t(u_2)\right) = W(u_1,u_2).
\end{align*}
This proves that Theorem \ref{thm::Antconv} is a special case of Theorem \ref{thm::Antmultconv} together with Corollary \ref{coro::Wdefthm}.
\end{remark}

\section{Numerical Results}
\label{sec::Num}

In this section, we provide a few numerical illustrations of the convergence results of this article. We select some useful functionals of the CSLNA model $(Z^n[T],A^n[T])$ \cite{Zhuetal23,Panetal24} described in equations (\ref{eqn::Znevolve})-(\ref{eqn::Anevolve}), and compare them to the same functionals applied to the mean-field limiting model $(Z\pcar{1:n}[T],A\pcar{1:n,1:n}[T])$ given in Definition \ref{defn::AZlim}.

\ind In Section \ref{ssec::Alg}, we provide details of the algorithms used to generate our figures. This includes a mean-field verification algorithm which approximates the limiting mean-field model. Our verification algorithm operates by using an iterative method to approximate $\mu$ from \eqref{eqn::mu2def} and then uses this approximate measure to compute the conditional expectations of \eqref{eqn::Lidef}. We also introduce a coupling between the limiting mean-field and $n$-particle model. This coupling, together with some of our numerical results, suggests that there exists some deeper structure in both models, which goes somewhat beyond the developed theory of this paper. Lastly, we describe the parameters applied in our simulations.

\ind In Section \ref{ssec::ill}, we provide our numerical simulations. We test four functionals. First, we compute the mean square error of the $n$-particle system with respect to the mean-field system. Next, we construct a network containing all edges on which the coupled $n$-particle system and the mean-field system disagree, and we plot the density of this network. This measure can be used to bound the operator norm distance between the respective adjacency matrices of the $n$-particle system and its coupled mean-field system \eqref{Eq:OperatorNormBound}. Then, we look at two global graph statistics: the average triangle density (which is useful for applications involving transitivity of interactions) and the average second-largest eigenvalue of the adjacency matrix.

\subsection{Algorithms}
\label{ssec::Alg}

\subsubsection{A Mean-Field Verification Algorithm}
\label{sssec::MFver}

In this section, we describe our mean-field verification algorithm and provide a heuristic argument for its correctness. In Section \ref{ssec::ill}, we provide numerical evidence that the mean-field algorithm is correct for a standard choice of parameters.

\ind Our mean-field verification algorithm relies on the fact that \eqref{eqn::Zevolve}-\eqref{eqn::mu2def} may be viewed as a fixed point equation for $\mu_t$. The idea is as follows: let $\wh{\mu}_t$ be any candidate distribution for $Z\pcar{i}[t]$. Let $Z\pcar{i}[t] \defeq \wh{Z}[t]$ solve \eqref{eqn::Zevolve}, \eqref{eqn::Bevolve}, \eqref{eqn::mu2def} and
\begin{equation}
\label{eqn::Lialtdef}
L\pcar{i}(t)=\frac{\exmu{\wh{\mu}_t^2}{Z'(t)B_t(Z\pcar{i}[t],Z'[t])\middle|Z\pcar{i}[s]}}{\exmu{\wh{\mu}_t^2}{B_t(Z\pcar{i}[t],Z'[t])\middle|Z\pcar{i}[t]}},
\end{equation}
where we recall that $Z'$ is an i.i.d. copy of $Z\pcar{i}$. If $\wh{\mu}_t = \law(Z\pcar{i}[t])$ for every $i$, then $\wh{\mu}_t$ solves \eqref{eqn::Zevolve}-\eqref{eqn::mu2def}, so $\wh{\mu}_t = \mu_t$. To find this fixed point solution, we simply iterate. Start with a guess $\wh{\mu}\pcar{0}_t$ such that $\wh{\mu}\pcar{0}_0 = \mu_0$. In our case, we let $\wh{Z}_{1:n}\pcar{0}[T]$ be the random walk defined by the following equation for all $t \in [T-1]$:
\begin{equation}
\label{eqn::Z0init}
\wh{Z}_{1:n}\pcar{0}(t+1) = (1 - \gamma)\wh{Z}_{1:n}\pcar{0}(t) + \xi_i(t).
\end{equation}
For each $k \in \N$, let $\wh{Z}\pcar{k+1}[t]$ solve \eqref{eqn::Zevolve}, \eqref{eqn::Bevolve}, \eqref{eqn::mu2def} and \eqref{eqn::Lialtdef} with $\wh{\mu}_t = \wh{\mu}\pcar{k}_t$. Set
\[\wh{\mu}\pcar{k+1}_t = \law(\wh{Z}\pcar{k+1}[t]).\]
Assuming $\wh{\mu}\pcar{k}_t \to \wh{\mu}_t$ for some measure $\wh{\mu}_t$, it follows that $\wh{\mu}_t$ is a fixed-point solution to \eqref{eqn::Zevolve}, \eqref{eqn::Bevolve}, \eqref{eqn::mu2def} and \eqref{eqn::Lialtdef}, so $\wh{\mu}_t = \mu_t$. We simply approximate $\law(\wh{Z}\pcar{k+1}[t])$ by the empirical measure obtained by generating $N$ i.i.d. copies $\wh{Z}\pcar{k+1}[t]$ for some sufficiently large parameter $N \in \N$. After iterating to convergence, we use the limiting reference measure $\wh{\mu}_t$ to generate our mean-field model. The algorithm is summarized below\footnote{All code used in this article is available on github: \url{https://github.com/Ankan-prob/CLSNA_single}. Simulation data is available upon request.}

\begin{algorithm}[H]
\begin{algorithmic}[1]
\Function{MeanFieldSample}{$\wh{\mu}_T, n$}
\State \Return i.i.d. $\wh{Z}_{1:n}[T]$ solving \eqref{eqn::Zevolve}, \eqref{eqn::Bevolve}, \eqref{eqn::mu2def} and \eqref{eqn::Lialtdef}.
\EndFunction
\State
\State \# Apply \Call{ReferenceSample}{} below to find the reference measure $\wh{\mu}_T$ to be used in \Call{MeanFieldSample}{}.
\Function{ReferenceSample}{$m, N$}
\State $\wh{Z}\pcar{0}_{1:N}[T] \leftarrow Z^N[T]$ where $Z^N[T]$ is generated from \eqref{eqn::Z0init}. 
\State $\wh{\mu}\pcar{0}_t \leftarrow \frac{1}{N}\sum_{j=1}^N \delta_{\wh{Z}\pcar{0}_j[T]}$.
\For{$i=1$ to $m$}
\State $\wh{Z}\pcar{i}_{1:N}[T] \leftarrow \Call{MeanFieldSample}{\wh{\mu}\pcar{i-1}_T,N}$. 
\State $\wh{\mu}\pcar{i}_T \leftarrow \frac{1}{N}\sum_{j=1}^N \delta_{\wh{Z}\pcar{i}_{j}[T]}$.
\EndFor
\State \Return $\wh{\mu}\pcar{m}_T$
\EndFunction
\end{algorithmic}
\caption{Mean-Field Algorithm}
\label{alg::MFalg}
\end{algorithm}

Once the mean-field particles $\wh{Z}_{1:n}[T] \overset{\te{(d)}}{\approx} Z\pcar{1:n}[T]$ are generated, we generate the network $A\pcar{1:n,1:n}[T]$ using the conditional Markov-chain formulation of Definition \ref{defn::AZlim}(b).

\begin{remark}
\label{rem::efficiency}
The goal of Algorithm \ref{alg::MFalg} is to test and demonstrate the validity of the mean field approximation $Z\pcar{1:n}[t]$ and not to provide an efficient algorithm to do so. Development of efficient approximation algorithms and the proof of their convergence is outside the scope of this paper. For completeness, we mention our algorithm for generating the mean-field reference measure $\wh{\mu}\pcar{m}_T$ has an asymptotic runtime of $O(N^2dTm)$, where $N$ is the number of particles in the reference measure, after which $O(NdT)$ time is required to simulate each mean-field particle using the given reference measure. In particular, $m$ represents the number of iterations of the \Call{ReferenceSample}{} function used to generate the reference measure in Algorithm \ref{alg::MFalg}. We selected $m$ large enough that we observed numerical convergence in our plots. 
It is indeed a very interesting problem to develop provably-efficient algorithms for simulating the mean-field process in such networks.
\end{remark}

\subsubsection{Coupling}
\label{sssec::coup}

In a few of our results, we examine a coupling between the mean-field model and the $n$-particle model. This coupling is relatively simple. First, we generate a reference measure $\wh{\mu}_T$ using the \Call{ReferenceSample}{} function from Algorithm \ref{alg::MFalg}. We then use the reference measure and \Call{MeanFieldSample}{$\wh{\mu},n$} to compute the mean-field particles. Lastly, we couple the $n$-particle system, $Z^n[t]$, together with a mean-field system with $n$ particles, $Z\pcar{1:n}[t]$, in the following manner:
\begin{itemize}
\item $Z\pcar{1:n}(0) = Z^n_{1:n}(0)$.
\item Both processes are generated using the same additive noise $\xi_i(t)$.
\end{itemize}

To generate the network $A^n[t]$, at each time $s$ and for each $(i,j)\in \incs_n$, we use \eqref{eqn::Anevolve} to compute the (conditional) probability that $A^n_{ij}(s) = 1$, then use i.i.d. uniform$(0,1)$ random variables $U^n_{ij}(s)$ to compute the realization of these edges. As mentioned in the previous section, the mean-field network, $A\pcar{1:n,1:n}[t]$, is computed using Definition \ref{defn::AZlim}(b), where we use the same uniform random variables $U^n_{ij}(s)$ to compute the realization of $A\pcar{ij}(s)$.

\subsubsection{Parameters}
\label{sssec::Param}

Above, we initialize $Z^n(0)$ and $Z\pcar{1:n}(0)$ to be a collection of $n$ i.i.d. standard normal random vectors in $\R^2$. The additive noise terms $\{\xi_i(t)\}_{i \in\N, t\in \N_0}$ are also i.i.d. standard normal random vectors in $\R^2$. We use a logistic link function:

\begin{align}
B_0(z_1,z_2) &= \frac{1}{1 + \exp\left(0.5\|z_1-z_2\|_2-1\right)},\\
B(a,z_1,z_2) &= \frac{1}{1 + \exp\left(0.5\|z_1-z_2\|_2-1-a\right)}.
\end{align}

All simulations are generated using the parameters given in the Table \ref{tbl::paramtable} below. The mean-field simulations are constructed using $R=20$ different independently generated reference measures $\wh{\mu}_1,\dots,\wh{\mu}_{20}$ which were each generated using $N=4000$ particles. For each reference measure $\wh{\mu}$, we use $\wh{\mu}$ to simulate $M=100$ copies of a coupled $n$-particle system and $n$-particle mean-field system for all $n \in \{10,20,50,100,200,500,1000\}$. We use multiple reference measures. 
\begin{table}[H]
\begin{center}
\begin{tabular}{|c|c|p{0.5\textwidth}|}
\hline
Parameter & Parameter Value(s) & Parameter Meaning\\\hline
$R$ & $20$ & Number of reference measures generated.\\\hline
$M$ & $100$ &Number of simulations used to compute process/network statistics.\\\hline
$N$ & $4000$ &Number of particles used to construct mean-field reference measure.\\\hline
$n$ & $10,20,50,100,200,500,1000$ & Number of particles in the simulations.\\\hline
$T$ & $100$ & The simulation runs $T$ timesteps.\\\hline
$\gamma$ & $0.3$ & See \eqref{eqn::Znevolve} and \eqref{eqn::Zevolve}.\\\hline
$m$ & $100$ & See \Call{ReferenceSample}{$m,N$} in Algorithm \ref{alg::MFalg}.\\\hline
\end{tabular}
\end{center}
\caption{This is a list of selected parameters of our numerical simulation and their values.}
\label{tbl::paramtable}
\end{table}

\subsection{Numerical Illustrations}
\label{ssec::ill}

Below, we provide simulations of a few different statistics comparing the $n$-particle process to the limiting mean-field process. For each statistic and each of $R=20$ reference measures, we calculate the average value of the statistic over $M = 100$ conditionally i.i.d. samples of the coupled particle system constructed using the reference measure. In this way, we obtain an i.i.d. sample of $20$ values of the statistic at each time. We report the mean value of these samples as the value of the statistic. Then, when needed, we use $t$-statistics to construct $95\%$ confidence intervals. Each statistic is accompanied by 3 or 4 figures. For each statistic, Figure (a) describes the evolution of the statistic in question over 100 time steps for all values of $n$ in $\{10,20,50,100,200,500,1000\}$. Figure (b) displays the same statistic but restricted to the $n = 1000$ along with a 95\% confidence band to display the spread of our data. Figure (c) displays the average value of the statistic in question for each value of $n$ along with error bars signifying 95\% confidence. The average is taken over all iterations of the simulation and at all times after the first 20 time steps. We remove the first 20 time steps to better measure each of our statistics at stationarity. For convenience, we use a log scale for $n$. Lastly, Figure \ref{fig::trianglePlot}, which measures the triangle density of the network, includes a fourth plot comparing the triangle density of the network to the triangle density of an Erd\"{o}s-R\`enyi plot with the same density as our network.

\ind To describe the measured statistics, we use the following useful notation. For each $n\in \{10,20,50,100,200,500,1000\}$, $i \in [1:n]$, $r \in [1:R]$ and $k \in [1:M]$, $Z^{n,k,r}_{i}[T]$ represents the $i$th particle of the $k$th simulation of the $n$-particle process constructed with respect to the $r$th reference measure. Likewise, $Z^{n,k,r,(i)}[T]$ represents the $i$th particle of the $k$th simulation of the mean-field process coupled to $Z^{n,k}[T]$ and constructed with respect to the $r$th reference measure. For any $t \leq T$, $A^{n,k,r}(t)$ is the adjacency matrix of the $k$th simulation of the $n$-particle process at time $t$ constructed with respect to the $r$th reference measure. Likewise, $A^{(1:n),(1:n),k,r}(t)$ represents the adjacency matrix of the $k$th simulation of the mean-field process coupled to the $n$-particle process at time $t$ and generated using the $r$th reference measure.

\vspace{8pt}

\textbf{Mean Square Error:} In Figure \ref{fig::MSEPlot}, we plot the average value of the mean-square error of the limiting mean-field model with respect to the $n$-particle process to which it is coupled. That is, for $n$ and $t \in [T]$,
\[MSE_n(t) = \frac{1}{nMR} \sum_{r=1}^R\sum_{i=1}^n\sum_{k=1}^{M} \left(Z^{n,k,r}_i(t) - Z^{n,k,r,(i)}(t)\right)^2.\]
We observe that the MSE increases linearly with time with a slope that seems to vanish as $n\to\infty$.
\begin{figure}[H]
\centering
\begin{subfigure}{0.48\textwidth}
\includegraphics[scale = 0.48]{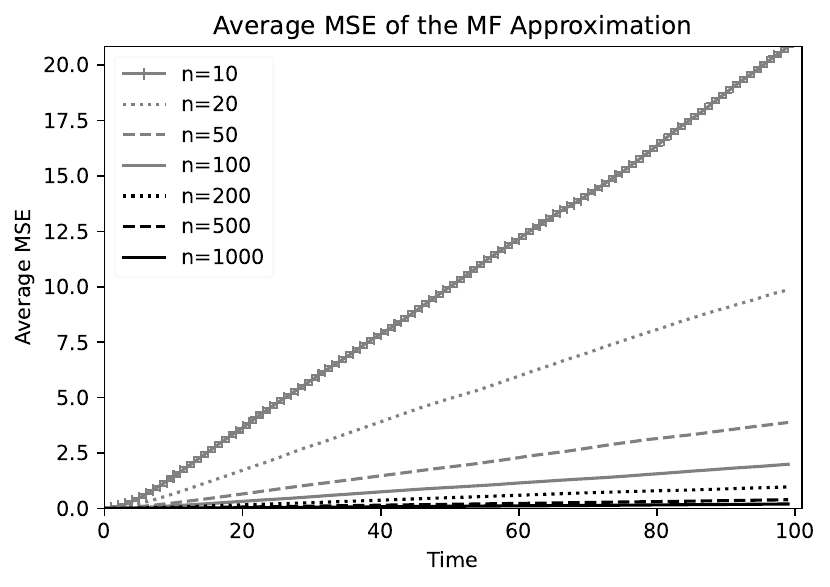}
\caption{MSE values for all $n$.}
\end{subfigure} \\ 
\begin{subfigure}{0.48\textwidth}
\includegraphics[scale = 0.48]{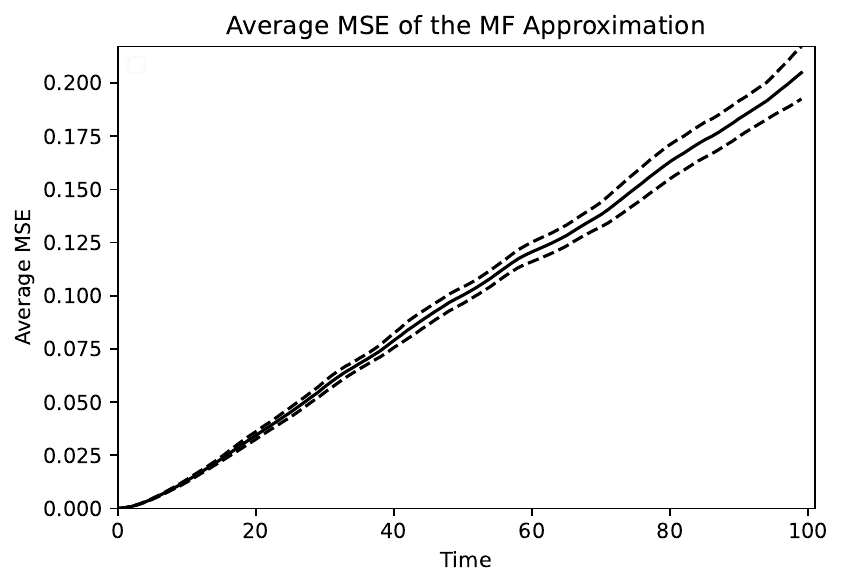}
\caption{95\% confidence bands for the $n = 1000$ case.}
\end{subfigure}~
\begin{subfigure}{0.48\textwidth}
\includegraphics[scale = 0.48]{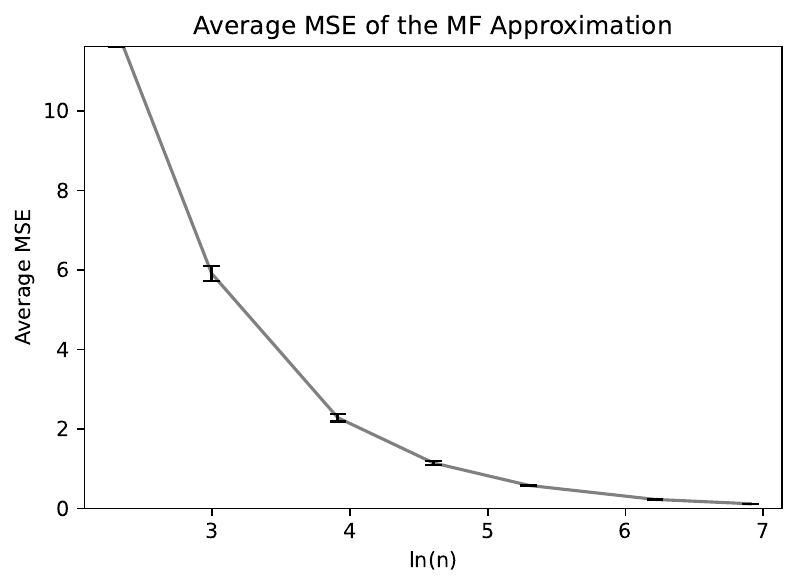}
\caption{MSE averaged over time for different $n$.}
\end{subfigure}
\caption{The Mean Square Error of the mean-field approximation of the particle trajectories averaged over all particles and simulations.}
\label{fig::MSEPlot}
\end{figure}

\textbf{Density of the Symmetric Difference Network:} For a given $n$, $t \in [T]$, $k \in [1:M]$ and $r \in [1:R]$, we define the symmetric difference network to be the graph $G^{n,k,r}_{SD}(t)$ with adjacency matrix $A^{n,k,r}_{SD}(t)$ whose edges are given by the vertex pairs on which $A^{n,k,r}(t)$ and $A^{(1:n),(1:n),k,r}(t)$ disagree. The density of the symmetric difference network is a useful measure of how different the two networks are. In fact, this measure can also be used to bound the operator distance between the respective adjacency matrices of the $n$-particle system and the mean-field model. Because the operator norm of a matrix is bounded from above by its Frobenius norm, a simple calculation shows that 
\begin{align}
\frac{1}{n}\left\|A^{n,k,r}(t) - A^{(1:n),(1:n),k,r}(t)\right\|_{\te{op}} \leq \sqrt{d(G^{n,k,r}_{SD}(t))},\label{Eq:OperatorNormBound}
\end{align}
where for any graph $G$, $d(G)$ is the density of edges in $G$.

\ind In Figure \ref{fig::symdiffPlot}, we plot the density of the symmetric difference network. Our coupling ensures $A^{n,k,r}(0)=A^{(1:n),(1:n),k,r}(0)$, so this network is initialized by the empty graph. The average graph density is given by:
\[d_n(t) \defeq \frac{2}{n(n-1)MR}\sum_{r=1}^R\sum_{k=1}^M \sum_{(i,j)\in \incs_{n}} A^{n,k,r}_{\te{SD},ij}(t).\]
\begin{figure}[H]
\centering
\begin{subfigure}{0.48\textwidth}
\includegraphics[scale = 0.48]{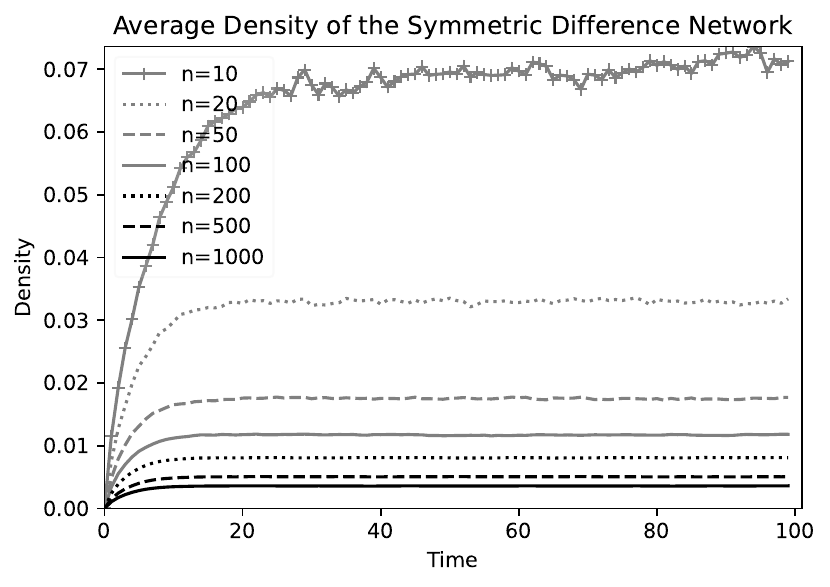}
\caption{Graph densities for all $n$.}
\end{subfigure}\\
\begin{subfigure}{0.48\textwidth}
\includegraphics[scale = 0.48]{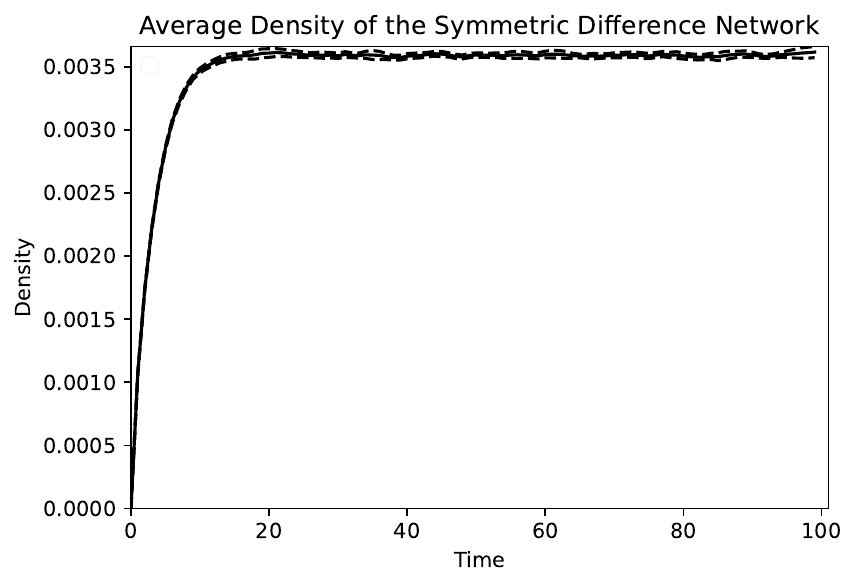}
\caption{95\% confidence bands for the $n = 1000$ case.}
\end{subfigure} ~
\begin{subfigure}{0.48\textwidth}
\includegraphics[scale = 0.48]{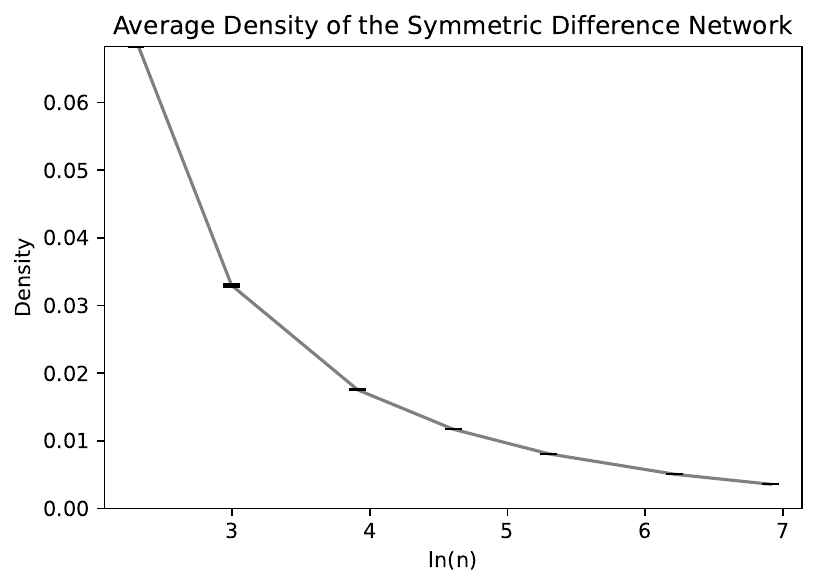}
\caption{Symmetric difference graph density averaged over time for different $n$.}
\end{subfigure}
\caption{The density of edges in the symmetric difference network $G^{n,k}_{SD}(t)$.}
\label{fig::symdiffPlot}
\end{figure}

Figure \ref{fig::symdiffPlot} is actually quite interesting as it suggests some structure in the limiting mean-field process that goes beyond what we proved in the theoretical section. For example, for moderate values of $n$, the density of the symmetric difference network rapidly stabilizes around 10 time steps into the simulation. After that, it remains at a consistent value for the remaining time. These values are also incredibly stable: the confidence band in Figure \ref{fig::symdiffPlot}(b) and the error bars in Figure \ref{fig::symdiffPlot}(b) are barely visible. This suggests that the mean density of the symmetric difference graph reaches stationarity very quickly, and its stationary distribution has very small fluctuations. That said, the stationary density appears to converge to $0$ as $n$ increases, which suggest that the interaction networks between agents of the two models differ by $o(n^2)$ different edges. It is worth noting that this form of convergence is slightly stronger than the convergence we proved in Theorems \ref{thm::Antconv} and \ref{thm::Antmultconv}.


\vspace{8pt}

\textbf{Triangle Density Errors:} In Figure \ref{fig::trianglePlot}, we plot the average difference in the homomorphism density of triangles in the limiting mean-field process with respect to the homomorphism density of triangles in the $n$-particle system. More specifically, we measure
\[T_n(t) \defeq \frac{1}{MR}\sum_{r=1}^R\sum_{k=1}^M \left(t\left(C_3,A^{(1:n),(1:n),k,r}(t)\right) - t\left(C_3,A^{n,k,r}(t)\right)\right),\]
where $C_3$ is the 3-cycle (or triangle) and $t$ is the homomorphism density function introduced in Section \ref{sssec::graphon}.
\begin{figure}[H]
\centering
\begin{subfigure}{0.48\textwidth}
\includegraphics[scale = 0.48]{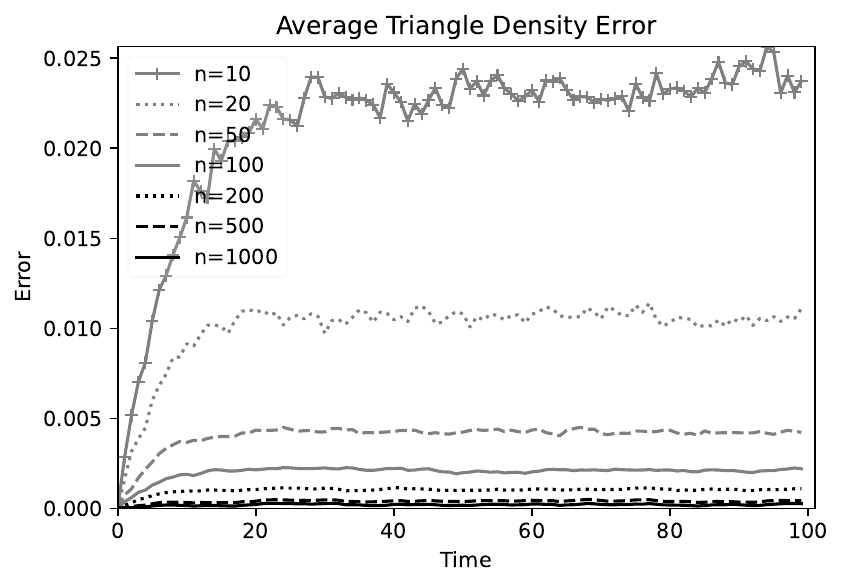}
\caption{Triangle densities errors for all $n$.}
\end{subfigure} ~ 
\begin{subfigure}{0.48\textwidth}
\includegraphics[scale = 0.48]{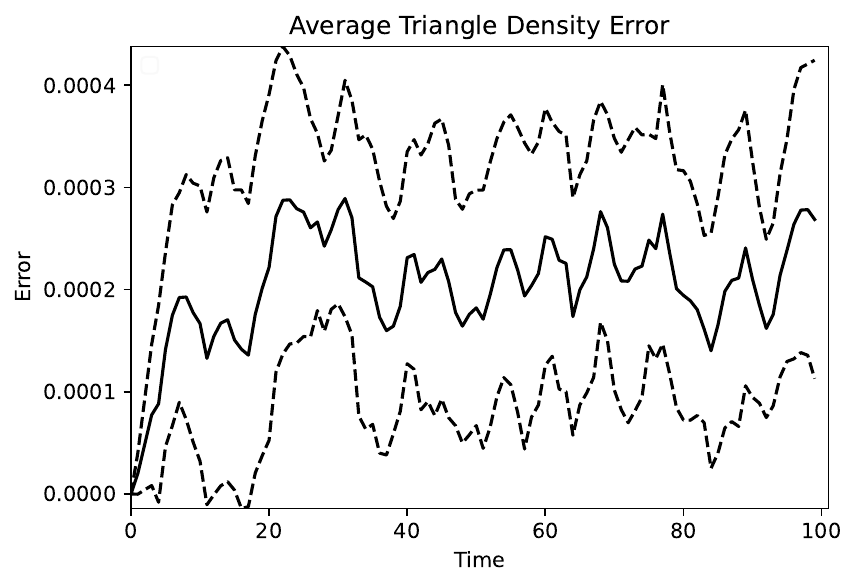}
\caption{95\% confidence bands for the $n = 1000$ case.}
\end{subfigure}\\
\begin{subfigure}{0.48\textwidth}
\includegraphics[scale = 0.48]{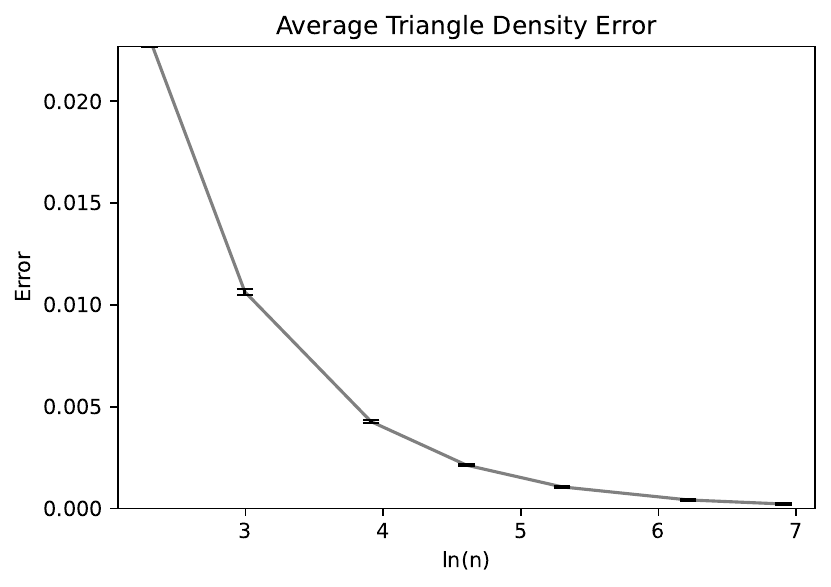}
\caption{Triangle density averaged over time for different $n$.}
\end{subfigure}~
\begin{subfigure}{0.48\textwidth}
\includegraphics[scale = 0.48]{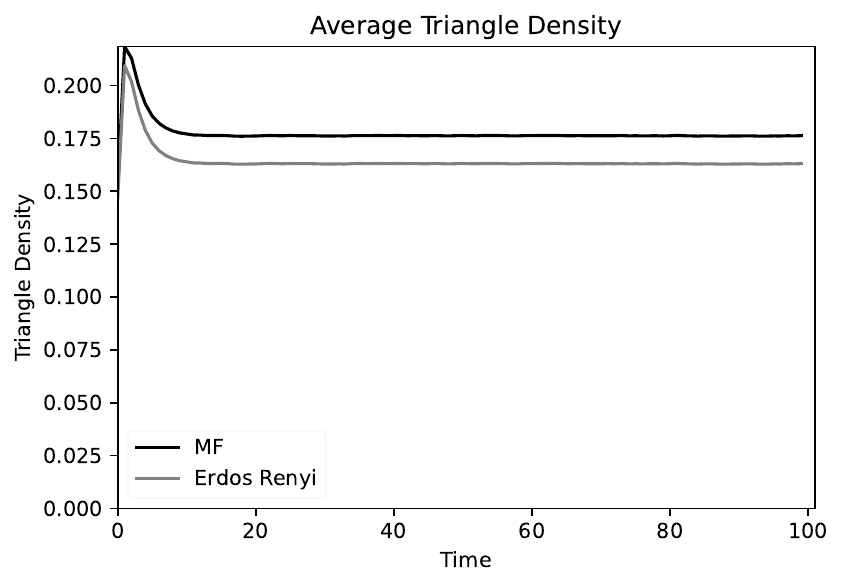}
\caption{Triangle densities of the mean-field network and comparable Erd\"{os}-R\`{e}nyi graphs.}
\end{subfigure}
\caption{In Figures \ref{fig::trianglePlot}(a)-(c), we examine the average signed error we would find if we tried to approximate the homomorphism densities of triangles in the $n$-particle system using the limiting mean-field process. In Figure (d), we compare the homomorphism density of triangles in the mean-field model vs. an Erd\"{os}-R\`{e}nyi graph with the same edge density as the mean-field model.}
\label{fig::trianglePlot}
\end{figure}

\ind It is an easy consequence of Theorem \ref{thm::Antconv} that these errors should converge to $0$ as $n \to \infty$, which we observe. What is interesting is that for finite $n$, the mean-field network generated by our algorithm seems to overestimate the triangle density of the network compared to the $n$-particle system. Indeed, when we checked the average triangle densities for coupled systems generated from individual reference measures, we found the monotonicity of the errors held for every reference measure we generated.

\ind Figure \ref{fig::trianglePlot}(d) shows that the mean-field model with $n = 1000$ has a higher triangle density than we would expect to see from an Erd\"os-R\`enyi graph with the same edge density. This suggests that there are correlations between edges of the $n$-particle network that do not vanish in the large $n$ limit. Indeed, these non-vanishing triangle correlations are consistent with the conditional Markovian dependence of the edges in the graph on the latent opinions $Z\pcar{1:n}$.

\vspace{8pt}

\textbf{Second Largest Eigenvalue Errors:} The second largest eigenvalue of the adjacency matrix of a network is a statistic with a large variety of applications.
Generally, the leading eigenvalues of the adjacency matrix can be used to assess model fit for stochastic blockmodels \cite{Bickel2016-me,Fishkind2013-kl} and other network models \cite{Chen2021-kn, Athreya2018-ww}.
The second largest eigenvalue is specifically associated with deviations from rank-1 models and is loosely related to the second smallest eigenvalue of the graph Laplacian, also known as the algebraic connectivity of a network.
For a given matrix $A$, let $\lambda_2(A)$ be the second largest eigenvalue of $A$. In Figure \ref{fig::2ndeig}, we examine the signed error of the mean-field approximation of the scaled second largest eigenvalue of the $n$-particle network:
\[\te{EIG}_n(t) \defeq \frac{1}{MR}\sum_{r=1}^R\sum_{k=1}^M \frac{1}{n}\left(\lambda_2\left(A^{(1:n),(1:n),k,r}(t)\right) - \lambda_2\left(A^{n,k,r}(t)\right)\right).\]

\ind It is a standard result that the appropriately scaled leading eigenvalues of the adjacency matrix of a network are continuous in the cut topology in the sense that $\frac{1}{n}\lambda_k(G_n) \to \lambda_k(W)$ when $G_n \to W$ in the cut topology (and $\lambda_k$ is the $k$th largest $L^2$ eigenvalue of $W$) \cite[Theorem 11.53]{Lov12}.

\ind Just as for the triangle density errors we measured in Figure \ref{fig::trianglePlot}, Figure \ref{fig::2ndeig} suggests that the error in the (scaled) second eigenvalue increases monotonically to a deterministic limit for large $n$. We likewise checked the same plot for the data generated from each reference measure individually and observed the same behavior but with a different limit. 
\begin{figure}[H]
\centering
\begin{subfigure}{0.48\textwidth}
\includegraphics[scale = 0.48]{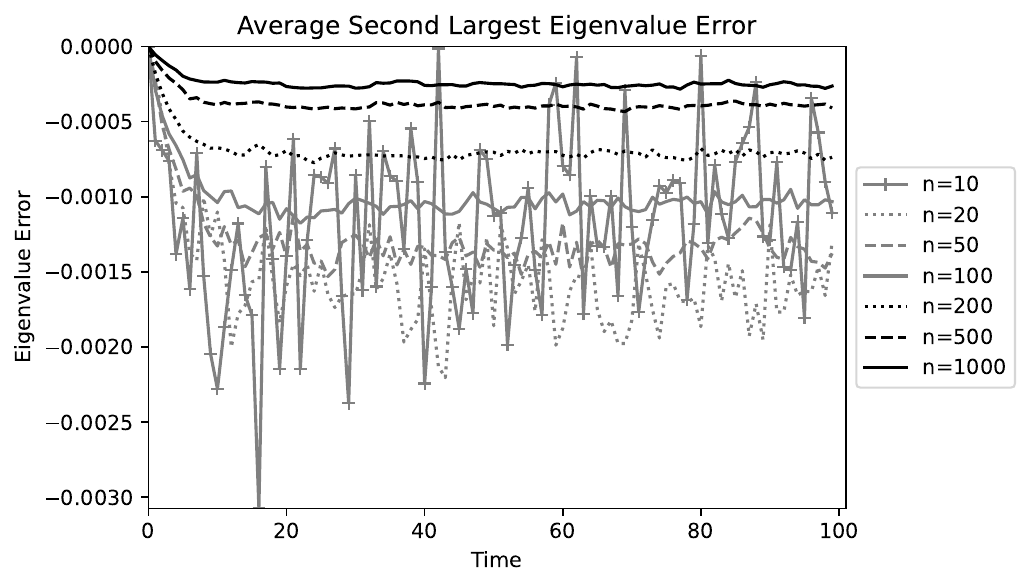}
\caption{2nd largest eigenvalue errors for all $n$.}
\end{subfigure} \\
\begin{subfigure}{0.48\textwidth}
\includegraphics[scale = 0.48]{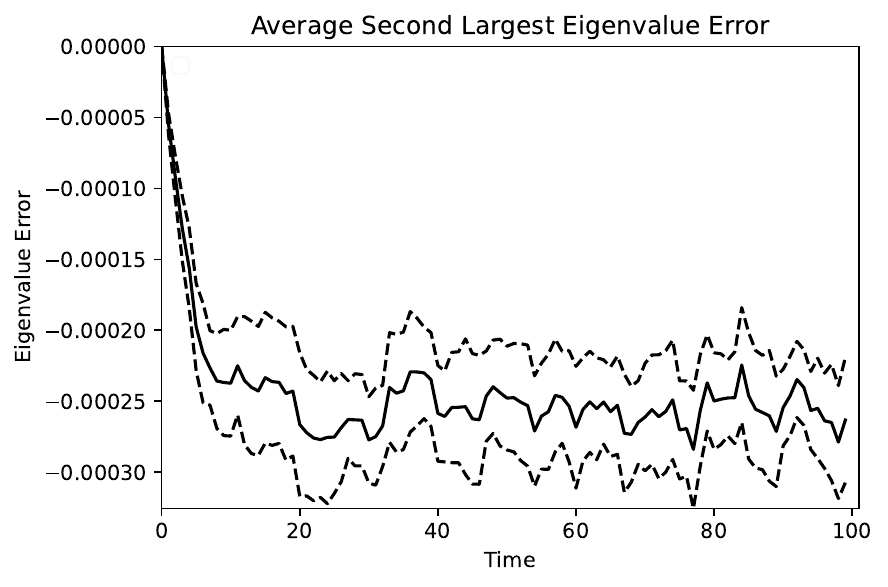}
\caption{95\% confidence bands for the $n = 1000$ case.}
\end{subfigure}~
\begin{subfigure}{0.48\textwidth}
\includegraphics[scale = 0.48]{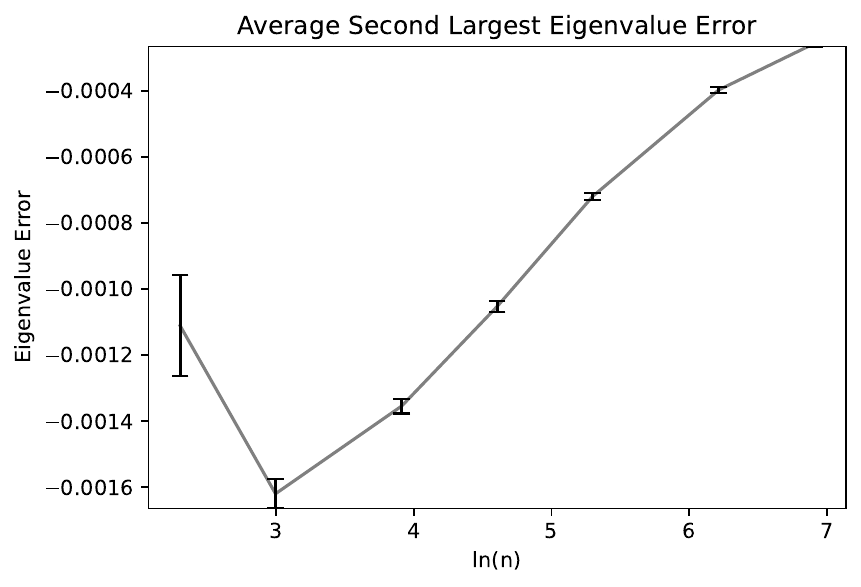}
\caption{2nd largest eigenvalues averaged over time for different $n$.}
\end{subfigure}
\caption{The average difference in the 2nd largest eigenvalue of the mean-field model and the $n$ particle model (divided by $n$).}
\label{fig::2ndeig}
\end{figure}

\section{Conditional Propagation of Chaos}
\label{sec::condpropchaos}

In this section, we introduce a conditional propagation of chaos result, which is of independent interest. We begin with a classical descriptor of propagation of chaos:

\begin{proposition}[Propagation of Chaos]
\label{prop::convres}
Let $\{X^n_{1:n}\}_{n \in \N}$ be a triangular array of $\Xmc$-random elements such that for all $n \in \N$, $(X^n_i)_{i=1}^n$ is exchangeable. Suppose $\eta \defeq \law(X)\in \P(\Xmc)$ is a deterministic probability measure. Then,
\[\eta_X \defeq \frac{1}{n}\sum_{i=1}^n \delta_{X^n_i} \to\eta \te{ in probability}\]
if and only if for any $k \in \N$,
\[(X^n_1,\dots,X^n_k) \Rightarrow (X\pcar{1},\dots,X\pcar{k}),\]
where $X\pcar{1},\dots,X\pcar{k}$ are i.i.d. copies of $X$. Moreover, the ``if'' statement of this lemma holds even when $k = 2$.
\end{proposition}
\begin{proof}
This is just a restatement of \cite[Proposition 2.2(i)]{Szn91} in our notation.
\end{proof}

As mentioned in Remark \ref{rmk::munaitdefcpropchaos}, this describes \eqref{eqn::muntdef} of Theorem \ref{thm::hydro}. However, it is not sufficient to establish \eqref{eqn::munaitdef} of the same theorem. In this section, we extend Proposition \ref{prop::convres} to a more general case, which we call \emph{conditional propagation of chaos}. Conditional propagation of chaos plays a key role in the proof of Theorem \ref{thm::Sampleconv}, as it is useful for establishing conditional law of large numbers result described in \eqref{eqn::clln}, which is then used to establish the weak limit of $(Z^n_i[t],L^n_i(t))$, and therefore of $Z^n_i[t+1]$. To state the generalization of Proposition \ref{prop::convres}, we require a few definitions.

\begin{definition}[Continuous Dependence]
\label{def::depcont}
Let $X$ and $Y$ be $\Xmc$ and $\Ymc$-random elements respectively. Then we say $Y$ \emph{depends continuously on} $X$ if there exists an $X$-almost surely continuous function $\phi: \Xmc \to \Ymc$ such that $\phi(X) = Y$ almost surely.
\end{definition}

It is important to note that in \eqref{eqn::munaitdef}, exchangeability notably fails. Specifically, the collection of random elements $(Z^n_j(t),A^n_{ij}(t))_{j=1}^n$ are not exchangeable. This is because $(Z^n_i(t),A^n_{ii}(t)) \overset{\te{(d)}}{\neq} (Z^n_j(t),A^n_{ij}(t))$ for $i \neq j$. Ultimately this does not matter as the collection $(Z^n_j(t),A^n_{ij}(t))_{j\neq i}$ \emph{is} exchangeable and the $\frac{1}{n}\delta_{Z^n_i[t],Z^n_i[t],A^n_{ii}[t]}$ term in \eqref{eqn::munaitdef} vanishes as $n \to \infty$. Taking this into account, we define a new notion of exchangeability. 

\begin{definition}[Exchangeability Excluding $i$]
\label{def::exchexcli}
For any $n \in \N$ and $i \in [1:n]$, a $\Xmc$-random element $X$ and a collection of $n$ $\Ymc$-random elements $Y \defeq (Y_1,\dots,Y_n)$ are said to be \emph{exchangeable excluding} $i$ if for every permutation $\sigma \in S_n$,
\[\left(X,(Y_j)_{j=1}^n\right) \deq \left(X,(Y_{\sigma(j)})_{j=1}^n\right)\te{ when for } j=i, \sigma(i) = i.\] 
\end{definition}

It is convenient to combine Definitions \ref{def::depcont} and \ref{def::exchexcli} into a single property.

\begin{definition}[$\Xmc/\Ymc$-Convenience]
\label{def::usefulassu}
Let $\mathbf{X} \defeq \{X^n\}_{n \in \N}\cup \{X\}$ be a sequence of $\Xmc$-valued random elements and let $\mathbf{Y} \defeq \{(Y^n_i)_{i\in [1:n]}\}_{n \in \N}\cup\{Y\}$ be a collection of random vectors with entries in $\Ymc$. Then $(\mathbf{X},\mathbf{Y})$ is said to be $\Xmc/\Ymc$-\emph{convenient} if for each $n \in \N$, $(X^n,Y^n)$ is exchangeable excluding 1 and $\eta_{XY} \defeq \law(X,Y|X)$ depends continuously on $X$. 
\end{definition}

We now state our conditional propagation of chaos result, which we prove in Section \ref{ssec::cppf}.

\begin{proposition}[Conditional Propagation of Chaos]
\label{prop::condconvres}
Suppose that $\Xmc$ and $\Ymc$ are locally compact. Let $(\mathbf{X},\mathbf{Y})$ be an $\Xmc/\Ymc$-convenient collection of random elements. Then
\begin{align}
\label{eqn::muXYiconv}
(X^n,\eta^n_{XY})&\defeq \left(X^n,\frac{1}{n}\sum_{j=1}^n \delta_{X^n,Y^n_j}\right) \Rightarrow (X,\eta_{XY}),
\end{align}
if and only if for every $k \in \N\setminus \{1,2\}$,
\begin{equation}
\label{eqn::XXXYYconv}
(X^n,Y^n_{2:k}) \Rightarrow (X,Y\pcar{2:k}),
\end{equation}
where
\begin{itemize}
\item $\law(X,Y\pcar{j}|X) = \eta_{XY}$ for all $j$,
\item and $(Y\pcar{j})_{j=2}^k$ are conditionally independent given $X$.
\end{itemize}
In addition, \eqref{eqn::XXXYYconv} implies \eqref{eqn::muXYiconv} when $k = 3$.
\end{proposition}

\begin{definition}[Conditional Propagation of Chaos]
\label{CPC}
We say a pair of random variable collections $(\mathbf{X},\mathbf{Y})$ satisfies the $\Xmc/\Ymc$-\emph{conditional propagation of chaos} property if $(\mathbf{X},\mathbf{Y})$ is $\Xmc/\Ymc$-convenient and satisfies \eqref{eqn::muXYiconv}.
\end{definition}

\begin{remark}[Extension to Continuous Time]
\label{rem::localcomp}
In Proposition \ref{prop::condconvres}, $\bm{X}$ and $\bm{Y}$ are intended to represent node and edge trajectories of a sequence of interacting particle systems on dynamic networks. However, in the continuous-time context, this will typically result in $\Xmc$ and $\Ymc$ failing to be locally compact. In this situation, the ``only if" direction of Proposition \ref{prop::condconvres} still holds. This is because local compactness is only needed to due to our application of Lemma \ref{convinprobsuff} which is not used to in the proof of the ``only if" direction of the proposition. We expect that the ``if" direction also holds in this case. To establish this direction of the proposition in the non-locally compact case, it will be necessary to place additional assumptions on $\{\eta^n_{XY}\}_{n \in \N}$ such that 
\[\langle \alt{\eta}^n_{XY},f\rangle \to \langle \alt{\eta}_{XY},f\rangle \te{ in probability for all }f \in C_b(\Xmc\times \Ymc)\te{ implies } \alt{\eta}^n_{XY} \to \alt{\eta}_{XY} \te{ in probability.}\]
For continuous time interacting particle systems, $\Xmc$ and $\Ymc$ are typically spaces of continuous or c\`adl\`ag functions on discrete or Euclidean state spaces. We expect the above convergence to hold for such function spaces under the additional condition that the sequence $\{\alt{\eta}^n_{XY}\}_{n\in \N}$ is a.s. tight (e.g. via an application of \cite[Theorems 7.1, 13.1]{Bil99}).
\end{remark}

\begin{remark}[Why Do We Exclude $Y^n_1$ and $Y\pcar{1}$?]
\label{why23?}
In Proposition \ref{prop::condconvres}, we work with the convergence $(X^n,Y^n_{2:k}) \Rightarrow (X,Y\pcar{2:k})$. Why is this? Because $(X^n,Y^n)$ is exchangeable excluding 1, this is equivalent to stating that $(X^n,Y^n_{j_{1:m}}) \Rightarrow (X,Y\pcar{j_{1:m}})$ for any distinct $j_1,\dots,j_m$ in $[2:n]$. We simply state everything for $j_i = i+1$ and $m=k-1$. This is because \eqref{eqn::muXYiconv} does not give us any information regarding the limit of the expression $(X^n,Y^n_1)$ due to the fact that $(X^n,Y^n)$ is exchangeable \emph{excluding} 1.
\end{remark}

\begin{remark}[Conditional Propagation of Chaos Under Exchangeability Excluding $i$]
\label{rmk::why1?}
Suppose that $\left(X,(Y_j)_{j=1}^n\right)$ is exchangeable excluding $i$ for some $i \in [1:n]$. Suppose also that $Z_1 = Y_i$, $Z_{j+1} = Y_{j}$ for $j < i$ and $Z_j = Y_j$ for $j > i$. Then it is easily seen that $\left(X,(Z_j)_{j=1}^n\right)$ is exchangeable excluding 1. This is why in Definition \ref{def::usefulassu}, we simply allow $(X^n,Y^n)$ to be exchangeable excluding 1.
\end{remark}
\subsection{Proof of the Conditional Propagation of Chaos Property}
\label{ssec::cppf}

\begin{proof}[Proof of Proposition \ref{prop::condconvres}]
First, suppose that \eqref{eqn::muXYiconv} holds. Then we need to show that \eqref{eqn::XXXYYconv} holds for all $k\geq 3$. Let $f \in C_b(\Xmc\times \Ymc^{k-1})$ be any bounded, continuous function. Additionally, define the function $\alt{f} \in C_b((\Xmc\times\Ymc)^{k-1})$ by
\[\alt{f}\left((x_{j-1},y_j)_{j=2}^k\right) = f(x_1,y_2,\dots,y_k).\]
By assumption, $\eta^n_{XY}\Rightarrow \eta_{XY}$, so $\left(\eta^n_{XY}\right)^{k-1} \Rightarrow \left(\eta_{XY}\right)^{k-1}$. Furthermore, by Lemma \ref{lem::CIgivenprod},
\[\left\langle \left(\eta_{XY}\right)^{k-1},\alt{f}\right\rangle = \ex{\alt{f}\left((X,Y\pcar{j})_{j=2}^k\right)\middle|X} = \ex{f(X,Y\pcar{2:k})\middle|X},\]
where $\law(X,Y\pcar{j}|X) = \eta_{XY}$ for all $j$ and $(Y\pcar{j})_{j=2}^k$ are conditionally independent given $X$. By exchangeability (excluding 1), $f(X^n,Y^n_2,Y^n_3) \deq f(X^n,Y^n_i,Y^n_j)$ for any $i,j \in [1:n]$ such that $i\neq j$ and $i,j\neq 1$. Applying this and the fact that $f$ and $\alt{f}$ are bounded and continuous and \eqref{eqn::muXYiconv},
\begin{align*}
\ex{f(X^n,Y^n_{2:k})} &= \frac{(n-k)!}{(n-1)!}\sum_{\substack{j_{1:k-1}\in [2:n]^{k-1}\\ \te{distinct}}}\ex{f(X^n,Y^n_{j_{1:k-1}})}\\
&= \frac{(n-k)!}{(n-1)!}\sum_{\substack{j_{1:k-1}\in [2:n]^{k-1}\\ \te{distinct}}}\ex{\alt{f}\left((X^n,Y^n_{j_{\ell}})_{\ell=1}^{k-1}\right)}\\
&= \ex{\frac{1}{n^{k-1}}\sum_{j_{1:k-1}\in [1:n]^{k-1}} \alt{f}\left((X^n,Y^n_{j_{\ell}})_{\ell=1}^{k-1}\right)} + R_n\\
&= \ex{\langle \left(\eta^n_{XY}\right)^{k-1}, \alt{f}\rangle} + R_n\\
&\to \ex{\langle \left(\eta_{XY}\right)^{k-1}, \alt{f}\rangle}\\
&= \ex{\ex{f(X,Y\pcar{2:k}\middle|X}}\\
&= \ex{f(X,Y\pcar{2:k})},
\end{align*}
where we note that the convergence holds because $|R_n|$ is $o(1)$ as we now show. Fix any $\ep > 0$ and suppose that $n > \frac{k}{\ep}$. Then,
\[n - k > (1-\ep)n \Rightarrow n-k + 1 > (1 - \ep)n \Rightarrow \frac{(n-1)!}{(n-k)!} > (1 - \ep)^{k-1}n^{k-1} > (1 - (k-1)\ep)n^{k-1} > (1 - k\ep)n^{k-1}.\]
Note that there are $\frac{(n-1)!}{(n-k)!}$ choices of $j_{1:k-1}$ in $[2:n]^{k-1}$ that are distinct. There are $n^{k-1}$ elements in the set $[1:n]^{k-1}$. Therefore there are $n^{k-1}-\frac{(n-1)!}{(n-k)!}$ choices of $j_{1:k-1}$ such that the $j$'s are not distinct or such that $\min\{j_i\} = 1$. Then,
\begin{align}
|R_n| &= \left|\left(\frac{(n-k)!}{(n-1)!}-\frac{1}{n^{k-1}}\right)\sum_{\substack{j_{1:k-1}\in [2:n]^{k-1}\\\te{distinct}}} \ex{\alt{f}\left((X^n,Y^n_{j_{\ell}})_{\ell=1}^{k-1}\right)} - \frac{1}{n^{k-1}}\sum_{\substack{j_{1:k-1}\in [1:n]^{k-1}\\\te{not distinct}\\\te{or } \min\{j_i\} = 1}} \ex{\alt{f}\left((X^n,Y^n_{j_{\ell}})_{\ell=1}^{k-1}\right)}\right|\nonumber\\
&\leq \left|1 - \frac{(n-1)!}{n^{k-1}(n-k)!}\right|\|\alt{f}\|_{\infty} + \left|\frac{n^{k-1} - \frac{(n-1)!}{(n-k)!}}{n^{k-1}}\right|\|\alt{f}\|_{\infty}\label{eqn::Rnbd}\\
&< 2k\ep \|\alt{f}\|_{\infty}\label{eqn::Rnbdo1}
\end{align}

Since $k$ is fixed and (as $n \to \infty$) $\ep$ can be arbitrarily small, it follows that $|R_n|$ is $o(1)$ as previously claimed. This completes the proof that \eqref{eqn::muXYiconv} implies that \eqref{eqn::XXXYYconv} holds for all $k \geq 3$.

\ind Now we instead suppose that \eqref{eqn::XXXYYconv} holds for $k = 3$. Because $X^n \Rightarrow X$, it is possible to construct the following coupling by applying the Skorokhod representation theorem. Let $(\bar{\Omega},\bar{\Fmc},\bar{\PP})$ be a complete probability space containing the random elements $(\bar{X}^n,\bar{Y}^n)\deq (X^n,Y^n)$ for all $n \in\N$  and $(\bar{X},\bar{Y}\pcar{2},\bar{Y}\pcar{3})\deq (X,Y\pcar{2},Y\pcar{3})$ such that $\bar{X}^n \to \bar{X}$ in probability. 
  Then, it suffices to show 
\[\bar{\eta}^n_{XY} \defeq \frac{1}{n}\sum_{j=1}^n \delta_{\bar{X}^n,\bar{Y}^n_j} \to \bar{\eta}_{XY} \defeq \law(\bar{X},\bar{Y}\pcar{2}|\bar{X})\te{ in probability.}\]
By Lemma \ref{convinprobsuff}, we can do this by showing that 
\[\langle \bar{\eta}^n_{XY}, f \rangle \to \langle \bar{\eta}_{XY},f\rangle\te{ in probability}\]
for all $f \in C_b(\Xmc\times \Ymc)$. In fact, we prove the following stronger $L^2$ convergence result:
\begin{equation}
\label{eqn::mucondconvsuff}
\lim_{n\to\infty}\ex{\left(\langle\bar{\eta}^n_{XY},f\rangle - \langle \bar{\eta}_{XY},f\rangle\right)^2} = 0.
\end{equation}
First we show convergence of the second moment of $\langle \bar{\eta}^n_{XY},f\rangle$ utilizing the exchangeability (excluding 1) of $(X^n,Y^n)$ and the bounded convergence theorem:
\begin{align*}
\lim_{n\to\infty}\ex{\langle \bar{\eta}^n_{XY},f\rangle^2} &= \lim_{n\to\infty}\frac{1}{n^2}\sum_{i,j=1}^n\ex{f(\bar{X}^n,\bar{Y}^n_i)f(\bar{X}^n,\bar{Y}^n_j)}\\
&= \lim_{n\to\infty}\frac{1}{(n-1)(n-2)}\sum_{\substack{i,j=2\\i\neq j}}^n\ex{f(\bar{X}^n,\bar{Y}^n_i)f(\bar{X}^n,\bar{Y}^n_j)} + R'_n\\
&= \lim_{n\to\infty}\ex{f(\bar{X}^n,\bar{Y}^n_2)f(\bar{X}^n,\bar{Y}^n_3)} + R'_n\\
&= \ex{f(\bar{X},\bar{Y}\pcar{2})f(\bar{X}^n,\bar{Y}\pcar{3})},
\end{align*}
where the last equality holds because $R'_n = o(1)$, where the proof that $R'_n$ is $o(1)$ is nearly identical to the proof that $R_n$ is $o(1)$ in the special case that $k=3$ (see \eqref{eqn::Rnbdo1}):
\begin{align*}
|R'_n| &= \left|\left(\frac{1}{n^2} - \frac{(n-3)!}{(n-1)!}\right)\sum_{\substack{i,j=2\\ i \neq j}}^n \ex{f(\bar{X}^n,\bar{Y}^n_i)f(\bar{X}^n,\bar{Y}^n_j)} + \frac{1}{n^2}\sum_{\substack{i,j=1\\ i=j\te{ or } i\wedge j = 1}}\ex{f(\bar{X}^n,\bar{Y}^n_i)f(\bar{X}^n,\bar{Y}^n_j)}\right|\\
&\leq \left|1 - \frac{(n-1)!}{n^2(n-3)!}\right|\|f\|^2_{\infty} + \left|\frac{n^2 - \frac{(n-1)!}{(n-3)!}}{n^2}\right|\|f\|^2_{\infty}\\
&= o(1).
\end{align*}

Checking the second moment of $\langle \bar{\eta}_{XY},f\rangle$ using the fact that $\bar{Y}\pcar{2}$ and $\bar{Y}\pcar{3}$ are conditionally i.i.d. given $\bar{X}$:
\begin{align*}
\ex{\langle\bar{\eta}_{XY},f\rangle^2} &= \ex{\ex{f(\bar{X},\bar{Y}\pcar{2})|\bar{X}}^2}\\
&= \ex{\ex{f(\bar{X},\bar{Y}\pcar{2})|\bar{X}}\ex{f(\bar{X},\bar{Y}\pcar{3})|\bar{X}}}\\
&= \ex{f(\bar{X},\bar{Y}\pcar{2})f(\bar{X},\bar{Y}\pcar{3})}\\
&=\lim_{n\to\infty} \ex{\langle \bar{\eta}^n_{XY},f\rangle^2}.
\end{align*}

Lastly we investigate the correlation term $\ex{\langle \bar{\eta}^n_{XY},f\rangle\langle \bar{\eta}_{XY},f\rangle}$. For this, note that we can assume without loss of generality that $\left((\bar{X},\bar{X}^n),\bar{Y}^n\right)$ is exchangeable excluding 1 for all $n \in \N$\footnote{To be completely rigorous, we can achieve this with a new coupling. Let $\{\sigma_n\}_{n \in \N}$ be a collection of mutually independent random permutations independent of $(\bar{X},\bar{Y}\pcar{2},\bar{Y}\pcar{3},\bar{X}^n,\bar{Y}^n)_{n \in \N}$. Assume that for each $n$, $\sigma_n$ is uniformly sampled from the subset $\{\psi \in S_n: \psi(1) = 1\}$. Then note that for any $\psi \in S_n$ such that $\psi(1) = 1$, $\psi\circ\sigma \deq \sigma$ so:
\[\left(\bar{X},\bar{X}^n,\alt{Y}^n\right) \defeq \left(\bar{X},\bar{X}^n,\left(\bar{Y}^n_{\sigma(j)}\right)_{j=1}^n\right) \deq \left(\bar{X},\bar{X}^n,\left(\bar{Y}^n_{\psi\circ\sigma(j)}\right)_{j=1}^n\right) = \left(\bar{X},\bar{X}^n,\left(\alt{Y}^n_{\psi(j)}\right)_{j=1}^n\right).\]
So we can get exchangeability excluding 1 by replacing $\bar{Y}^n$ with $\alt{Y}^n$. Furthermore, by exchangeability excluding 1, $(\bar{X}^n,\alt{Y}^n) \deq (\bar{X}^n,\bar{Y}^n) \deq (X^n,Y^n)$ so this is a valid coupling.}. In addition, $\bar{\eta}_{XY}$ depends continuously on $\bar{X}$ by assumption, which implies that there exists a bounded, a.s. continuous $\phi:\Xmc\to \R$ such that $\phi(\bar{X}) = \langle \bar{\eta}_{XY},f\rangle$ a.s.. Using the fact that $g: \Xmc\times \Xmc\times \Ymc \mapsto \R$ defined by $g(x,x^n,y^n) = \phi(x)f(x^n,y^n)$ is bounded and a.s. continuous and $(\bar{X},\bar{X}^n,\bar{Y}^n_2) \Rightarrow (\bar{X},\bar{X},\bar{Y}\pcar{2})$,
\begin{align*}
\lim_{n\to\infty}\ex{\langle \bar{\eta}^n_{XY},f\rangle\langle \bar{\eta}_{XY},f\rangle} &= \lim_{n\to\infty}\ex{\phi(\bar{X})\left(\frac{1}{n}\sum_{j=1}^n f(\bar{X}^n,\bar{Y}^n_j)\right)}\\
&= \lim_{n\to\infty}\ex{\phi(\bar{X})\left(\frac{1}{n-1}\sum_{j=2}^n f(\bar{X}^n,\bar{Y}^n_j)\right)} + R''_n\\
&= \lim_{n\to\infty}\ex{\phi(\bar{X})f(\bar{X}^n,\bar{Y}^n_2)} + R''_n\\
&= \ex{\phi(\bar{X})f(\bar{X},\bar{Y}\pcar{2})}\\
&= \ex{\ex{f(\bar{X},\bar{Y}\pcar{2})|\bar{X}}f(\bar{X},\bar{Y}\pcar{2})}\\
&= \ex{\ex{f(\bar{X},\bar{Y}\pcar{2}|\bar{X}}^2}\\
&= \ex{\langle \bar{\eta}_{XY},f\rangle^2}.
\end{align*}

Once more, the convergence above holds because $R''_n = o(1)$ as shown below:
\begin{align*}
|R''_n| &= \left|\left(\frac{1}{n}-\frac{1}{n-1}\right)\sum_{j=2}^n \ex{\phi(\bar{X})f(\bar{X}^n,\bar{Y}^n_j)} + \frac{1}{n}\ex{\phi(\bar{X})f(\bar{X}^n,\bar{Y}^n_1)}\right|\\
&\leq \left|-\frac{1}{n(n-1)}*(n-1)\right|\|\phi\|_{\infty}\|f\|_{\infty} + \frac{1}{n}\|\phi\|_{\infty}\|f\|_{\infty}\\
&= \frac{2}{n}\|\phi\|_{\infty}\|f\|_{\infty}\\
&= o(1).
\end{align*}

Combining the three computations above yields
\begin{align*}
\lim_{n\to\infty}\ex{\left(\langle\bar{\eta}^n_{XY},f\rangle - \langle \bar{\eta}_{XY},f\rangle\right)^2} &= \lim_{n\to\infty}\ex{\langle\bar{\eta}^n_{XY},f\rangle^2 - 2\langle\bar{\eta}^n_{XY},f\rangle\langle \bar{\eta}_{XY},f\rangle + \langle \bar{\eta}_{XY},f\rangle^2}\\
&=  \ex{\langle \bar{\eta}_{XY},f\rangle^2} - 2\ex{\langle \bar{\eta}_{XY},f\rangle^2} + \ex{\langle \bar{\eta}_{XY},f\rangle^2}\\
&= 0.
\end{align*}
Since we have shown \eqref{eqn::mucondconvsuff}, we are now done.
\end{proof}

\section{Proof of Proposition \ref{prop::welldef} and Corollary \ref{coro::morestruct}}
\label{sec::proppf}

In this section and Section \ref{sec::pfsmple}, it is convenient to define the following shorthand notation. Let $F: \Xmc \to [0,1]$ be any function with a range in $[0,1]$. Then we define the function $\alt{F}: [0,1]\times \Xmc \to [0,1]$ by
\begin{equation}
\label{eqn::altdef}
\alt{F}(p,x) = pF(x) + (1 - p)(1 - F(x)).
\end{equation}
We also repeatedly use the following useful result:
\begin{lemma}
\label{lem::condpBtild}
Fix any $s \geq 1$, $(a,a') \in \{0,1\}^2$ and $(i,j) \in \incs_k$. Then under Definition \ref{defn::AZlim}(a), (b) or (c),
\begin{equation}
\label{eq::condpBtild}
\PP\left(A\pcar{ij}(s) = a\middle|\Gmc^{A,k}_s,A\pcar{ij}(s-1) = a'\right) = \alt{B}\left(a,a',Z\pcar{i}(s),Z\pcar{j}(s)\right),
\end{equation}
where $\alt{B}\left(a,a',Z\pcar{i}(s),Z\pcar{j}(s)\right)$ is defined via \eqref{eqn::altdef} with $p=a$ and $x=\left(a',Z\pcar{i}(s),Z\pcar{j}(s)\right)$.
\end{lemma}
\begin{proof}
By Definition \ref{defn::AZlim}, $A\pcar{ij}$ satisfies \eqref{eqn::Aevolve}. Setting $a = 1$, this implies that
\begin{align*}
\PP\left(A\pcar{ij}(s) = 1\middle|\Gmc^{A,k}_s,A\pcar{ij} = a'\right) &= \ex{A\pcar{ij}(s)\middle|\Gmc^{A,k}_s,A\pcar{ij} = a'} = B(a',Z\pcar{i}(s),Z\pcar{j}(s)).
\end{align*}
Likewise, if $a = 0$, then 
\begin{align*}
\PP\left(A\pcar{ij}(s) = 0\middle|\Gmc^{A,k}_s,A\pcar{ij} = a'\right) &= 1-\ex{A\pcar{ij}(s)\middle|\Gmc^{A,k}_s,A\pcar{ij} = a'} = 1-B(a',Z\pcar{i}(s),Z\pcar{j}(s)).
\end{align*}
Hence,
\[\PP\left(A\pcar{ij}(s) = a\middle|\Gmc^{A,k}_s,A\pcar{ij} = a'\right) = \alt{B}\left(a,a',Z\pcar{i}(s),Z\pcar{j}(s)\right).\]
\end{proof}
\subsection{Proof of Proposition \ref{prop::welldef}}
\label{ssec::wdpf}
\sloppy It suffices to show that Definition \ref{defn::AZlim}(b) completely characterizes the distribution of $(Z\pcar{1:k}[t],A\pcar{1:k,1:k}[t])$ and that Definition \ref{defn::AZlim}(a) implies \ref{defn::AZlim}(b) implies \ref{defn::AZlim}(c) implies \ref{defn::AZlim}(a). We prove these statements in Lemmas \ref{lem::achar}-\ref{lem::ctoa} below.

\ind Before we prove this, we establish an intermediate result that will be extremely useful for the remainder of this section. This result establishes that the limiting subnetwork $A\pcar{1:k,1:k}$ is conditionally independent of future latent opinions given current and past latent opinions.

\begin{lemma}
\label{lem::futind}
If $(Z\pcar{1:k}[t],A\pcar{1:k,1:k}[t])$ satisfies the conditions laid out in Definition \ref{defn::AZlim}(a),(b) or (c), then for any $s \in [t-1]$, $A\pcar{1:k,1:k}[s]$ is conditionally independent of $Z\pcar{1:k}[s+1:t]$ given $Z\pcar{1:k}[s]$.
\end{lemma}
\begin{proof}
By \eqref{eqn::Lidef}, $L\pcar{i}(s)$ is $\Gmc^k_s$-measurable for any $i \in [1:k]$. It follows by \eqref{eqn::Zevolve} that $Z\pcar{i}(s+1) - \xi_i(s)$ is $\Gmc^k_s$-measurable for each $i \in [1:k]$. It immediately follows that $\Gmc^k_{s+1} = \Gmc^k_s\vee \sigma(\xi_{1:k}(s))$. Iterating this argument, we get
\[\Gmc^k_t = \Gmc^k_s\vee\sigma\left(\xi_{1:k}[s:t-1]\right).\]

\ind Recall from \eqref{eqn::Zevolve}-\eqref{eqn::mu2def} that $\xi_{1:k}[s:t-1]$ is independent of $\Fmc^k_s = \sigma(A\pcar{1:k,1:k}[s])\vee\sigma(Z\pcar{1:k}[s])$. The result then follows from \cite[Proposition 2.5(b)]{PutSch85}, where $F_1 = \sigma(\xi_{1:k}[s:t-1])$, $F_2 = \sigma(A\pcar{1:k,1:k}[s])$ and $G = \Gmc^k_s$.
\end{proof}

We now show that Definition \ref{defn::AZlim}(b) completely characterizes the desired distribution.

\begin{lemma}
\label{lem::achar}
The distribution of $(Z\pcar{1:k}[t],A\pcar{1:k,1:k})$ is completely characterized by Definition \ref{defn::AZlim}(b).
\end{lemma}
\begin{proof}
The first part of Definition \ref{defn::AZlim} completely characterizes the marginal distribution of $Z\pcar{1:k}[t]$. Moreover, $A\pcar{ij}[t] = A\pcar{ji}[t]$ for all $i,j \in [1:k]$ and $A\pcar{ii}[t] = 1$ for all $i \in [1:k]$. Therefore, it suffices to prove that the conditional distribution of $(A\pcar{ij}[t])_{(i,j) \in \incs_k}$ given $\Gmc^k_t = \sigma(Z\pcar{1:k}[t])$ is well-defined, where we recall that $\incs_k = \{(i,j)\in[k]^2: 1\leq i <j\leq k\}$.

\ind Fix any $\bm{a}\defeq (a_{ij}[t])_{(i,j)\in\incs_k} \in (\mf_t(\{0,1\}))^{\incs_k}$. Then by Definition \ref{defn::AZlim}(b), Lemma \ref{lem::futind}, \cite[Proposition 3.2(a)]{PutSch85} (setting $F_1 = \Gmc^k_t$, $F_2 = \sigma(A\pcar{ij}(s-1))$, $F_3 = \sigma(A\pcar{1:k,1:k}[s])$ and $G = \Gmc^k_s$), \eqref{eqn::filtrations}, and \eqref{eqn::Aevolve}:
\begin{align*}
\PP&\left((A\pcar{ij}[t])_{(i,j)\in \incs_k} = (a_{ij}[t])_{(i,j)\in\incs_k}\middle|\Gmc^k_t\right)=\\
&= \prod_{(i,j)\in \incs_k}\PP\left(A\pcar{ij}[t] = a_{ij}[t]\middle|\Gmc^k_t\right)\\
&=\prod_{(i,j)\in \incs_k}\PP\left(A\pcar{ij}(0) = a_{ij}(0)\middle|\Gmc^k_t\right)\prod_{s=1}^t\PP\left(A\pcar{ij}(s) = a_{ij}(s)\middle|\Gmc^k_t,A\pcar{ij}(s-1)=a_{ij}(s-1)\right)\\
&=\prod_{(i,j)\in \incs_k}\PP\left(A\pcar{ij}(0) = a_{ij}(0)\middle|\Gmc^{k}_0\right)\prod_{s=1}^t\PP\left(A\pcar{ij}(s) = a_{ij}(s)\middle|\Gmc^k_s,A\pcar{ij}(s-1)=a_{ij}(s-1)\right)\\
&=\prod_{(i,j)\in \incs_k}\PP\left(A\pcar{ij}(0) = a_{ij}(0)\middle|\Gmc^{A,k}_0\right)\\
&\ind\prod_{s=1}^t\ex{\PP\left(A\pcar{ij}(s) = a_{ij}(s)\middle|\Gmc^{A,k}_s,A\pcar{ij}(s-1)=a_{ij}(s-1)\right)\middle|\Gmc^k_s,A\pcar{ij}(s-1)=a_{ij}(s-1)}\\
&=\prod_{(i,j)\in \incs_k}\alt{B}_0\left(a_{ij}(0),Z\pcar{i}(0),Z\pcar{j}(0)\right)\\
&\ind\prod_{s=1}^t\ex{\alt{B}\left(a_{ij}(s),a_{ij}(s-1),Z\pcar{i}(s),Z\pcar{j}(s)\right)\middle|\Gmc^k_s,A\pcar{ij}(s-1)=a_{ij}(s-1)}\\
&=\prod_{(i,j)\in \incs_k}\alt{B}_0\left(a_{ij}(0),Z\pcar{i}(0),Z\pcar{j}(0)\right)\prod_{s=1}^t\alt{B}\left(a_{ij}(s),a_{ij}(s-1),Z\pcar{i}(s),Z\pcar{j}(s)\right),
\end{align*}
where the function $\alt{B}$ is defined via $B$ based on \eqref{eqn::altdef}, see Lemma \ref{lem::condpBtild}. This concludes the proof. \end{proof}

We now prove the equivalences of definitions.

\begin{lemma}
\label{lem::atob}
Definition \ref{defn::AZlim}(a) implies Definition \ref{defn::AZlim}(b).
\end{lemma}
\begin{proof}
Fix any $\bm{a}\defeq (a_{ij}[t])_{(i,j)\in\incs_k} \in (\mf_t(\{0,1\})^{\incs_k}$. Then applying \eqref{eqn::filtrations}, Definition \ref{defn::AZlim}(a) and \eqref{eqn::Aevolve},
\begin{align*}
\PP&\left((A\pcar{ij}[t])_{(i,j)\in \incs_k} = (a_{ij}[t])_{(i,j)\in \incs_k}\middle|\Gmc^k_t\right)\\
&= \PP\left((A\pcar{ij}(0))_{(i,j)\in \incs_k} = (a_{ij}(0))_{(i,j)\in \incs_k}\middle|\Gmc^k_t\right)\\
&\ind \prod_{s=1}^t\PP\left((A\pcar{ij}(s))_{(i,j)\in \incs_k} = (a_{ij}(s))_{(i,j)\in \incs_k}\middle|\Gmc^k_t,(A\pcar{ij}[s-1])_{(i,j)\in\incs_k} = (a\pcar{ij}[s-1])_{(i,j)\in\incs_k}\right)\\
&= \PP\left((A\pcar{ij}(0))_{(i,j)\in \incs_k} = (a_{ij}(0))_{(i,j)\in \incs_k}\middle|\Gmc^k_0\right)\\
&\ind \prod_{s=1}^t\PP\left((A\pcar{ij}(s))_{(i,j)\in \incs_k} = (a_{ij}(s))_{(i,j)\in \incs_k}\middle|\Gmc^k_s,(A\pcar{ij}[s-1])_{(i,j)\in\incs_k} = (a\pcar{ij}[s-1])_{(i,j)\in\incs_k}\right)\\
&= \PP\left((A\pcar{ij}(0))_{(i,j)\in \incs_k} = (a_{ij}(0))_{(i,j)\in \incs_k}\middle|\Gmc^{A,k}_0\right)\\
&\ind \prod_{s=1}^t\PP\left((A\pcar{ij}(s))_{(i,j)\in \incs_k} = (a_{ij}(s))_{(i,j)\in \incs_k}\middle|\Gmc^{A,k}_s,(A\pcar{ij}[s-1])_{(i,j)\in\incs_k} = (a\pcar{ij}[s-1])_{(i,j)\in\incs_k}\right)\\
&=\prod_{(i,j) \in \incs_k}\PP\left(A\pcar{ij}(0) = a_{ij}(0)\middle|\Gmc^{A,k}_0\right) \times\nonumber\\
&\ind\times \prod_{s=1}^t\PP\left(A\pcar{ij}(s) = a_{ij}(s)\middle|\Gmc^{A,k}_s,(A\pcar{ij}[s-1])_{(i,j)\in\incs_k} = (a\pcar{ij}[s-1])_{(i,j)\in\incs_k}\right)\\
&=\prod_{(i,j) \in \incs_k}\alt{B_0}\left(a_{ij}(0),Z\pcar{i}(0),Z\pcar{j}(0)\right) \prod_{s=1}^t\alt{B}\left(a_{ij}(s),a_{ij}(s-1),Z\pcar{i}(s),Z\pcar{j}(s)\right),
\end{align*}
by Lemma \ref{lem::condpBtild}. The latter relation proves both the conditional independence (with respect to $\Gmc^k_t$) of $\{A\pcar{ij}[t]\}_{(i,j)\in \incs_k}$ and the conditional Markovian structure of $A\pcar{ij}[t]$ given $\Gmc^k_t$ for each $(i,j) \in \incs_k$. Therefore, Definition \ref{defn::AZlim}(a) indeed implies that  Definition \ref{defn::AZlim}(b) holds.
\end{proof}

\begin{lemma}
\label{lem::btoc}
Definition \ref{defn::AZlim}(b) implies Definition \ref{defn::AZlim}(c).
\end{lemma}
\begin{proof}
Fix any $s \in [t]$. If $s = t$, then we are done, so assume $s < t$. Fix any $\bm{a}\defeq (a_{ij}[s])_{(i,j)\in\incs_k} \in \left(\mf_s(\{0,1\})\right)^{\incs_k}$. Then applying Lemma \ref{lem::futind} and Definition \ref{defn::AZlim}(b),
\begin{align*}
\PP\left((A\pcar{ij}[s])_{(i,j)\in\incs_k} = \bm{a}\middle|\Gmc^k_s\right) &= \PP\left((A\pcar{ij}[s])_{(i,j)\in\incs_k} = \bm{a}\middle|\Gmc^k_t\right)\\
&=\prod_{(i,j)\in\incs_k} \PP\left(A\pcar{ij}[s] = \bm{a}_{ij}[s]\middle|\Gmc^k_t\right)\\
&=\prod_{(i,j)\in\incs_k} \PP\left(A\pcar{ij}[s] = \bm{a}_{ij}[s]\middle|\Gmc^k_s\right),
\end{align*}
which proves that $A\pcar{ij}[s]$ is conditionally independent given $\Gmc^k_s$.
\end{proof}

\begin{lemma}
\label{lem::ctoa}
Definition \ref{defn::AZlim}(c) implies Definition \ref{defn::AZlim}(a).
\end{lemma}
\begin{proof}
This is a simple consequence of \cite[Proposition 3.2(a)]{PutSch85} and Lemma \ref{CIsuff}. Fix any $s\in [t]$ and $(i,j) \in \incs_k$. By Definition \ref{defn::AZlim}(c) and Lemma \ref{CIsuff},
\[A\pcar{ij}[s] \indp (A\pcar{i'j'}[s])_{(i',j')\neq (i,j)}\Big|\Gmc^k_s.\]
Define $F_1 = \sigma(A\pcar{ij}[s])$, $F_2 = \sigma\left((A\pcar{i'j'}[s-1])_{(i',j')\neq (i,j)}\right)$, $F_3 = \sigma\left((A\pcar{i'j'}(s))_{(i',j')\neq (i,j)}\right)$ and $G = \Gmc^k_s$. Then by \cite[Proposition 3.2(a)]{PutSch85},
\[A\pcar{ij}[s] \indp (A\pcar{i'j'}(s))_{(i',j')\neq (i,j)}\Big|\Gmc^k_s\vee\sigma\left((A\pcar{i'j'}[s-1])_{(i',j')\neq (i,j)}\right).\]
Now set $F_1 = \sigma\left((A\pcar{i'j'}(s))_{(i',j')\neq (i,j)}\right)$, $F_2 = \sigma(A\pcar{ij}[s-1])$, $F_3 = \sigma(A\pcar{ij}(s))$ and $G = \Gmc^k_s\vee\sigma\left((A\pcar{i'j'}[s-1])_{(i',j')\neq (i,j)}\right)$. Then by \cite[Proposition 3.2(a)]{PutSch85},
\[A\pcar{ij}(s) \indp (A\pcar{i'j'}(s))_{(i',j')\neq (i,j)}\Big|\Gmc^k_s\vee\sigma\left((A\pcar{i'j'}[s-1])_{(i',j')\neq (i,j)}\right)\vee\sigma\left(A\pcar{ij}[s-1]\right) = \Gmc^{A,k}_s.\]
By Lemma \ref{CIsuff}, this implies that the random variables in $\{A\pcar{ij}(s)\}_{(i,j)\in \incs_k}$ are mutually conditionally independent given $\Gmc^{A,k}_s$ as desired.
\end{proof}

\subsection{Proof of Corollary \ref{coro::morestruct}}
\label{ssec::coromspf}

By Definition \ref{defn::AZlim}, $Z\pcar{i}(0)$, $i \in [1:k]$ are i.i.d.. By \eqref{eqn::Zevolve}-\eqref{eqn::mu2def}, there exists a function $F: \R^d\times \mf_{t-1}(\R^d) \to \mf_t(\R^d)$ such that
\[F\left(Z\pcar{i}(0),\xi_i[t-1]\right) = Z\pcar{i}[t]\te{ for }i \in [1:k].\]
By the statement after \eqref{eqn::Zevolve}-\eqref{eqn::mu2def}, $\{\xi_i(s)\}_{i \in [1:k],s\in [t]}$ are i.i.d. and independent of $\Fmc^k_0 = \sigma(Z\pcar{1:k}(0))$. This implies that
\[(Z\pcar{i}[t])_{i\in [1:k]} = \left(F\left(Z\pcar{i}(0),\xi_i[t-1]\right)\right)_{i \in [1:k]}\]
is an i.i.d. collection of random elements as desired.

\ind Next, we establish \eqref{eqn::exAH}. Fix $(i,j) \in \incs_k$ and $s \in [t]$. For $s = 0$, $\Hmc^k_0 = \Gmc^{A,k}_0$, so \eqref{eqn::exAH} holds. Then,
\begin{align*}
\ex{A\pcar{ij}(s)\middle|\Hmc^k_s} &= \ex{\ex{A\pcar{ij}(s)\middle|\Gmc^{A,k}_s}\middle|\Hmc^k_s}\\
&= \ex{B\left(A\pcar{ij}(s-1),Z\pcar{i}(s),Z\pcar{j}(s)\right)\middle|\Hmc^k_s}\\
&= B\left(A\pcar{ij}(s-1),Z\pcar{i}(s),Z\pcar{j}(s)\right)\\
&= \ex{A\pcar{ij}(s)\middle|\Gmc^{A,k}_s}.
\end{align*}
Likewise, if $s = 0$, then \eqref{eqn::exAG} holds by \eqref{eqn::Aevolve}. Now, fix $s \in [1:t]$ and assume that \eqref{eqn::exAG} holds for $s - 1$. Then applying \eqref{eqn::Aevolve}, \eqref{eqn::Bsdef}, and \eqref{eqn::Bhatdef},
\begin{align*}
\ex{A\pcar{ij}(s)\middle|\Gmc^k_s} &= \ex{\ex{A\pcar{ij}(s)\middle|\Gmc^k_s,A\pcar{ij}(s-1)}\middle|\Gmc^k_s}\\
&= \ex{B\left(A\pcar{ij}(s-1),Z\pcar{i}(s),Z\pcar{j}(s)\right)\middle|\Gmc^k_s}\\
&= \sum_{a = 0}^1 \ex{B\left(a,Z\pcar{i}(s),Z\pcar{j}(s)\right)\middle|\Gmc^k_s,A\pcar{ij}(s-1)=a}\PP\left(A\pcar{ij}(s-1) = a\middle|\Gmc^k_s\right)\\
&= \sum_{a = 0}^1 B\left(a,Z\pcar{i}(s),Z\pcar{j}(s)\right)\PP\left(A\pcar{ij}(s-1) = a\middle|\Gmc^k_{s-1}\right)\\
&= \wh{B}\left(B_{s-1}(Z\pcar{i}[s-1],Z\pcar{j}[s-1]),Z\pcar{i}(s),Z\pcar{j}(s)\right)\\
&= B_s\left(Z\pcar{i}[s],Z\pcar{j}[s]\right).
\end{align*}
All that remains now is to prove the conditional independence of $A\pcar{ij}(s)$, $(i,j) \in \incs_k$ given $\Hmc^k_s$. Fix any $s\in [t]$ and $\bm{a} \defeq (a_{ij})_{(i,j)\in \incs_k}\in \{0,1\}^{\incs_k}$. Applying Lemma \ref{lem::condpBtild},
\begin{align*}
\PP\left(A\pcar{ij}(s)=a_{ij}\middle|\Hmc^k_s\right) &= \ex{\PP\left(A\pcar{ij}(s)=a_{ij}\middle|\Gmc^{A,k}_s\right)\middle|\Hmc^k_s} \\
&= \ex{\alt{B}\left(a_{ij},A\pcar{ij}(s-1),Z\pcar{i}(s),Z\pcar{j}(s)\right)\middle|\Hmc^k_s} \\
&= \alt{B}\left(a_{ij},A\pcar{ij}(s-1),Z\pcar{i}(s),Z\pcar{j}(s)\right).
\end{align*}
Applying the above display, \eqref{eqn::exAH}, and Definition \ref{defn::AZlim}(b), 
\begin{align*}
&\PP\left((A\pcar{ij}(s))_{(i,j)\in \incs_k} = \bm{a}\middle|\Hmc^k_s\right) \\
&= \ex{\PP\left((A\pcar{ij}(s))_{(i,j)\in \incs_k} = \bm{a}\middle|\Gmc^{A,k}_{s}\right)\middle|\Hmc^k_s}\\
&= \ex{\prod_{(i,j)\in \incs_k} \PP\left(A\pcar{ij}(s) = a_{ij}\middle|\Gmc^{A,k}_{s}\right)\middle|\Hmc^k_s}\\
&= \ex{\prod_{(i,j)\in \incs_k} \alt{B}\left(A\pcar{ij}(s-1),Z\pcar{i}(s),Z\pcar{j}(s)\right)\middle|\Hmc^k_s}\\
&= \prod_{(i,j)\in \incs_k} \alt{B}\left(A\pcar{ij}(s-1),Z\pcar{i}(s),Z\pcar{j}(s)\right)\\
&=\prod_{(i,j)\in \incs_k} \PP\left(A\pcar{ij}(s)=a_{ij}\middle|\Hmc^k_s\right).
\end{align*}
This completes the proof.

\section{Proof of Theorem \ref{thm::Sampleconv}}
\label{sec::pfsmple}

\subsection{A Related Result}
\label{ssec::general}

We begin with a slightly more general result, which highlights the inductive claims used in our proof of Theorem \ref{thm::Sampleconv}. To this end, we use a notion of joint exchangeability, which slightly modifies the concept of an exchangeable collection of random elements.

\begin{definition}[Joint Exchangeability]
\label{def::jointexchangeability}
For any $n \in\N$, let $X = (X_1,\dots,X_n)$ be an $\Xmc^n$-random element and let $M = (M_{ij})_{i,j=1}^n$ be an $n\times n$ random matrix with entries in $\Ymc$. Then we say that the pair $(X,M)\defeq (X_{1:n},M_{1:n,1:n})$ is \emph{jointly exchangeable} if and only if for any permutation $\sigma \in S_n$,
\[(X_{1:n},M_{1:n,1:n}) \deq (X_{\sigma(1:n)},M_{\sigma(1:n)\sigma(1:n)}).\]
\end{definition}

Note that joint exchangeability is closely related to the notion of exchangeability excluding $i$:

\begin{remark}
\label{rmk::jointtoexcl}
If a $\Xmc^n\times \Ymc^{n\times n}$-random collection $(X,M)$ is jointly exchangeable, then for all $i \in [1:n]$, the collections $\left(X_i,(M_{ij})_{j=1}^n\right)$ and $\left(X_i,(X_j,M_{ij})_{j=1}^n\right)$ are exchangeable excluding $i$. The proof is simple: for any $\sigma \in S_n$ such that $\sigma(i) = i$, 
\[\left(X_i,(M_{ij})_{j=1}^n\right) \deq \left(X_{\sigma(i)},(M_{\sigma(i)\sigma(j)})_{j=1}^n\right) = \left(X_{i},(M_{i\sigma(j)})_{j=1}^n\right).\]
The proof for $(X_i,(X_j,M_{ij}))$ is essentially the same.
\end{remark}

To this end, we define the following collection of conditions on both the prelimit and the limiting systems, which are assumed to hold at a certain time $t$. We call these conditions \emph{property A at time $t$}. We later show that property A holding at all times implies the conclusions of Theorem \ref{thm::Sampleconv}, which allows our proof of the theorem to reduce to an inductive proof that property A holds at all times.

\begin{definition}
\label{def::property}(Property A at time $t$)
For a given $t \in \N_0$, we say the models given in \eqref{eqn::Znevolve}-\eqref{eqn::Anevolve} and Definition \ref{defn::AZlim} satisfies property A at time $t$ if the following conditions hold for all $n \in \N$:
\begin{enumerate}[label = (\alph*)]
\item\label{cond::jexch} $(Z^n[t],A^n[t])$ is jointly exchangeable.
\item\label{cond::ui} The collection $\{Z^n_{i}(t)\}_{n \in \N, i \in [1:n]}$ is uniformly integrable.
\item\label{cond::jconv} For any $k \in \N$, the following convergence holds:
\[(Z^n_{1:k}[t],A^n_{1:k,1:k}[t]) \Rightarrow (Z\pcar{1:k}[t],A\pcar{1:k,1:k}[t]).\]
\item\label{cond::Btbdd} $Z\pcar{1}[t]$ is absolutely continuous.
\end{enumerate}
\end{definition}

As mentioned above, property A holding at all times implies the conclusions of Theorem \ref{thm::Sampleconv}:

\begin{proposition}[Dynamics Preserve Property A]
\label{prop::genstat}
Under Assumptions \ref{assu::cond}, \ref{assu::init} and \ref{assu::bds}, and for any $t \in \N_0$, if the model satisfies property A at all times $s \leq t$
, then it satisfies property A at time $t + 1$ as well. 
\end{proposition}

We prove Proposition \ref{prop::genstat} in Section \ref{ssec::genstatpf}. 

\subsection{Proof of Theorem \ref{thm::Sampleconv} given Proposition \ref{prop::genstat}}
\label{ssec::sampleconvpf}

The proof of Theorem \ref{thm::Sampleconv} reduces to the proof of the following lemmas:

\begin{lemma}
\label{lem::t0}
Under Assumptions \ref{assu::cond}, \ref{assu::init} and \ref{assu::bds}, the model satisfies property A at time $0$.
\end{lemma}

\begin{lemma}
\label{lem::induct}
If the model satisfies property A at all times $t \in \N_0$, then the conclusions of Theorem \ref{thm::Sampleconv} hold.
\end{lemma}

\begin{proof}[Proof of Theorem \ref{thm::Sampleconv} given Lemmas \ref{lem::t0} and \ref{lem::induct}]
By Lemma \ref{lem::t0}, the model satisfies property A at time 0. By Proposition \ref{prop::genstat} and induction, this implies that the model satisfies property A at all times $t \in \N_0$, so by Lemma \ref{lem::induct}, the conclusions of Theorem \ref{thm::Sampleconv} hold.
\end{proof}

We now start with the proof of Lemma \ref{lem::t0}.

\begin{proof}[Proof of Lemma \ref{lem::t0}]
Fix an $n \in \N$. We next establish that each condition of property A at time $0$ holds.
\begin{enumerate}[label = (\alph*)]
\item Recall from Section \ref{ssec::notation} that $S_n$ is the permutation group on the set $[1:n]$. Let $\sigma \in S_n$ be any permutation. Let $Y$ be the $n\times n$ identity matrix. For any $i,j \in [1:n]$, let
\[P_{ij} = \begin{cases}
B_0(Z^n_i(0),Z^n_j(0)) &\te{ if } i \neq j,\\
1 &\te{ if } i = j.
\end{cases}\]
For each $i \in [1:n]$, let $X_i = Z^n_i(0)$. Then Assumption \ref{assu::init} (exchangeability of $Z^n(0)$ and initialization of $A^n(0)$) implies that for any permutation $\sigma \in S_n$,
\begin{align*}
\left(X_{1:n},(Y_{ij},P_{ij})_{i,j \in [1:n]}\right) &= \left(\begin{cases}
\left(Z^n_{1:n}(0),(0,B_0(Z^n_i(0),Z^n_j(0)))\right) &\te{ if } i \neq j,\\
\left(Z^n_{1:n}(0),(0,1)\right) &\te{ if } i=j
\end{cases}\right)_{i,j\in [1:n]}\\
&\hspace{-24pt}\deq \left(\begin{cases}
\left(Z^n_{\sigma(1:n)}(0),(0,B_0(Z^n_{\sigma(i)}(0),Z^n_{\sigma(j)}(0)))\right) &\te{ if } \sigma(i) \neq \sigma(j),\\
\left(Z^n_{\sigma(1:n)}(0),(0,1)\right) &\te{ if } \sigma(i) = \sigma(j)
\end{cases}\right)_{i,j\in [1:n]}\\
&\hspace{-24pt}= \left(X_{\sigma(1:n)},(Y_{\sigma(i)\sigma(j)},P_{\sigma(i)\sigma(j)})_{i,j\in [1:n]}\right).
\end{align*}
Therefore, the collection $\left(X_{1:n},(Y_{ij},P_{ij})_{i,j \in [1:n]}\right)$ is jointly exchangeable. Let $A = A^n(0)$. By Assumption \ref{assu::cond}(b), $(A_{ij})_{(i,j)\in \incs_n}$ are mutually conditionally independent given $\sigma(X,Y) = \sigma(Z^n_{1:n}(0)) = \Fmc^{A,n}_0$ and for each $i,j\in[1:n]$, $P_{ij} = \PP(A_{ij}=1|X,Y)$. Thus, $X,Y,P$ and $A$ satisfy the conditions of Lemma \ref{lem::Addberexchexcl}, so
\[\left(X_{1:n},(Y_{ij},A_{ij})_{i,j \in [1:n]}\right) = \left(Z^n_{1:n}(0),\left(\indic{i=j},A^n_{ij}(0)\right)_{i,j\in [1:n]}\right),\]
is jointly exchangeable, which implies that $\left(Z^n(0),A^n(0)\right)$ is likewise jointly exchangeable.

\item This follows directly from Assumption \ref{assu::bds}(a) and \cite[Theorem 4.5.9]{Bog07}, which states that a sequence of random variables $\{X_i\}_{i \in \N}$ are uniformly integrable if and only if there exists a convex, superlinear function $G$ such that $\{G(X_i)\}_{i\in \N}$ is bounded in expectation.

\item Fix any $k \in\N$ and assume $n > k$. For any $\bm{a} \defeq (a_{ij})_{(i,j)\in \incs_k} \in \{0,1\}^{\incs_k}$, \eqref{eqn::Anevolve}, Assumptions \ref{assu::cond}(b), \ref{assu::init}(b) and \eqref{eqn::Aevolve} imply 
\begin{align*}
\PP\left((A^n_{ij})_{(i,j)\in \incs_k} = \bm{a}\middle|Z^n_{1:k}(0)\right) & = \ex{\PP\left((A^n_{ij})_{(i,j)\in \incs_k} = \bm{a}\middle|Z^n(0)\right)\middle| Z^n_{1:k}(0)}\\ 
&= \ex{\prod_{(i,j)\in \incs_k}\alt{B}_0(a_{ij},Z^n_i(0),Z^n_j(0))\middle| Z^n_{[1:k]}(0)}\\
&= \prod_{(i,j)\in \incs_k}\alt{B}_0(a_{ij},Z^n_i(0),Z^n_j(0))\\
&=\prod_{(i,j)\in \incs_k}\PP\left(A^n_{ij} = a_{ij}\middle|Z^n_{1:k}(0)\right)
\end{align*}

This proves that $(A^n_{ij}(0))_{(i,j)\in \incs_k}$ are mutually conditionally independent given $Z^n_{1:k}(0)$ and that for any $(i,j)\in \incs_k$,
\[\ex{A^n_{ij}(0)\middle|Z^n_{1:k}(0)} = B_0(Z^n_i(0),Z^n_j(0)).\]
Let $c_1,\dots,c_{\binom{k}{2}}$ be an enumeration of the set $\incs_k$. Given the a.e. continuity of $B_0$, it follows by Assumption \ref{assu::init}(a) (together with Proposition \ref{prop::convres}) that the conditions of Lemma \ref{lem::cibrlem} are satisfied for $X^n \defeq Z^n_{1:k}(0)$, $X = Z\pcar{1:k}(0)$, $I = \{0,1\}$, $D^n_j = A^n_{c_j}$, $\phi_j^1(z) = B(z_{c_{j,1}},z_{c_{j,2}})$ and $\phi_j^0(z) = 1 - \phi_j^1(z)$. Thus, by Lemma \ref{lem::cibrlem},
\[(Z^n_{1:k}(0),(A^n_{ij}(0))_{(i,j)\in \incs_k}) \Rightarrow (Z\pcar{1:k}(0), (B\pcar{ij})_{(i,j)\in \incs_k}),\]
where $(B\pcar{ij})_{(i,j)\in \incs_k}$ are mutually conditionally independent Bernoulli random variables given $Z\pcar{1:k}(0)$ and for each $(i,j)\in \incs_k$,
\[\PP(B\pcar{ij} = 1|Z\pcar{1:k}(0)) = B_0(Z\pcar{i}(0),Z\pcar{j}(0)).\]
However, this is precisely how the distribution of $A\pcar{ij}(0)$ is defined in Definition \ref{defn::AZlim}, so by Proposition \ref{prop::welldef},
\[(Z^n_{1:k}(0),(A^n_{ij}(0))_{(i,j)\in \incs_k}) \Rightarrow (Z\pcar{1:k}(0), (B\pcar{ij})_{(i,j)\in \incs_k}) \deq (Z\pcar{1:k}(0), (A\pcar{ij}(0))_{(i,j)\in \incs_k}).\]
By symmetry ($A\pcar{ij} = A\pcar{ji}$) and the fact that $A\pcar{ii}(0) = 1$ for all $i \in \N$, it immediately follows that 
\[\left(Z^n_{1:k}(0),A^n_{1:k,1:k}(0)\right) \Rightarrow \left(Z\pcar{1:k}(0),A\pcar{1:k,1:k}(0)\right)\]
as desired.

\item When $t = 0$, this is given by Assumption \ref{assu::init}(a).
\end{enumerate}
\end{proof}

We finish with the proof of Lemma \ref{lem::induct}.

\begin{proof}[Proof of Lemma \ref{lem::induct}]
Fix any $t \in \N_0$ and suppose that $M^n_k$ is a sample without replacement. Then we may note that there exists a uniformly random $\sigma \in S_n$ independent of $Z^n[t],A^n[t],Z\pcar{1:k}[t]$, and $A\pcar{1:k,1:k}[t]$ such that
\[M^n_k = \{\sigma(1),\dots,\sigma(k)\}.\]
Note that for each $\phi \in S_n$, $\PP(\sigma = \phi) = \PP(\sigma = I_{S_n})$ where $I_{S_n}\in S_n$ is the identity permutation. Then by property A(a), $(Z^n[t],A^n[t])$ is jointly exchangeable, so by property A(c),
\begin{align*}
\left(Z^n_{m^n_{1:k}}[t],A^n_{m^n_{1:k}m^n_{1:k}}[t]\right) &= \left(Z^n_{\sigma(1:k)}[t],A^n_{\sigma(1:k)\sigma(1:k)}[t]\right) \\
&\deq \left(Z^n_{1:k}[t],A^n_{1:k,1:k}[t] \right) \\
&\Rightarrow \left(Z\pcar{1:k}[t],A\pcar{1:k,1:k}[t]\right).
\end{align*}
This is precisely the conclusion of Theorem \ref{thm::Sampleconv}. Now, suppose $M^n_k$ is a sample taken with replacement. Let $\Emc^n_k = \{m^n_i \neq m^n_j\te{ for all distinct }i,j \in [1:k]\}$. Then we note that
\[\PP\left(\Emc^n_k\right) = \frac{\frac{n!}{(n-k)!}}{n^k} \to 1\te{ as } n\to \infty.\]
Therefore, for any $f\in C_b\left(\left(\mf_t\left((\R^d)^k\right)\right) \times \left(\mf_t\left(\{0,1\}^{k\times k}\right)\right)\right)$, 
\begin{align*}
\ex{f\left(Z^n_{m^n_{1:k}}[t],A^n_{m^n_{1:k}m^n_{1:k}}[t]\right)\indic{(\Emc^n_k)^c}} \to 0\te{ as } n \to \infty.
\end{align*}
Given that $\sigma(M^n_k,\Emc^n_k)$ is independent of $(Z^n[t],A^n[t])$ and that the conditional distribution of $M^n_k$ given $\Emc^n_k$ is that of a uniform sample without replacement,
\begin{align*}
\lim_{n\to\infty} \ex{f\left(Z^n_{m^n_{1:k}}[t],A^n_{m^n_{1:k}m^n_{1:k}}[t]\right)} &= \lim_{n\to\infty} \ex{f\left(Z^n_{m^n_{1:k}}[t],A^n_{m^n_{1:k}m^n_{1:k}}[t]\right)\indic{\Emc^n_k}}\\
&= \lim_{n\to\infty} \ex{f\left(Z^n_{1:k}[t],A^n_{1:k,1:k}[t]\right)}\PP\left(\Emc^n_k\right)\\
&= \ex{f\left(Z\pcar{1:k}[t],A\pcar{1:k,1:k]}[t]\right)},
\end{align*}
so 
\begin{align*}
\left(Z^n_{m^n_{1:k}}[t],A^n_{m^n_{1:k}m^n_{1:k}}[t] \right) \Rightarrow \left(Z\pcar{1:k}[t],A\pcar{1:k,1:k}[t]\right),
\end{align*}
completing the proof.
\end{proof}

\subsection{Proof of Proposition \ref{prop::genstat}}
\label{ssec::genstatpf}

Fix any $t \in \N_0$. Assume that property A holds at time $t$. Then, to complete the proof, we need to show that it also holds at time $t + 1$. We break this proof down into multiple parts.

\subsubsection{Proof that Property A(a) holds at time $t + 1$}
\label{sssec::Aapf}

\ind To establish property A(a) at time $t + 1$, we start with the following useful intermediate result:

\begin{lemma}
\label{lem::ztp1atjexch}
If property A holds at time $t$, then the quantity $\left(Z^n[t+1],A^n[t]\right)$ is jointly exchangeable for all $n \in \N$.
\end{lemma}
\begin{proof}
Fix any $n \in\N$ and $\sigma \in S_n$. By Lemma \ref{ZLexch}, $(Z^n[t],L^n(t),A^n[t])$ is jointly exchangeable. By Assumption \ref{assu::cond}(a), $(\xi_i(t))_{i=1}^n$ are i.i.d. and independent of $(Z^n[t],L^n(t),A^n[t])$, so it follows that $(Z^n[t],L^n(t),\xi(t),A^n[t])$ is also jointly exchangeable. Then
\begin{align*}
\left(Z^n_i[t+1],A^n_{ij}[t]\right)_{i,j \in [n]} &= \left(Z^n_i[t],(1 - \gamma)Z^n_i(t) + \gamma L^n_i(t) + \xi_i(t),A^n_{ij}[t]\right)_{i,j \in [n]}\\
&\deq \left(Z^n_{\sigma(i)}[t],(1 - \gamma)Z^n_{\sigma(i)}(t) + \gamma L^n_{\sigma(i)}(t) + \xi_{\sigma(i)}(t),A^n_{{\sigma(i)}{\sigma(j)}}[t]\right)_{i,j \in [n]}\\
&=\left(Z^n_{\sigma(i)}[t+1],A^n_{\sigma(i)\sigma(j)}[t]\right)_{i,j \in [n]},
\end{align*}
where the penultimate equality holds by permuting the terms in the summation. Thus, $(Z^n[t+1],A^n[t])$ is jointly exchangeable for all $n \in \N$.
\end{proof}

Now we can establish property A(a):

\begin{lemma}
\label{jexchpf}
If property A holds at time $t$, then property A(a) holds at time $t + 1$.
\end{lemma}
\begin{proof}
Fix any $n \in\N$. By Lemma \ref{lem::ztp1atjexch}, $\left(Z^n[t+1],A^n[t]\right)$ is jointly exchangeable. For each $i,j \in [n]$ define
\[P^n_{ij} = B(A^n_{ij}(t),Z^n_i(t+1),Z^n_j(t+1)).\]
Then it is easy to see that the joint exchangeability of $\left(Z^n[t+1],A^n[t]\right)$ extends to $\left(Z^n[t+1], A^n[t], P^n\right)$. Furthermore, by assumption, $\{A^n_{ij}(t+1)\}_{1\leq i\leq j\leq n}$ are mutually conditionally independent given $\sigma(Z^n[t+1],A^n[t])$, $A^n_{ij}(t+1) = A^n_{ji}(t+1)$ and
\[\ex{A^n_{ij}(t+1)\middle|Z^n[t+1], A^n[t]} = P^n_{ij}\]
for all $i,j \in [n]$. Then by Lemma \ref{Addbernexch}, $(Z^n[t+1],A^n[t+1])$ is jointly exchangeable as desired.
\end{proof}

\subsubsection{Proof that Property A(d) holds at time $t + 1$}
\label{sssec::Adpf}

It suffices to prove that the conditional distribution of $Z\pcar{1}(t+1)$ given $Z\pcar{1}[t]$ is absolutely continuous. By \eqref{eqn::Zevolve}, $Z\pcar{1}(t+1) - \xi_1(t)$ is $\sigma(Z\pcar{1}[t])$-measurable, so conditioned on $Z\pcar{1}[t]$, it is constant. By Assumption \ref{assu::cond}(a), $\xi_1(t)$ is independent of $Z\pcar{1}[t]$ and absolutely continuous. So, conditioned on $Z\pcar{1}[t]$, $Z\pcar{1}(t+1)$ is the sum of a constant vector and an absolutely continuous random vector and is therefore absolutely continuous as well.

\subsubsection{Proof that Property A(b) holds at time $t + 1$}
\label{sssec::Abpf}

The proof differs depending on whether Assumption \ref{assu::bds}(c)(i) holds or Assumption \ref{assu::bds}(c)(ii) holds. \\

\textbf{Proof of property A(b) at time $t+1$ given Assumption \ref{assu::bds}(c)(ii):} We begin with the claim that the random vectors $\{L^n_1\}_{n \in \N}$ are uniformly integrable. Note that by Lemma \ref{lem::ztp1atjexch}, $\{L^n_i(t)\}_{i \in [1:n]}$ is exchangeable for all $n$. This implies that $L^n_i(t) \deq L^n_1(t)$ for all $n \in \N$ and $i \in [1:n]$ so the claim implies that the sequence $\{L^n_i(t)\}_{n \in \N,i \in [1:n]}$ is uniformly integrable.

\ind It is a standard result that given two uniformly integrable sequences of random vectors $\{X_i\}_{i \in I}, \{Y_i\}_{i \in I}$ (where $I$ is a countable index set) and two constants $a,b \in \R$, $\{aX_i + bY_i\}_{i \in I}$ is likewise uniformly integrable. This naturally extends to three sequences. The lemma then follows by noting that $\{Z^n_i(t)\}_{n \in \N, i \in [1:n]}$ is uniformly integrable because property A(b) holds at time $t$, $\{L^n_i(t)\}_{n \in \N, i \in [1:n]}$ is uniformly integrable as argued above, and by Assumption \ref{assu::cond}(a), $\{\xi_i(t)\}_{n \in \N, i \in [1:n]}$ is a collection of $L^1$, i.i.d. random variables and is therefore also uniformly integrable. Thus,
\[\{Z^n_i(t+1)\}_{n \in \N, i \in [1:n]} = \{\gamma Z^n_i(t) + (1 - \gamma)L^n_i(t) + \xi_i(t)\}_{n \in \N, i \in [1:n]}\]
is uniformly integrable, completing the proof.

\ind We now prove the claim. By \cite[Theorem 4.5.9]{Bog07}, property A(b) at time $t$ implies that there exists a convex, superlinear function $g: \R \to \R$ such that
\[\sup_{n \in \N, i\in [1:n]} \ex{g(|2Z^n_i(t)|)} \defeq M < \infty,\]
where we use the fact that property A(b) at time $t$ implies that $\{2Z^n_i(t)\}_{n \in \N, i \in [1:n]}$ is uniformly integrable. We note that by Assumption \ref{assu::bds}(c)(ii), $A^n_{1j}(t) = 0$ whenever $|Z^n_1(t) - Z^n_j(t)| > C_b$. Notably, this implies that $A^n_{1j}(t) = 0$ when $|Z^n_j(t)| > |Z^n_1(t)| + C_b$. So,
\[|L^n_1(t)|\leq \sum_{j \in [1:n]} \frac{|A^n_{1j}(t)Z^n_j(t)|}{\sum_{k \in [1:n]} A^n_{1k}} \leq |Z^n_1(t)| + C_b.\]
Therefore, by convexity of $g$,
\[\sup_{n \in \N} \ex{g(|L^n_1(t)|)}\leq \sup_{n \in \N} \ex{g(|Z^n_1(t)|+ C_b)}\leq \frac{1}{2}\left(\ex{g(2C_b)} + \sup_{n \in \N}\ex{g(2|Z^n_1(t)|}\right) < \infty.\]
By \cite[Theorem 4.5.9]{Bog07}, this proves that $\{L^n_1(t)\}_{n \in \N}$ is uniformly integrable.\\

\textbf{Proof of property A(b) at time $t+1$ given Assumption \ref{assu::bds}(c)(i):} We provide an inductive proof of the following two equations for all $s \leq t+1$ and $C > 0$:
\begin{align}
\limsup_{n\to\infty} \ex{\exp{C|Z^n_1(s)|}} &< \infty\label{eqn::zsupexpLn}\\
\label{eqn::lsupexpLn}
\limsup_{n\to\infty} \ex{\exp{C|L^n_1(s)|}} &< \infty.
\end{align}
As a base case, we show that \eqref{eqn::zsupexpLn} holds for $s=0$. Then, we apply two inductive arguments. First, we show that \eqref{eqn::zsupexpLn} implies \eqref{eqn::lsupexpLn} at any time $s$. Then we show that \eqref{eqn::lsupexpLn} and \eqref{eqn::zsupexpLn} at time $s$ imply \eqref{eqn::zsupexpLn} at time $s+1$. Together, these arguments plus the base case imply the above equations for all values of $s$.

\textbf{Base Case:} \eqref{eqn::zsupexpLn} holds for $s=0$ by Assumption \ref{assu::bds}(a).\\

\textbf{First Inductive Argument:} We show that if \eqref{eqn::zsupexpLn} holds at time $s$, then \eqref{eqn::lsupexpLn} holds at time $s$ as well. For each $C > 0$, define the function $\psi_C: \R^d \to \R$ by
\[\psi_C(z)\defeq \exp\left(C|z|\right).\]
It is easily verified that $\psi_C$ is increasing, convex and superlinear for all $C > 0$. We define the non-decreasing function $M_{s,Z}$ by
\begin{equation}
\label{eqn::MC}
M_{s,Z}(C) \defeq \limsup_{n\to\infty} \ex{\exp\left(C|Z^n_1(s)|\right)}.
\end{equation}
By \eqref{eqn::zsupexpLn}, $M_{s,Z}: \R\to \R_+$ is a finite-valued function. 

\ind Fix $C > 0$. Then applying the convexity of $\psi_C$ and the fact that $\sum_{j=1}^n \frac{A^n_{1j}(s)}{\sum_{i=1}^n A^n_{1i}(s)} = 1$, we can break up the expectation of $\ex{\psi_C(L^n_1(s))}$ as follows:
\begin{equation}
\label{eqn::breakupsum}
\ex{\psi_C(L^n_1(s))} = \ex{\psi_C\left(\frac{1}{\sum_{i=1}^n A^n_{1i}(s)}\sum_{j=1}^n A^n_{1j}(s)Z^n_j(s)\right)} \leq \sum_{j=1}^n\ex{\frac{A^n_{1j}(s)}{\sum_{i=1}^n A^n_{1i}(s)}\psi_C\left(Z^n_j(s)\right)}.
\end{equation}
The $j=1$ term of the sum above is easily reduced using the fact that $\frac{A^n_{1j}(s)}{\sum_{i=1}^n A^n_{1i}(s)} \leq 1$:
\begin{equation}
\label{eqn::j=1breakup}
\limsup_{n\to\infty}\ex{\frac{A^n_{11}(s)}{\sum_{i=1}^n A^n_{1i}(s)}\psi_C\left(Z^n_1(s)\right)}\leq \limsup_{n\to\infty}\ex{\psi_C\left(Z^n_1(s)\right)} = M_{s,Z}(C).
\end{equation}
We can now examine the remaining terms. Recall the $\sigma$-algebras $\Fmc^{A,n}_s \defeq \sigma(Z^n(s),A^n(s-1))$ if $s > 0$, where $\Fmc^{A,n}_0 \defeq \sigma(Z^n(0))$. For $j \in [1:n]$, \eqref{eqn::Anevolve} and Assumption \ref{assu::init}(b) imply
\[P^n_j \defeq \ex{A^n_{1j}(s)|\Fmc^{A,n}_s} = \begin{cases}
B\left(A^n_{1j}(s-1),Z^n_1(s),Z^n_j(s)\right)  &\te{ if } s > 0\te{ and }j\neq 1,\\
B_0(Z^n_1(0),Z^n_j(0))&\te{ if } s = 0\te{ and }j\neq 1,\\
1 &\te{ if } j = 1.
\end{cases}\]
Define
\[P^n_{-2} =\frac{1}{n-1}\sum_{i\neq 2} P^n_i.\]
Then, we perform the following computation. In \eqref{eqn::j>1exch}, we apply joint exchangeability of $(Z^n(s),A^n(s))$ (which holds by property A(a)). In \eqref{eqn::j>1A12=1}, we make use of the fact that $\psi_C(Z^n_2(s))$ is $\Fmc^{A,n}_s$-measurable and that $\{A^n_{1j}\}_{j \in [n]}$ are conditionally independent given $\Fmc^{A,n}_s$. The equation is obtained by noticing that when $A^n_{12}(s) = 0$, the whole expression in the expectation is equal to $0$. In \eqref{eqn::j>1tonelli}, we make use of the fact that for any $x \geq 0$, $\int_0^1 s^x \,ds = \frac{1}{1+x}$, then apply Tonelli's theorem to pull the integral out of the conditional expectation. \eqref{eqn::j>1CI} follows from the fact that $\{A^n_{1j}(s)\}_{j \in [n]}$ are conditionally independent given $\Fmc^{A,n}_s$. \eqref{eqn::j>1AMGM} follows from an application of the AM-GM inequality (arithmetic means are greater than or equal to geometric means). Lastly, \eqref{eqn::j>1Holder} is just an application of H\"{o}lder's inequality.
\begin{align}
\limsup_{n\to\infty}\sum_{j=2}^n&\ex{\frac{A^n_{1j}(s)}{\sum_{i=1}^n A^n_{1i}(s)}\psi_C\left(Z^n_j(s)\right)} \nonumber\\
&= \limsup_{n\to\infty}\sum_{j=2}^n\ex{\frac{A^n_{12}(s)}{\sum_{i=1}^n A^n_{1i}(s)}\psi_C\left(Z^n_2(s)\right)}\label{eqn::j>1exch}\\
&\leq \limsup_{n\to\infty} n\ex{\frac{A^n_{12}(s)}{\sum_{i=1}^n A^n_{1i}(s)}\psi_C\left(Z^n_2(s)\right)}\nonumber\\
&= \limsup_{n\to\infty} n\ex{\ex{\frac{A^n_{12}(s)}{\sum_{i=1}^n A^n_{1i}(s)}\psi_C\left(Z^n_2(s)\right)\middle|\Fmc^{A,n}_s}}\nonumber\\
&= \limsup_{n\to\infty} n\ex{P^n_2\psi_C\left(Z^n_2(s)\right)\ex{\frac{1}{1 + \sum_{i\neq 2} A^n_{1i}(s)}\middle|\Fmc^{A,n}_s}}\label{eqn::j>1A12=1}\\
&= \limsup_{n\to\infty} n\ex{P^n_2\psi_C\left(Z^n_2(s)\right)\int_0^1\ex{s^{\sum_{i\neq 2} A^n_{1i}(s)}\middle|\Fmc^{A,n}_s}\,ds}\label{eqn::j>1tonelli}\\
&= \limsup_{n\to\infty} n\ex{P^n_2\psi_C\left(Z^n_2(s)\right)\int_0^1\prod_{i\neq 2}\ex{s^{A^n_{1i}(s)}\middle|\Fmc^{A,n}_s}\,ds}\label{eqn::j>1CI}\\
&= \limsup_{n\to\infty} n\ex{P^n_2\psi_C\left(Z^n_2(s)\right)\int_0^1\prod_{i\neq 2}\left(P^n_is + (1 - P^n_i)\right)\,ds}\nonumber\\
&= \limsup_{n\to\infty} n\ex{P^n_2\psi_C\left(Z^n_2(s)\right)\int_0^1\prod_{i\neq 2}\left(1 - P^n_is\right)\,ds}\nonumber\\
&\leq \limsup_{n\to\infty} n\ex{P^n_2\psi_C\left(Z^n_2(s)\right)\int_0^1\left(1 - P^n_{-2}s\right)^{n-1}\,ds}\label{eqn::j>1AMGM}\\
&= \limsup_{n\to\infty} \ex{P^n_2\psi_C\left(Z^n_2(s)\right)
\frac{1 - (1 - P^n_{-2})^n}{P^n_{-2}}}\nonumber\\
&\leq \limsup_{n\to\infty} \ex{\frac{\psi_C\left(Z^n_2(s)\right)}{P^n_{-2}}}\nonumber\\
&\leq \limsup_{n\to\infty}\sqrt{\ex{(\psi_C(Z^n_2(s))^2}\ex{\frac{1}{(P^n_{-2})^2}}},\label{eqn::j>1Holder}\\
&\leq \limsup_{n\to\infty}\sqrt{\ex{\psi_{2C}(Z^n_1(s)}\ex{\frac{1}{(P^n_{-2})^2}}},\nonumber\\
&= \limsup_{n\to\infty} \sqrt{\ex{\frac{M_{s,Z}(2C)}{(P^n_{-2})^2}}}.\label{eqn::j>1}
\end{align}

By Assumption \ref{assu::bds}(c)(i), there exists a constant $\bar{C}> 0$ such that for all $n \in \N$ and $j \in [1:n]$,
\[P^n_j \geq \bar{C}\exp\left(-C_b|Z^n_1(s) - Z^n_j(s)|\right).\]
Then, we can make the following computation, where we apply the AM-GM inequality once more in \eqref{eqn::PsqinvAMGM}.
\begin{align}
\frac{1}{(P^n_{-2})^2} &\leq \frac{1}{\left(\frac{\bar{C}}{n-1}\sum_{j\neq 2} \exp\left(-C_b|Z^n_1(s) - Z^n_j(s)|\right)\right)^2}\nonumber\\
&=\frac{1}{\frac{\bar{C}^2}{(n-1)^2}\sum_{i,j\neq 2} \exp\left(-C_b\left(|Z^n_1(s) - Z^n_i(s)| + |Z^n_1(s) - Z^n_j(s)|\right)\right)}\nonumber\\
&=\left(\frac{1}{\bar{C}^2}\right)\frac{1}{\frac{1}{(n-1)^2}\sum_{i,j\neq 2} \exp\left(-C_b\left(|Z^n_1(s) - Z^n_i(s)| + |Z^n_1(s) - Z^n_j(s)|\right)\right)}\nonumber\\
&\leq\left(\frac{1}{\bar{C}^2}\right)\frac{1}{\exp\left(-\frac{C_b}{(n-1)^2}\sum_{i,j\neq 2}\left(|Z^n_1(s) - Z^n_i(s)| + |Z^n_1(s) - Z^n_j(s)|\right)\right)}\label{eqn::PsqinvAMGM}\\
&=\left(\frac{1}{\bar{C}^2}\right)\exp\left(\frac{C_b}{(n-1)^2}\sum_{i,j\neq 2}\left(|Z^n_1(s) - Z^n_i(s)| + |Z^n_1(s) - Z^n_j(s)|\right)\right)\nonumber\\
&=\left(\frac{1}{\bar{C}^2}\right)\exp\left(\frac{C_b(2n-3)}{(n-1)^2}\sum_{k=3}^n|Z^n_1(s) - Z^n_k(s)|\right).\label{eqn::Psqinv}
\end{align}

We can compute the expected value of this quantity in the limit as $n\to \infty$ by making use of the following consequence of the generalized H\"{o}lder's inequality. If $X_1,\dots,X_m$ are identically distributed (but not necessarily independent), then 
\[\ex{\prod_{i=1}^m X_i} \leq \prod_{i=1}^m \|X_i\|_m = \|X_1\|_m^m = \ex{X_1^m}.\]

Below, we apply this result in \eqref{eqn::exPsqinvGHI} and \eqref{eqn::exPsqinvGHI2}. In \eqref{eqn::exPsqinvfn<3}, we use the fact that $\frac{(n-2)(2n-3)}{(n-1)^2} < 2$ for all $n \in \N$. \eqref{eqn::exPsqinvsimple} applies the triangle inequality. 

\begin{align}
\limsup_{n\to\infty} \ex{\frac{1}{(P^n_{-2})^2}} &\leq \frac{1}{\bar{C}^2}\limsup_{n\to\infty} \ex{\exp\left(\frac{C_b(2n-3)}{(n-1)^2}\sum_{k=3}^n|Z^n_1(s) - Z^n_k(s)|\right)}\nonumber\\
&=\frac{1}{\bar{C}^2}\limsup_{n\to\infty} \ex{\prod_{k=3}^n \exp\left(\frac{C_b(2n-3)}{(n-1)^2}|Z^n_1(s) - Z^n_k(s)|\right)}\nonumber\\
&\leq \frac{1}{\bar{C}^2}\limsup_{n\to\infty} \ex{\exp\left(\frac{C_b(n-2)(2n-3)}{(n-1)^2}|Z^n_1(s) - Z^n_3(s)|\right)}\label{eqn::exPsqinvGHI}\\
&\leq \frac{1}{\bar{C}^2}\limsup_{n\to\infty} \ex{\exp\left(2C_b|Z^n_1(s) - Z^n_3(s)|\right)}\label{eqn::exPsqinvfn<3}\\
&\leq \frac{1}{\bar{C}^2}\limsup_{n\to\infty} \ex{\exp\left(2C_b\left(|Z^n_1(s)| + |Z^n_3(s)|\right)\right)}\label{eqn::exPsqinvsimple}\\
&\leq \frac{1}{\bar{C}^2}\limsup_{n\to\infty} \ex{\exp\left(4C_b|Z^n_1(s)|\right)}\label{eqn::exPsqinvGHI2}\\
&\leq \frac{M_{s,Z}(4C_b)}{\bar{C}^2}.\label{eqn::exPsqinv}
\end{align}

Combining \eqref{eqn::j>1} and \eqref{eqn::exPsqinv},
\begin{align*}
\limsup_{n\to\infty} \ex{\psi_C(L^n_1(s))}&\leq \limsup_{n\to\infty}\sqrt{\ex{\frac{M_{s,Z}(2C)}{(P^n_{-2})^2}}}\\
&\leq \frac{\sqrt{M_{s,Z}(2C)M_{s,Z}(4C_b)}}{\bar{C}}\label{z-xi}\\
&< \infty.
\end{align*}

\textbf{Second Inductive Argument:} We show that \eqref{eqn::zsupexpLn} and \eqref{eqn::lsupexpLn} at time $s$ imply \eqref{eqn::zsupexpLn} at time $s+1$. Note that
\[Z^n_1(s+1) - \xi^1(s) = (1-\gamma)Z^n_1(s) + \gamma L_1^n(s).\]
By convexity of $\psi_C$ and \eqref{eqn::zsupexpLn}-\eqref{eqn::lsupexpLn}, this implies that
\begin{align*}
\limsup_{n\to\infty} \ex{\psi_C(Z^n_1(s+1) - \xi^1(s))}&=\limsup_{n\to\infty} \ex{\psi_C((1-\gamma)Z^n_1(s) + \gamma L^n_1(s))}\\
&\leq \limsup_{n\to\infty} (1 - \gamma) \ex{\psi_C(Z^n_1(s))} + \gamma\ex{\psi_C(L^n_1(s))} \\
&\leq(1-\gamma)M_{s,Z}(C)+\gamma\frac{\sqrt{M_{s,Z}(2C)M_{s,Z}(4C_b)}}{\bar{C}}\\
&< \infty \te{ for all }C > 0.
\end{align*}
By \eqref{eqn::Znevolve} and \eqref{eqn::Lnidef}, $Z^n_1(s+1)- \xi_1(s)$ is $\filt^n_t$-measurable, so by Assumption \ref{assu::cond}(a), $Z^n_1(s+1)-\xi_1(s)$ and $\xi_1(s)$ are independent. Applying this independence and the above display,
\begin{align*}
M_{s+1,Z}(C) &\defeq \limsup_{n\to\infty} \ex{\psi_C(Z^n_1(s+1))} \\
&= \limsup_{n\to\infty} \ex{\exp\left(C|Z^n_1(s+1)-\xi_1(s) + \xi_1(s)|\right)}\\
&\leq \limsup_{n\to\infty} \ex{\exp\left(C\left(|Z^n_1(s+1)-\xi_1(s)| + |\xi_1(s)|\right)\right)}\\
&\leq \limsup_{n\to\infty} \ex{\exp\left(2C|Z^n_1(s+1)-\xi_1(s)|\right)\exp\left(2C|\xi_1(s)|\right)}\\
&\leq \limsup_{n\to\infty} \ex{\exp\left(2C|Z^n_1(s+1)-\xi_1(s)|\right)}\ex{\exp\left(2C|\xi_1(s)|\right)}\\
&= \left((1-\gamma)M_{s,Z}(2C) + \gamma\frac{\sqrt{M_{s,Z}(4C)M_{s,Z}(8C_b)}}{\bar{C}}\right)M_{\xi}(2C)\\
&< \infty,
\end{align*}
where $M_{\xi}$ is the mapping defined in Assumption \ref{assu::bds}(b).

\subsubsection{Proof that Property A(c) holds at time $t+1$}
\label{sssec::Acpf}
Throughout the section, recall that Assumptions \ref{assu::cond} and \ref{assu::bds} hold at time $t$, and Property A holds at time $t$. For any $k \in \N$, this implies
\begin{equation}
\label{eqn::ZAconvtimet}
(Z^n_{1:k}[t],A^n_{1:k,1:k}[t]) \Rightarrow \left(Z\pcar{1:k}[t],A\pcar{1:k,1:k}[t]\right).
\end{equation} 

The full proof that Property A(c) holds at time $t+1$ is long, so we first provide a proof outline in which technical details are omitted.

\begin{proof}[Proof of Property A(c) Outline:]
We prove this in four steps, some of which are described by a lemma. Consider the following random measures defined for $n \in \N$:
\begin{equation}
\label{eqn::altmu}
\alt{\mu}^{n}_t \defeq \frac{1}{n}\sum_{j=1}^n \delta_{Z^n_1[t],Z^n_j[t],A^n_{1j}(t)}\te{ and } \alt{\mu}_t \defeq \law\left(Z\pcar{1}[t],Z\pcar{2}[t],A\pcar{12}(t)|Z\pcar{1}[t]\right).
\end{equation}
In step 1 of the proof, we use Proposition \ref{prop::condconvres} to show that conditional propagation of chaos holds in this regime:
\begin{lemma}
\label{lem::cpcpf}
The following convergence holds:
\begin{equation}
\label{eqn::altmucpc}
\left(Z^n_1[t],\alt{\mu}^n_t\right)\Rightarrow \left(Z\pcar{1}[t],\alt{\mu}_t\right).
\end{equation}
\end{lemma}

In step 2 of the proof, we apply exchangeability, \eqref{eqn::ZAconvtimet}, \eqref{eqn::altmucpc} and a conditional Slutzky's lemma (Lemma \ref{Slutzkycond}) to establish the joint convergence of $Z^n_{1:k}[t]$, $L^n_{1:k}[t]$ and $A^n_{1:k,1:k}[t]$:
\begin{lemma}
The following convergence holds:
\label{lem::ZALjoint}
\begin{equation}
\label{eqn::ZALjoint}
\left(Z^n_{1:k}[t],L^n_{1:k}(t),A^n_{1:k,1:k}[t]\right) \Rightarrow \left(Z\pcar{1:k}[t],L\pcar{1:k}(t),A\pcar{1:k,1:k}[t]\right).
\end{equation}
\end{lemma}
The proof of Lemma \ref{lem::ZALjoint} requires the following technical lemma:

\begin{lemma}
\label{lem::posdenom}
The following expression holds:
\[\ex{A\pcar{12}(t)|Z\pcar{1}[t]} > 0 \te{ a.s..}\]
\end{lemma}
This ensures that the denominator of $L^n_1(t)$ does not vanish as $n \to\infty$.

\ind In step 3, we show that Lemma \ref{lem::ZALjoint} implies the joint convergence of $Z^n_{1:k}[t+1]$ and $A^n_{1:k,1:k}[t]$:
\begin{lemma}
\label{lem::Zt+Atjoint}
The following convergence holds: 
\begin{equation}
\label{eqn::Zt+Atjoint}
\left(Z^n_{1:k}[t+1],A^n_{1:k,1:k}[t]\right) \Rightarrow \left(Z\pcar{1:k}[t+1],A\pcar{1:k,1:k}[t]\right).
\end{equation}
\end{lemma}
In step 4, we apply the continuous mapping theorem and Lemma \ref{lem::cibrlem} to complete the proof.
\end{proof}

We now prove the result starting with step 4.

\begin{proof}[Proof that Property A(c) holds at time $t+1$ given Lemma \ref{lem::Zt+Atjoint}] Suppose we know Lemma \ref{lem::Zt+Atjoint} holds. Then
\[\left(Z^n_{1:k}[t+1],A^n_{1:k,1:k}[t]\right) \Rightarrow \left(Z\pcar{1:k}[t+1],A\pcar{1:k,1:k}[t]\right)\]
by Lemma \ref{lem::Zt+Atjoint}.

\ind  We now apply Lemma \ref{lem::cibrlem}. Let $\Xmc = (\mf_{t+1}(\R^d))^k\times \left(\mf_t\left(\{0,1\}\right)\right)^{k\times k}$. For each $n$ let $X_n \defeq (Z^n_{1:k}[t+1],A^n_{1:k,1:k}[t])$ be an $\Xmc$-random element. In addition, let $X \defeq (Z\pcar{1:k}[t+1],A\pcar{1:k,1:k}[t])$ be an $\Xmc$-random element. Then by the above display, $X_n \Rightarrow X$, and $\sigma(X_n) = \sigma\left(Z^n_{1:k}[t+1],A^n_{1:k,1:k}[t]\right)\subset \Fmc^{A,n}_{t+1}$. Let $a_1,\dots,a_{k(k-1)/2}$ be an enumeration of the set $\incs_k$ and let $D^n_j = A^n_{a_j}(t+1)$ (where $D^n_j$ are $I \defeq \{0,1\}$-valued random variables). Then by Assumption \ref{assu::cond}(b), $\{D^n_j\}_{j \in [1:k(k-1)/2]}$ is a conditionally independent sequence of Bernoulli random variables given $\Fmc^{A,n}_{t+1}$. Given that $\sigma(X_n)\subset \Fmc^{A,n}_{t+1}$ and for each $j$,
\[\PP(D^n_j=1|\Fmc^{A,n}_{t+1}) = B(A^n_{a_j}(t),Z^n_{a_{j,1}}(t+1),Z^n_{a_{j,2}}(t+1))\]
is $\sigma(X_n)$-measurable, Lemma \ref{lem:Bercondindsubfilt} implies that $\{D^n_j\}_{j \in [1:k(k-1)/2]}$ is a conditionally independent sequence of Bernoulli random variables given $X_n$ for each $n$. Moreover, 
\[\phi_j^1(X_n) \defeq B\left(A^n_{a_j}(t),Z^n_{a_{j,1}}(t+1),Z^n_{a_{j,2}}(t+1)\right) \te{ and } \phi_j^0(X_n)\defeq 1 - \phi_j^1(X_n),\]
are bounded and a.e. continuous by definition. This implies that for all $j$ and $a \in I$, $\phi_j^a$ is also $X$-a.s. continuous by property A\ref{cond::Btbdd}. Therefore all the properties of Lemma \ref{lem::cibrlem} are satisfied, so
\[(X^n,D^n_{1:k(k-1)/2}) \Rightarrow (X,D\pcar{1:k(k-1)/2}),\]
where $\{D\pcar{j}\}_{j=1}^{k(k-1)/2}$ are mutually conditionally independent given $\sigma(X) = \Gmc^{A,k}_{t+1}$ and 
\[\ex{D\pcar{j}|X} = \ex{D\pcar{j}\middle|\Gmc^{A,k}_{t+1}} = B\left(A\pcar{a_j}(t),Z\pcar{a_{j,1}}(t+1),Z\pcar{a_{j,2}}(t+1)\right).\]

By Definition \ref{defn::AZlim}(a) and Proposition \ref{prop::welldef}, it follows that
\begin{align*}
\left(Z\pcar{1:k}[t+1],A\pcar{1:k,1:k}[t],D\pcar{1:k(k-1)/2}\right)\deq \left(Z\pcar{1:k}[t+1],A\pcar{1:k,1:k}[t],A\pcar{a_{1:k(k-1)/2}}(t+1)\right),
\end{align*}
so
\begin{align*}
\left(Z^n_{1:k}[t+1],A^n_{1:k,1:k}[t],A^n_{a_{1:k(k-1)/2}}(t+1)\right) &= \left(X^n,D^n_{1:k(k-1)/2}\right) \\
&\Rightarrow \left(X,D\pcar{1:k(k-1)/2}\right) \\
&= \left(Z\pcar{1:k}[t+1],A\pcar{1:k,1:k}[t],A\pcar{a_{1:k(k-1)/2}}(t+1)\right).
\end{align*}

By the continuous mapping theorem, we may conclude, as desired,
\[\left(Z^n_{1:k}[t+1],A^n_{1:k,1:k}[t+1]\right) \Rightarrow \left(Z\pcar{1:k}[t+1],A\pcar{1:k,1:k}[t+1]\right).\]
\end{proof}

We now prove each lemma in sequence beginning with step 1 of the proof,  given by Lemma \ref{lem::cpcpf}.

\begin{proof}[Proof of Lemma \ref{lem::cpcpf}]
Let $\Xmc = \mf_t(\R^d)$ and for each $n \in \N$, let $X^n \defeq Z^n_1[t]$ be a $\Xmc$-random element. Let $\Ymc = \mf_t(\R^d)\times \{0,1\}$ and for each $n \in \N,i \in [1:n]$, let $Y^n_i = (Z^n_i[t],A^n_{1i}(t))$ be a $\Ymc$-random element. Let $X = Z\pcar{1}[t]$ and let $Y = (Z\pcar{2}[t],A\pcar{12}(t))$. We show that these random elements and spaces satisfy the conditions of Proposition \ref{prop::condconvres}.

\ind Recall that we assume Property A holds at time $t$, so $(Z^n[t],A^n[t])$ is jointly exchangeable. For each $n \in \N$ and $\sigma \in S_n$ such that $\sigma(1) = 1$, joint exchangeability of $(Z^n[t],A^n[t])$ implies:
\[(X^n,(Y^n_i)_{i=1}^n) = (Z^n_1[t],(Z^n_i[t],A^n_{1i}[t])_{i =1}^n) \deq (Z^n_1[t],(Z^n_{\sigma(i)}[t],A^n_{1\sigma(i)}[t])_{i=1}^n = (X^n,(Y^n_{\sigma(i)})_{i=1}^n),\]
so $(X^n,Y^n)$ is exchangeable excluding 1. Recall that $\Gmc^2_t = \sigma(Z\pcar{1}[t],Z\pcar{2}[t])$. By Corollary \ref{coro::morestruct}, 
\begin{equation}
\label{eqn::ABtdef}
\PP\left(A\pcar{12}(t) = 1\middle|\Gmc^2_t\right) = B_t(Z\pcar{1}[t],Z\pcar{2}[t]) \te{ for } t\geq 0.
\end{equation}

\ind Because $B$ and $B_0$ are a.e. continuous functions, $B_t$ is likewise a.e. continuous. By Property A(d) at time $t$ and Definition \ref{defn::AZlim}, $(Z\pcar{1}[t], Z\pcar{2}[t])$ is the cartesian product of two independent, absolutely continuous random vectors and is therefore absolutely continuous. Thus, $B_t$ is also $(Z\pcar{1}[t],Z\pcar{2}[t])$-a.s. continuous. Lemma \ref{lem::condmapcont} then implies that
\[\eta\defeq \law(Z\pcar{1}[t],Z\pcar{2}[t],A\pcar{12}(t)|Z\pcar{1}[t]) = \law(X,Y|X),\]
depends continuously on $X=Z\pcar{1}[t]$, so $(\mathbf{X},\mathbf{Y})$ is $\Xmc/\Ymc$-convenient.

\ind By \eqref{eqn::ZAconvtimet} and the continuous mapping theorem,
\[\left(X^n,(Y^n_j)_{j=2}^3\right) = \left(Z^n_1[t],(Z^n_{j}[t],A^n_{1j}(t))_{j=2}^3\right) \Rightarrow \left(Z\pcar{1}[t],(Z\pcar{j}[t],A\pcar{1j}(t))_{j=2}^3\right) = \left(X,(Y\pcar{j})_{j=2}^3\right).\]
Moreover, Corollary \ref{coro::morestruct} and \eqref{eqn::ABtdef} (which holds if all instances of $A\pcar{12}$, $Z\pcar{2}$ and $\Gmc^2_t$ are replaced by $A\pcar{13}$, $Z\pcar{3}$  and $\sigma(Z\pcar{1}[t],Z\pcar{3}[t])$ respectively) together imply that for $j=2,3$, $(X,Y\pcar{j}) \deq (X,Y)$. Lastly, note that $\{Z\pcar{i}[t]\}_{i=1}^3$ are i.i.d. and $A\pcar{12}(t)\indp A\pcar{13}(t)|\Gmc^3_t$ (Definition \ref{defn::AZlim}(b)). By Corollary \ref{coro::morestruct}, $\ex{A\pcar{1j}(t)\middle|\Gmc^3_t} = B_t(Z\pcar{1}[t],Z\pcar{j}[t])$. It then follows for any bounded, measurable functions $f_j: \Ymc \to \R$, $j =2,3,$
\begin{align*}
\ex{\prod_{j=2}^3 f_j(Z\pcar{j}[t],A\pcar{1j}(t))\middle|Z\pcar{1}[t]} &= \ex{\ex{\prod_{j=2}^3 f_j(Z\pcar{j}[t],A\pcar{1j}(t))\middle|\Gmc^3_t}\middle|Z\pcar{1}[t]}\\
&= \ex{\prod_{j=2}^3\ex{ f_j(Z\pcar{j}[t],A\pcar{1j}(t))\middle|\Gmc^3_t}\middle|Z\pcar{1}[t]}\\
&= \ex{\prod_{j=2}^3 \left(\sum_{a=0}^1 f_j(Z\pcar{j}[t],a)\alt{B}_t(a,Z\pcar{1}[t],Z\pcar{j}[t])\right)\middle|Z\pcar{1}[t]}\\
&= \prod_{j=2}^3\ex{ \left(\sum_{a=0}^1 f_j(Z\pcar{j}[t],a)\alt{B}_t(a,Z\pcar{1}[t],Z\pcar{j}[t])\right)\middle|Z\pcar{1}[t]}\\
&= \prod_{j=2}^3 \ex{f_j(Z\pcar{j}[t],A\pcar{1j}(t))\middle|Z\pcar{1}[t]}.
\end{align*}
Above we use the notation $\alt{B}_t$ introduced at \eqref{eqn::altdef} and the fact that $Z\pcar{1:3}[t]$ are i.i.d. and therefore $Z\pcar{2}[t]\indp Z\pcar{3}[t]|Z\pcar{1}[t]$. This proves that $Y\pcar{2}\indp Y\pcar{3}|X$. Then by Proposition \ref{prop::condconvres} and \eqref{eqn::altmu},
\begin{equation}
\label{eqn::altmuconv}
(X^n,\eta^n_{XY})\defeq (Z^n_1[t],\alt{\mu}^n_t) \Rightarrow (Z\pcar{1}[t],\alt{\mu}_t) \defeq (X,\eta_{XY}).
\end{equation}
\end{proof}

Now, we move to step 2 of the proof. To prove Lemma \ref{lem::ZALjoint}, we first prove the technical lemma.

\begin{proof}[Proof of Lemma \ref{lem::posdenom}]
Suppose Assumption \ref{assu::bds}(c)(i) holds. Then applying Definition \ref{defn::AZlim}(b), there exists a $\bar{C} > 0$ such that if $t > 0$,
\begin{align}
\ex{A\pcar{12}(t)|Z\pcar{1}[t]} &= \ex{\ex{A\pcar{12}(t)|\Gmc^{2}_t,A\pcar{12}(t-1)}\middle|Z\pcar{1}[t]}\nonumber\\
&= \ex{B(A\pcar{12}(t-1),Z\pcar{1}(t),Z\pcar{2}(t))|Z\pcar{1}[t]}\nonumber\\
&\geq \min_{a \in \{0,1\}} \ex{B(a,Z\pcar{1}(t),Z\pcar{2}(t))|Z\pcar{1}[t]}\label{eqn::exA|Zbd}\\
&\geq \ex{\bar{C}\exp\left(-C_b|Z\pcar{1}(t)-Z\pcar{2}(t)|\right)\middle| Z\pcar{1}[t]}\nonumber\\
& > 0.\nonumber
\end{align}
If $t = 0$, the same computation holds replacing $B(a,\cdot,\cdot)$ by $B_0(\cdot,\cdot)$ above. 

\ind Now suppose Assumption \ref{assu::bds}(c)(ii) holds instead. Then because $B(\cdot,z,z)>0$ and because $(a,z,z)$ is a continuity point of $B$ for all $(a,z) \in \{0,1\}\times\R^d$, there exist measurable functions $\ep,\delta: \R^d \to \R_+$ such that for all $z \in \R^d$, $\delta(z) > 0$, $\ep(z) > 0$ and for any $z' \in \R^d$ such that $|z-z'| < \ep(z)$,
\[\min_{a \in \{0,1\}} B(a,z,z') \geq \delta(z).\]
Let $Z'$ be an i.i.d. copy of $Z\pcar{1}$. By Property A(d), $Z\pcar{1}(t)$ is absolutely continuous for each $i$. Let $f_{Zt}$ be pdf of $Z\pcar{1}(t)$. Then,
\begin{align*}
\PP\left(\min_{a \in\{0,1\}} |Z\pcar{1}(t)-Z'(t)|< \ep(Z\pcar{1}(t))\middle|Z\pcar{1}[t]\right) &= \int_{\R^d}\indic{|z'-Z\pcar{1}(t)| < \ep(Z\pcar{1}(t))}f_{Zt}(z')\,dz'\\
&\geq \sup_{\ep' < \ep(Z\pcar{1}(t))} \int_{\R^d}\indic{|z'-Z\pcar{1}(t)| < \ep}f_{Zt}(z')\,dz'.
\end{align*}
The above quantity is a.s. strictly positive due to the Lebesgue differentiation theorem which states
\[\lim_{n\to\infty} \frac{1}{\ep}\int_{\R^d}\indic{|z'-Z\pcar{1}(t)| < \ep}f_{Zt}(z')\,dz' = f_{Zt}(Z\pcar{1}(t)) > 0\te{ a.s..}\]
Then by \eqref{eqn::exA|Zbd} and setting $Z\pcar{2} = Z'$,
\begin{align*}
\ex{A\pcar{12}(t)|Z\pcar{1}[t]} &\geq \min_{a \in \{0,1\}} \ex{B(a,Z\pcar{1}(t),Z'(t))|Z\pcar{1}[t]}\\
&\geq \ex{\delta(Z\pcar{1}(t))\middle|Z\pcar{1}[t]}\PP\left(\min_{a\in \{0,1\}} B(a,Z\pcar{1}(t),Z'(t))\geq \delta(Z\pcar{1}(t))\middle|Z\pcar{1}[t]\right)\\
&\geq \ex{\delta(Z\pcar{1}(t))\middle|Z\pcar{1}[t]}\PP\left(|Z\pcar{1}(t)-Z'(t)|< \ep(Z\pcar{1}(t))\middle|Z\pcar{1}[t]\right)\\
&>0 \te{ a.s..}
\end{align*}
\end{proof}

We now prove Lemma \ref{lem::ZALjoint}.
\begin{proof}
By the Skorokhod representation theorem, there exists a probability space $(\alt{\Omega},\alt{\Fmc},\alt{\PP})$ which supports the following random elements 
\begin{align*}
\left(\wh{Z}^n_{1:n}[t],\wh{A}^n_{1:n,1:n}[t]\right)&\deq\left(Z^n_{1:n}[t],A^n_{1:n,1:n}[t]\right) \te{ for all }n \in \N\\
\left(\wh{Z}\pcar{1:k}[t],\wh{A}\pcar{1:k,1:k}[t]\right) &\deq \left(Z\pcar{1:k}[t],A\pcar{1:k,1:k}[t]\right),
\end{align*}
such that
\[(\wh{X}^n,\wh{\eta}^n_{XY})\defeq (\wh{Z}^n_1[t],\wh{\alt{\mu}}^n_t) \to (\wh{Z}\pcar{1}[t],\wh{\alt{\mu}}_t)\defeq (\wh{X},\wh{\eta}_{XY})\te{ in probability,}\]
where
\[\wh{\alt{\mu}}^n_t \defeq \frac{1}{n}\sum_{i=1}^n \delta_{\wh{Z}^n_1[t],\wh{Z}^n_j[t],\wh{A}^n_{1j}(t)}\te{ and } \wh{\alt{\mu}}_t \defeq \law\left(\wh{Z}\pcar{1}[t],\wh{Z}\pcar{2}[t],\wh{A}\pcar{12}[t]\middle|\wh{Z}\pcar{1}[t]\right).\] 
Consider the continuous function $f: \mf_t(\R^d)\times \mf_t(\R^d)\times \{0,1\}\to \R^d$ given by $f(z_1,z_2,a) = az_2(t)$. Note that uniform integrability of $\{Z^n_i(t)\}_{n \in \N, i\in [1:n]}$ and the fact that $A^n_{ij}(t) \in \{0,1\}$ implies the uniform integrability of $\{f(Z^n_1(t),Z^n_j(t),A^n_{1j}(t))\}_{n \in \N, j\in [1:n]}$ and therefore $\{f(\wh{Z}^n_1(t),\wh{Z}^n_j(t),\wh{A}^n_{1j}(t))\}_{n \in \N, j\in [1:n]}$, so by Corollary \ref{unifintflem}, 
\[\langle \wh{\eta}^n_{XY},f \rangle \to \langle \wh{\eta}_{XY},f\rangle \te{ in probability.}\]
This implies
\begin{align*}
\left(Z^n_1[t],\frac{1}{n}\sum_{j=1}^n Z^n_j(t)A^n_{1j}(t)\right) &= \left(Z^n_1[t],\langle \alt{\mu}^n_t,f\rangle\right) \\
&=\left(X^n,\langle \eta^n_{XY},f\rangle\right)\\
&\deq \left(\wh{X}^n,\langle  \wh{\eta}^n_{XY},f\rangle\right)\\
& \to \left(\wh{X},\langle \wh{\eta}_{XY},f\rangle\right) \te{ in probability,} \\
&\deq \left(Z\pcar{1}[t],\ex{Z\pcar{2}[t]A\pcar{12}(t)|Z\pcar{1}[t]}\right).
\end{align*}
\sloppy Repeating the same computation replacing $f$ by the bounded, continuous function $g \in C_b\left(\mf_t(\R^d)\times \mf_t(\R^d)\times \{0,1\}\right)$ given by $g(z_1,z_2,a) = a$ yields
\[\left(Z^n_1[t],\frac{1}{n}\sum_{j=1}^nA^n_{1j}(t)\right) \Rightarrow \left(Z\pcar{1}[t],\ex{A\pcar{12}(t)|Z\pcar{1}[t]}\right).\]
By joint exchangeability, for any $i \in \N$,
\begin{align}
\left(Z^n_i[t],\frac{1}{n}\sum_{j=1}^n Z^n_j(t)A^n_{ij}(t)\right) &\Rightarrow \left(Z\pcar{i}[t],\ex{Z'(t)A'(t)\middle|Z\pcar{i}[t]}\right),\label{eqn::ZsZAconv}\\
\left(Z^n_i[t],\frac{1}{n}\sum_{j=1}^n A^n_{ij}(t)\right) &\Rightarrow \left(Z\pcar{i}[t],\ex{A'(t)\middle|Z\pcar{i}[t]}\right),\label{eqn::ZsAconv}
\end{align}
where $(Z\pcar{i}[t],Z'[t],A'(t)) \deq (Z\pcar{1}[t],Z\pcar{2}[t],A\pcar{12}(t))$. We have now shown that the marginal distributions in \eqref{eqn::ZALjoint} converge. We next apply Lemma \ref{Slutzkycond} twice to show the joint convergence in \eqref{eqn::ZALjoint}.\\

\textbf{First application of Lemma \ref{Slutzkycond}:} Fix $i \in [1:k]$. Set $k = 2$. For $m = 1,2,$ and $n \geq i$, define $X^n_1 = X^n_2 = Z^n_{i}[t]$. For $n \geq i$, let $Y^n_1 = \frac{1}{n}\sum_{j=1}^n Z^n_j(t)A^n_{ij}(t)$ and $Y^n_2 = \frac{1}{n}\sum_{j=1}^n A^n_{ij}(t)$. Lastly, let $X_1 = X_2 = Z\pcar{i}[t]$, $Y_1 = \ex{Z'(t)A'(t)\middle|Z\pcar{i}[t]}$. Likewise define $Y_2 = \ex{A'(t)\middle|Z\pcar{i}[t]}$. Then $Y_m$ is $\sigma(X_m)$ measurable for $m = 1,2$, and \eqref{eqn::ZsZAconv}-\eqref{eqn::ZsAconv} imply that $(X^n_{m},Y^n_{m})\Rightarrow (X_m,Y_{m})$ for $m = 1,2$. Then by Lemma \ref{Slutzkycond},
\begin{align*}
\left(Z^n_i[t],\frac{1}{n}\sum_{j=1}^n Z^n_j(t)A^n_{ij}(t),\frac{1}{n}\sum_{j=1}^n A^n_{ij}(t)\right)&=(X^n_{m},Y^n_{m})_{m=1,2} \\
&\Rightarrow (X_m,Y_m)_{m=1,2}\\
&\hspace{-4pt}=\left(Z\pcar{i}[t],\ex{Z\pcar{\ell}(t)A\pcar{i,\ell}(t)\middle|Z\pcar{i}[t]}, \ex{A\pcar{i,\ell}(t)\middle|Z\pcar{i}[t]}\right).
\end{align*}
By \eqref{eqn::Lidef}, the continuous mapping theorem and Lemma \ref{lem::posdenom} (which ensures the denominator of \eqref{eqn::Lidef} is a.s. positive), this implies
\begin{equation}
\label{eqn::Lniconv}
(Z^n_i[t],L^n_i(t)) \Rightarrow (Z\pcar{i}[t],L\pcar{i}(t))
\end{equation}
for all $i \in [k]$.\\

\textbf{Second Application of Lemma \ref{Slutzkycond}:} Now set $K = k+1$. For $n \in \N$ and $i\leq k$ let $X^n_i = Z^n_i[t]$ and let $X^n_{k+1} = A^n_{1:k,1:k}[t]$. For $n \in \N$ and $i \leq k$ define $Y^n_i = L^n_i(t)$ and define $Y^n_{k+1}(t) = 1$. Lastly, define $X_i = Z\pcar{i}[t]$, $Y_i = L\pcar{i}(t)$ for $i \leq k$ and $X_{k+1} = A\pcar{1:k,1:k}[t]$, $Y_{k+1} = 1$. Then by \eqref{eqn::Lniconv}, $(X^n_i,Y^n_i)\Rightarrow (X_i,Y_i)$ for $i\leq k$. By assumption, $(Z^n_{1:k}[t],A^n_{1:k,1:k}[t])\Rightarrow (Z\pcar{1:k}[t],A\pcar{1:k,1:k}[t])$ which implies $(X^n_{k+1},Y^n_{k+1}) \Rightarrow (X_{k+1},Y_{k+1})$ and $X^n_{1:k+1} \Rightarrow X_{1:k+1}$. Lastly, each $L\pcar{i}(t)$ is $\sigma(Z\pcar{i}[t])$-measurable, so for $i \leq k$, there exists a measurable function $\phi_i$ such that $\phi_i(X_i) = Y_i$. For $i = k+1$, we may simply set $\phi_{k+1} \equiv 1$ which also yields $\phi_{k+1}(X_{k+1}) = 1 = Y_{k+1}$. This completes the verification of the conditions of Lemma \ref{Slutzkycond}, so
\begin{align*}
\left(Z^n_{1:k}[t],A^n_{1:k,1:k}[t],L^n_{1:k}(t),1\right) &= (X^n_{1:k+1},Y^n_{1:k+1}) \\
&\Rightarrow (X_{1:k+1},Y_{1:k+1}) \\
&= \left(Z\pcar{1:k}[t],A\pcar{1:k}[t],L\pcar{1:k}(t),1\right).
\end{align*}
This completes the proof of the lemma.
\end{proof}

We now finish the proof by establishing Lemma \ref{lem::Zt+Atjoint}.

\begin{proof}[Proof of Lemma \ref{lem::Zt+Atjoint}:]
By Assumption \ref{assu::cond}(a), $\xi_{1:k}(t)$ are i.i.d. and independent of $\vee_{n \in \N} \Fmc^n_t\vee \Fmc^k_t$. Therefore, Lemma \ref{lem::ZALjoint} implies that
\[\left(Z^n_{1:k}[t],L^n_{1:k}(t),A^n_{1:k,1:k}[t],\xi_{1:k}(t)\right) \Rightarrow \left(Z\pcar{1:k}[t],L\pcar{1:k}(t),A\pcar{1:k,1:k}[t],\xi_{1:k}(t)\right).\] By \eqref{eqn::Zevolve} and the continuous mapping theorem, this directly implies, as desired,
\[\left(Z^n_{1:k}[t+1],A^n_{1:k,1:k}[t]\right) \Rightarrow \left(Z\pcar{1:k}[t+1],A\pcar{1:k,1:k}[t]\right)\]
\end{proof}

\subsection{Proof of Theorem \ref{thm::hydro}}
\label{ssec::hydropf}

Suppose Assumptions \ref{assu::cond}-\ref{assu::bds} hold. By Lemma \ref{lem::t0} and Proposition \ref{prop::genstat}, property A holds at all times $s \in [t]$. By property A(a), $(Z^n[t],A^n[t])$ is jointly exchangeable for all $n \in \N$, which implies that $(Z^n_i[t])_{i\in [n]}$ is an exchangeable collection of random vectors. Furthermore, by property A(c),
\[\left(Z^n_{1:2}[t],A^n_{1:2}[t]\right) \Rightarrow \left(Z\pcar{1:2}[t],Z\pcar{1:2}[t]\right),\]
so 
\[Z^n_{1:2}[t]\Rightarrow Z\pcar{1:2}[t],\]
where by Definition \ref{defn::AZlim}, $Z\pcar{1}[t]$ and $Z\pcar{2}[t]$ are i.i.d.. Then by Proposition \ref{prop::convres},
\[\mu^n_t = \frac{1}{n}\sum_{i=1}^n \delta_{Z^n_i[t]} \Rightarrow \te{Law}\left(Z\pcar{1}[t]\right) = \mu_t,\]
so \eqref{eqn::muntdef} holds.

\ind Fix $\Xmc\defeq \mf_t(\R^d)$ and $\Ymc\defeq \mf_t(\R^d\times\{0,1\})$. For all $n \in \N$, define $X^n = Z^n_1[t]$ and for all $j \in [n]$ define $Y^n_j = (Z^n_j[t],A^n_{1j}[t])$. By property A(a) at time $t$, $(Z^n_{1:n}[t],A^n_{1:n,1:n}[t])$ is jointly exchangeable. Then for any $\sigma \in S_n$ such that $\sigma(1) = 1$,
\[\left(X^n,(Y^n_j)_{j=1}^n\right) = \left(Z^n_1[t],\left(Z^n_j[t],A^n_{1j}[t]\right)_{j=1}^n\right) \deq \left(Z^n_1[t],\left(Z^n_{\sigma(j)}[t],A^n_{1\sigma(j)}[t]\right)_{j=1}^n\right) = \left(X^n,(Y^n_{\sigma(j)})_{j=1}^n\right),\]
so $(X^n,Y^n)$ is exchangeable excluding 1.

\ind Next, let $X = Z\pcar{1}[t]$ and for each $j$ let $Y\pcar{j} = (Z\pcar{j}[t],A\pcar{1j}[t])$. Then for $j > 1$, $\law(X,Y\pcar{j})\defeq \eta$ does not depend on $j$. Letting $Y = Y\pcar{2}$, define $\eta_{XY} = \law(X,Y|X)$. To show that $\eta_{XY}$ depends continuously on $X$, it suffices to show that for any $f \in C_b\left(\Xmc\times \Ymc\right)$, there exists a continuous $\phi_f: \Xmc \to \R$ satisfying $\phi_f(X) = \langle \eta_{XY},f\rangle$ a.s.. To show this, we first note that for any $\bm{a} \in \mf_t(\{0,1\})$, Definition \ref{defn::AZlim}(b) implies
\begin{align*}
\phi_{2,\bm{a}}(Z\pcar{1:2}[t]) &\defeq \PP\left(A\pcar{12}[t] = \bm{a}\middle|\Gmc^2_t\right)\\
& = \alt{B}_0\left(a(0),Z\pcar{1}(0),Z\pcar{2}(0)\right)\prod_{s=1}^t \alt{B}\left(a(s),a(s-1),Z\pcar{1}(s),Z\pcar{2}(s)\right),
\end{align*}
where $\alt{B}$ is defined via \eqref{eqn::altdef}. 
So, we get that  $\phi_{2,\bm{a}}$ is a bounded and a.s. continuous function of $Z\pcar{1:2}[t]$. Then,

\begin{align*}
\phi_f(X) &= \ex{f(X,Y)|X} \\
&= \ex{f\left(Z\pcar{1}[t],Z\pcar{2}[t],A\pcar{12}[t]\right)\middle|Z\pcar{1}[t]}\\
&= \ex{\ex{f\left(Z\pcar{1}[t],Z\pcar{2}[t],A\pcar{12}[t]\right)\middle|\Gmc^2_t}\middle|Z\pcar{1}[t]}\\
&= \ex{\sum_{\bm{a} \in \mf_t(\{0,1\})} f\left(Z\pcar{1}[t],Z\pcar{2}[t],\bm{a}\right)\PP\left(A\pcar{12}[t] = \bm{a}\middle|\Gmc^2_t\right)\middle|Z\pcar{1}[t]}\\
&= \ex{\sum_{\bm{a} \in \mf_t(\{0,1\})} f\left(Z\pcar{1}[t],Z\pcar{2}[t],\bm{a}\right)\phi_{2,\bm{a}}(Z\pcar{1}[t],Z\pcar{2}[t])\middle|Z\pcar{1}[t]}.
\end{align*}
Because $Z\pcar{1}[t]$ and $Z\pcar{2}[t]$ are independent, it follows that for any $z \in \mf_t(\R^d)$,
\begin{align*}
\phi_f(z) &= \ex{\sum_{\bm{a} \in \mf_t(\{0,1\})} f\left(z,Z\pcar{2}[t],\bm{a}\right)\phi_{2,\bm{a}}(z,Z\pcar{2}[t])}.
\end{align*}
Since the term inside the expectation is bounded and a.s. continuous for $Z\pcar{1}[t]$ a.s. values of $z$, it follows that $\phi_f$ is a.s. continuous, so $\eta_{XY}$ depends continuously on $X$. This establishes that $(\bm{X},\bm{Y})$ is $\Xmc/\Ymc$-convenient. Furthermore, by property A(c),
\[(X^n,Y^n_{2:3}) = \left(Z^n_{1:3}[t],A^n_{1,2:3}[t]\right) \Rightarrow \left(Z\pcar{1:3}[t],A\pcar{1,2:3}[t]\right) = (X,Y\pcar{2:3}),\]
where $\law(X,Y\pcar{j}|X) = \eta_{XY}$ for $j=2,3$. Lastly, given $f_2,f_3: \mf_t\left(\R^d\times\{0,1\}\right)$, Definition \ref{defn::AZlim}(b) implies
\begin{align*}
\ex{\prod_{j=2}^3 f_j\left(Y\pcar{j}\right)\middle|X} &= \ex{\prod_{j=2}^3 f_j\left(Z\pcar{j}[t],A\pcar{1j}[t]\right)\middle|Z\pcar{1}[t]}\\
&= \ex{\ex{\prod_{j=2}^3 f_j\left(Z\pcar{j}[t],A\pcar{1j}[t]\right)\middle|\Gmc^3_t}\middle|Z\pcar{1}[t]}\\
&= \ex{\prod_{j=2}^3\ex{ f_j\left(Z\pcar{j}[t],A\pcar{1j}[t]\right)\middle|\Gmc^3_t}\middle|Z\pcar{1}[t]}\\
&= \ex{\prod_{j=2}^3\ex{f_j\left(Z\pcar{j}[t],A\pcar{1j}[t]\right)\middle|Z\pcar{1}[t],Z\pcar{j}[t]}\middle|Z\pcar{1}[t]}\\
&= \prod_{j=2}^3\ex{\ex{f_j\left(Z\pcar{j}[t],A\pcar{1j}[t]\right)\middle|Z\pcar{1}[t],Z\pcar{j}[t]}\middle|Z\pcar{1}[t]}\\
&= \prod_{j=2}^3\ex{f_j\left(Z\pcar{j}[t],A\pcar{1j}[t]\right)Z\pcar{1}[t]}\\
&=\prod_{j=2}^3\ex{ f_j\left(Y\pcar{j}\right)\middle|X},
\end{align*}
where the fourth equality stems from the independence of $(Z\pcar{1}[t],Z\pcar{j}[t],A\pcar{1j}[t])$ and $Z\pcar{5-j}[t]$ for $j = 2,3$.\footnote{Indeed, note that for a measurable function $f$, we have that
\begin{align*}
\ex{f(Z\pcar{1}[t],Z\pcar{j}[t],A\pcar{1j}[t])\middle|Z\pcar{5-j}[t]} &= \ex{\ex{f(Z\pcar{1}[t],Z\pcar{j}[t],A\pcar{1j}[t])\middle|\Gmc^3_t}\middle|Z\pcar{5-j}[t]}\\
&= \ex{\phi(Z\pcar{1}[t],Z\pcar{j}[t])\middle|Z\pcar{5-j}[t]}\\
&= \ex{\phi(Z\pcar{1}[t],Z\pcar{j}[t])},
\end{align*}
where 
\[\phi(z_1,z_2) = \sum_{\bm{a}\in \mf_t(\{0,1\})} f(z_1,z_2,\bm{a})\alt{B}_0(a(0),z_1(0),z_2(0))\prod_{s=1}^t \alt{B}(a(s),a(s-1),z_1(s),z_2(s)).\]} This implies that $Y\pcar{2}$ and $Y\pcar{3}$ are independent given $X$. The result then follows from Proposition \ref{prop::condconvres}.

\section{Proof of Theorem \ref{thm::Antmultconv} and Corollary \ref{coro::Wdefthm}}
\label{sec::graphon}

\subsection{Left-Convergence of Multiplexons}
\label{ssec::mgraphonres}
Left-convergence of decorated graphs was first studied by \cite{LovSze10}, and recently a class of probability graphons was developed which captures many convenient properties of graphons in the more general decorated graph case \cite{AbrDelWei23}. A more specialized framework for multiplexes was introduced by \cite{BhaGan25}. In this section, we present a few useful results from \cite{BhaGan25} which we then apply to the proof of Theorem \ref{thm::Antmultconv} in Section \ref{ssec::Antmultconvpfpf}. In particular, the results from \cite{BhaGan25} show that to prove the convergence of the graph trajectories $\bm{G}_n$, it suffices to establish the convergence of homomorphism densities of $\bm{G}_n$.

\begin{definition}[Multiplex Homomorphism Densities]
\label{defn::multhomdens}
\cite[Definition 3.1]{BhaGan25} If $\mathbf{H}$ and $\mathbf{G}$ are two $(t+1)$-layer multiplexes,
\[\te{Hom}(\mathbf{H},\mathbf{G}) = \bigcap_{s=0}^t \te{Hom}(H(s),G(s)),\]
and
\[t(\mathbf{H},\mathbf{G}) = \frac{\left|\te{Hom}(\mathbf{H},\mathbf{G})\right|}{|V_G|^{|V_H|}}.\]
\end{definition}

We now introduce a different multiplex decomposition to that used in Definition \ref{defn::multdecomp}.
\begin{definition}[Disjoint Decomposition]
\label{defn::multcumuldecomp}
Let $\mathbf{H}$ be a $t+1$ layer multiplex. Then set $H\pcar{S}= (V_H,E\pcar{S}_H)$, where
\[E\pcar{S}_H = \left(\cap_{s \in S} E_H(s)\right)\setminus \left(\cup_{s\notin S} E_H(s)\right),\]
so $E\pcar{S}_H$ contains all edges that lie in $E_H(s)$ if and only if $s \in S$. Let $E_H = \cup_{s \in S}E_H(s)$ be the set of all edges that lie in \emph{any} layer of $\bm{H}$.
\end{definition}

Using this decomposition, we can now extend the definition of a homomorphism to multiplexons:

\begin{definition}[Multiplexon Homomorphism Densities]
\label{multhomdensgraphon}
Let $\mathbf{H}$ be a $(t+1)$-layer multiplex with vertex set $V_H = [k]$ and let $\mathbf{W}$ be a $(t+1)$-layer multiplexon. Then
\[t(\mathbf{H},\mathbf{W}) = \int_{[0,1]^k} \prod_{\substack{S\subseteq [t]\\ S\neq \emptyset}} \prod_{\{i,j\} \in E_H\pcar{S}} W^S(x_i,x_j)\,dx_1,\dots,dx_k.\]
\end{definition}

We apply the following useful results which follow from \cite{AbrDelWei23}:

\begin{proposition}[Homomorphism Density Equivalence]
\label{multhomequiv}
For any $t+1$-multiplexes $\mathbf{H}$ and $\mathbf{G}$,
\[t(\mathbf{H},\mathbf{G}) = t\left(\mathbf{H},\mathbf{W}^{\mathbf{G}}\right).\]
For any $(t+1)$-layer multiplexons $\mathbf{W}_1 \sim \mathbf{W}_2\sim \alt{\mathbf{W}}$,
\[t\left(\mathbf{H},\alt{\mathbf{W}}\right) \defeq t\left(\mathbf{H},\mathbf{W}_1\right) = t\left(\mathbf{H},\mathbf{W}_2\right).\]
\end{proposition}
\begin{proof}
This is a restatement of Proposition 3.13 and Corollary 6.3 of \cite{BhaGan25}.
\end{proof}

\begin{proposition}
\label{tmultcont}
A sequence of $t+1$-layer multiplexon classes $\alt{\mathbf{W}}_1,\alt{\mathbf{W}}_2,\dots$ converges to a multiplexon class $\alt{\mathbf{W}}$ if and only if
\[\lim_{n\to\infty} t\left(\mathbf{H},\alt{\mathbf{W}}_n\right) = t\left(\mathbf{H},\alt{\mathbf{W}}\right).\]
\end{proposition}
\begin{proof}
This is a direct consequence of Theorem 7.1 of \cite{BhaGan25}.
\end{proof}

\subsection{Proof of Theorem \ref{thm::Antmultconv} and Corollary \ref{coro::Wdefthm}}
\label{ssec::Antmultconvpfpf}

We apply the propositions above to prove Theorem \ref{thm::Antmultconv} and Corollary \ref{coro::Wdefthm} together.

\begin{proof}[Proof of Theorem \ref{thm::Antmultconv} and Corollary \ref{coro::Wdefthm}]
Fix $t > 0$ and choose an arbitrary $\lambda$-$\mu_t$ measure-preserving transformation $\theta_t$. Let $\mathbf{H}$ be any multiplex. Assume without loss of generality that the vertex set of $\mathbf{H}$ is $V_H = [1:k]$ for some $k \in \N$. Let $\phi: [1:k] \to [1:n]$ be a map from the vertex set of $\mathbf{H}$ to the vertex set of $\bm{G}_n$. Then $\phi$ is a homomorphism if for every non-empty $S \subseteq [t]$ and $\{i,j\} \in E_H\pcar{S}$ (recall Definition \ref{defn::multcumuldecomp}), $\{\phi(i),\phi(j)\}$ also lies in $E_{\bm{G}_n}^S$, or equivalently,
\[\prod_{s \in S} A^n_{\phi(i)\phi(j)}(s) = 1.\]
Then,
\begin{align*}
t(\mathbf{H},\bm{G}_n) &= \frac{1}{n^k} \sum_{\phi: [1:k]\to [1:n]} \indic{\phi\in \te{Hom}(\mathbf{H},\bm{G}_n)}\\
&= \frac{1}{n^k} \sum_{\phi: [1:k]\to [1:n]} \prod_{\substack{S\subseteq [t]\\S\neq \emptyset}} \prod_{\{i,j\}\in E_H\pcar{S}} \prod_{s \in S} A^n_{\phi(i)\phi(j)}(s)\\
&= \frac{1}{n^k}\sum_{\mathbf{m}\in [1:n]^k} \prod_{\substack{S\subseteq [t]\\S\neq \emptyset}} \prod_{\{i,j\} \in E_H\pcar{S}} \prod_{s \in S} A^n_{m_i,m_j}(s).
\end{align*}

This allows us to make the following computation. In the third equality below, we apply the joint exchangeability of $(Z^n[t],A^n[t])$. The fourth equality applies the fact that $R_n$ is $o(1)$ by an argument we provide below in \eqref{eqn::RtHGn}. We also apply Proposition \ref{prop::genstat} and Lemma \ref{lem::t0} which together imply that property A holds at all times. By property A(c) at time $t$, $(A^n_{ij}[t])_{(i,j)\in \incs_k} \Rightarrow (A\pcar{ij}[t])_{(i,j)\in \incs_k}$. The sixth equality follows by Definition \ref{defn::AZlim}(b). The seventh equality follows by \eqref{eq::Bzz0def}, \eqref{eq::Bzzdef}, Remark \ref{rem::Bzzeqdist} and the fact that $\ex{\prod_{i\in S}A\pcar{ij}(s)|\Gmc^k_t}$ is $\sigma(Z\pcar{i}[t],Z\pcar{j}[t])$-measurable for any $S\subseteq [t]$. The penultimate equality follows from the definition of $\mathbf{W}$ given in \eqref{eqn::Wmultdef}. The final equality is simply the definition of multiple homomorphism density (Definition \ref{defn::multhomdens}).
\begin{align*}
\lim_{n\to\infty} \ex{t(\mathbf{H},\bm{G}_n)} &= \lim_{n\to\infty}\ex{\frac{1}{n^k}\sum_{\mathbf{m} \in [1:n]^k} \prod_{\substack{S\subseteq [t]\\S\neq\emptyset}}\prod_{\{i,j\}\in E_H\pcar{S}}\prod_{s \in S} A^n_{m_i,m_j}(s)}\\
&= \lim_{n\to\infty}\left(\ex{\frac{(n-k)!}{n!} \sum_{\substack{\mathbf{m} \in [1:n]^k\\ (m_1,\dots,m_k)\te{ distinct.}}} \prod_{\substack{S\subseteq [t]\\S\neq\emptyset}}\prod_{\{i,j\}\in E_H\pcar{S}}\prod_{s \in S} A^n_{m_i,m_j}(s)} + R_n\right)\\
&= \lim_{n\to\infty}\ex{\prod_{\substack{S\subseteq [t]\\S\neq\emptyset}}\prod_{\{i,j\}\in E_H\pcar{S}}\prod_{s \in S} A^n_{i,j}(s)} + \lim_{n\to\infty}R_n\\
&= \ex{\prod_{\substack{S\subseteq [t]\\S\neq\emptyset}}\prod_{\{i,j\}\in E_H\pcar{S}}\prod_{s \in S}A\pcar{ij}(s)}\\
&= \ex{\ex{\prod_{\substack{S\subseteq [t]\\S\neq\emptyset}}\prod_{\{i,j\}\in E_H\pcar{S}}\prod_{s \in S}A\pcar{ij}(s)\middle|\Gmc^k_t}}\\
&= \ex{\prod_{\substack{S\subseteq [t]\\S\neq\emptyset}}\prod_{\{i,j\}\in E_H\pcar{S}}\ex{\prod_{s \in S}A\pcar{ij}(s)\middle|\Gmc^k_t}}\\
&= \ex{\prod_{\substack{S\subseteq [t]\\S\neq\emptyset}}\prod_{\{i,j\}\in E_H\pcar{S}}\ex{\prod_{s \in S}B_{Z\pcar{i}[t],Z\pcar{j}[t]}(s)\middle|Z\pcar{i}[t],Z\pcar{j}[t]}}\\
&= \ex{\prod_{\substack{S\subseteq [t]\\S\neq\emptyset}}\prod_{\{i,j\}\in E_H\pcar{S}}\ex{\prod_{s \in S}B_{\theta_t(U_i),\theta_t(U_j)}(s)\middle|\theta_t(U_i),\theta_t(U_j)}}\\
&= \ex{\prod_{\substack{S\subseteq [t]\\S\neq\emptyset}}\prod_{\{i,j\}\in E_H\pcar{S}}W^S(U_i,U_j)}\\
&= t\left(\mathbf{H},\mathbf{W}\right),
\end{align*}

Above, $|R_n| \to 0$ by the following combinatorial argument:

\begin{align}
|R_n| &= \Bigg|\frac{1}{n^k}\sum_{\substack{\mathbf{m} \in [1:n]^k\\ (m_1,\dots,m_k)\te{ not distinct.}}} \prod_{\substack{S\subseteq [t]\\ S\neq \emptyset}}\prod_{\{i,j\} \in E_H\pcar{S}}\prod_{s\in S} A^n_{m_i,m_j}(s) + \nonumber\\
&\ind\ind \left(\frac{1}{n^k}-\frac{(n-k)!}{n!}\right) \sum_{\substack{\mathbf{m} \in [1:n]^k\\ (m_1,\dots,m_k)\te{ distinct.}}}\prod_{\{i,j\} \in E_H\pcar{S}}\prod_{s\in S} A^n_{m_i,m_j}(s)\Bigg|\nonumber\\
&\leq \left|\frac{1}{n^k}\left(n^k - \frac{n!}{(n-k)!}\right)\right| + \left|\left(\frac{1}{n^k} - \frac{(n-k)!}{n!}\right)\frac{n!}{(n-k)!}\right|\nonumber\\
&= 2\left|1 - \frac{n!}{n^k(n-k)!}\right|\nonumber\\
& \to 0.\label{eqn::RtHGn}
\end{align}

Using the same arguments, we also get the second-order condition. In this case, the fact that $R'_n = o(1)$ follows by a combinatorial argument we describe in \eqref{eqn::RtsqHGn} below. We additionally apply the fact that
\[\left(\prod_{\substack{S\subseteq [t]\\S \neq \emptyset}} \prod_{\{i,j\} \in E\pcar{S}_H} \mathbf{W}^S(U_i,U_j)\right) \indp \left(\prod_{\substack{S\subseteq [t]\\S \neq \emptyset}} \prod_{\{i,j\} \in E\pcar{S}_H} \mathbf{W}^S(U_{k+i},U_{k+i})\right).\]

\begin{align*}
&\lim_{n\to\infty} \ex{(t(\mathbf{H},\bm{G}_n))^2} \\
&= \lim_{n\to\infty}\ex{\left(\frac{1}{n^k}\sum_{\mathbf{m} \in [1:n]^k} \prod_{\substack{S\subseteq [t]\\S\neq\emptyset}}\prod_{\{i,j\} \in E_H\pcar{S}} \prod_{s \in S}A^n_{m_i,m_j}(s)\right)^2}\\
&= \lim_{n\to\infty}\ex{\frac{1}{n^{2k}}\sum_{\mathbf{m},\mathbf{m}' \in [1:n]^k} \prod_{\substack{S\subseteq [t]\\S\neq\emptyset}}\prod_{\{i,j\} \in E_H\pcar{S}} \prod_{s \in S} A^n_{m_i,m_j}(s)A^n_{m'_i,m'_j}(s)}\\
&= \lim_{n\to\infty}\left(\ex{\frac{(n-2k)!}{n!}\sum_{\substack{\mathbf{m},\mathbf{m}' \in [1:n]^k\\ (m_1,\dots,m_k,m'_1,\dots,m'_k)\te{ distinct.}}} \prod_{\substack{S\subseteq [t]\\S\neq\emptyset}}\prod_{\{i,j\} \in E_H\pcar{S}} \prod_{s \in S} A^n_{m_i,m_j}(s)A^n_{m'_i,m'_j}(s)} + R'_n\right)\\
&= \lim_{n\to\infty}\ex{\prod_{\substack{S\subseteq [t]\\S\neq\emptyset}}\prod_{\{i,j\} \in E_H\pcar{S}} \prod_{s \in S} A^n_{i,j}(s)A^n_{k+i,k+j}(s)} + \lim_{n\to\infty} R'_n\\
&= \ex{\prod_{\substack{S\subseteq [t]\\S\neq\emptyset}}\prod_{\{i,j\} \in E_H\pcar{S}} \prod_{s \in S} A\pcar{ij}(s)A\pcar{k+i,k+j}(s)}\\
&= \ex{\ex{\prod_{\substack{S\subseteq [t]\\S\neq\emptyset}}\prod_{\{i,j\} \in E_H\pcar{S}} \prod_{s \in S} A\pcar{ij}(s)A\pcar{k+i,k+j}(s)\middle|\Gmc^{2k}_t}}\\
&= \ex{\prod_{\substack{S\subseteq [t]\\S\neq\emptyset}}\prod_{\{i,j\} \in E_H\pcar{S}}\ex{ \prod_{s \in S} A\pcar{ij}(s)\middle|\Gmc^{2k}_t}\ex{\prod_{s \in S}A\pcar{k+i,k+j}(s)\middle|\Gmc^{2k}_t}}\\
&= \ex{\prod_{\substack{S\subseteq [t]\\S\neq\emptyset}}\prod_{\{i,j\} \in E_H\pcar{S}}\ex{ \prod_{s \in S} A\pcar{ij}(s)\middle|\Gmc^{2k}_t}\ex{\prod_{s \in S}A\pcar{k+i,k+j}(s)\middle|\Gmc^{2k}_t}}\\
&= \ex{\prod_{\substack{S\subseteq [t]\\S\neq\emptyset}}\prod_{\{i,j\} \in E_H\pcar{S}}\ex{ \prod_{s \in S} B_{Z\pcar{i}[t],Z\pcar{j}[t]}(s)\middle|Z\pcar{i}[t],Z\pcar{j}[t]}\ex{\prod_{s \in S}B_{Z\pcar{k+i}[t],Z\pcar{k+j}[t]}(s)\middle|Z\pcar{k+i}[t],Z\pcar{k+j}[t]}}\\
&= \ex{\prod_{\substack{S\subseteq [t]\\S\neq\emptyset}}\prod_{\{i,j\} \in E_H\pcar{S}}\ex{ \prod_{s \in S} B_{\theta_t(U_i),\theta_t(U_j)}(s)\middle|\theta_t(U_i),\theta_t(U_j)}\ex{\prod_{s \in S}B_{\theta_t(U_{k+i}),\theta_t(U_{k+j})}(s)\middle|\theta_t(U_{k+i}),\theta_t(U_{k+j})}}\\
&= \ex{\prod_{\substack{S\subseteq [t]\\S\neq\emptyset}}\prod_{\{i,j\} \in E_H\pcar{S}}\mathbf{W}^S(U_i,U_j)\mathbf{W}^S(U_{k+i},U_{k+j})}\\
&= \ex{\prod_{\substack{S\subseteq [t]\\S\neq\emptyset}}\prod_{\{i,j\} \in E_H\pcar{S}}\mathbf{W}^S(U_i,U_j)}\ex{\prod_{\substack{S\subseteq [t]\\S\neq\emptyset}}\prod_{\{i,j\} \in E_H\pcar{S}}\mathbf{W}^S(U_{k+i},U_{k+j})}\\
&= \left(t\left(\mathbf{H},\mathbf{W}\right)\right)^2.
\end{align*}

Above, $|R'_n| \to 0$ by the following combinatorial argument:

\begin{align}
|R'_n| &= \left|\frac{1}{n^{2k}}\sum_{\substack{\mathbf{m},\mathbf{m}'\in [n]^k\\ (m_1,\dots,m_k,m'_1,\dots,m'_k)\te{ not distinct.}}} \prod_{\substack{S\subseteq [t]\\S\neq\emptyset}}\prod_{\{i,j\} \in E_H\pcar{S}}\prod_{s\in S} A^n_{m_i,m_j}(s)A^n_{m'_i,m'_j}(s)\right|\nonumber\\
 &\ind + \left|\left(\frac{1}{n^{2k}} - \frac{(n-2k)!}{n!}\right)\sum_{\substack{\mathbf{m},\mathbf{m}'\in [n]^k\\ (m_1,\dots,m_k,m'_1,\dots,m'_k)\te{ distinct.}}} \prod_{\substack{S\subseteq [t]\\S\neq\emptyset}}\prod_{\{i,j\} \in E_H\pcar{S}}\prod_{s\in S} A^n_{m_i,m_j}(s)A^n_{m'_i,m'_j}(s)\right|\nonumber\\
 &\leq \left|\frac{1}{n^{2k}}\left(n^{2k} - \frac{n!}{(n-2k)!}\right)\right| + \left|\left(\frac{1}{n^{2k}} - \frac{(n-2k)!}{n!}\right)\frac{n!}{(n-2k)!}\right|\nonumber\\
 &= 2\left|1 - \frac{n!}{n^{2k}(n-2k)!}\right|\nonumber\\
 &\to 0.\label{eqn::RtsqHGn}
\end{align}

It follows that $t(\bm{H},\bm{G}_n[t]) \to t(\bm{H},\bm{W})$ in probability, so by Proposition \ref{multhomdensgraphon}, 
\[t(\bm{H},\alt{\bm{G}_n[t]}) \to t(\bm{H},\alt{\bm{W})} \te{ in probability}.\]
Because $\theta_t$ was chosen arbitrarily, it follows from Proposition \ref{tmultcont} that the choice of $\theta_t$ does not affect the isomorphism class $\alt{\bm{W}}$ of $\bm{W}$.
\end{proof}

\appendix

\section{Useful Lemmas}
\label{sec::useful}

This appendix contains several minor lemmas that we used to prove our results throughout the article. It is split into lemmas to establish continuity, convergence and other miscellaneous lemmas.

\subsection{Establishing Continuous Dependence}
\label{ssec::contdep}

This section includes some lemmas which allow us to establish continuous dependence (Definition \ref{def::depcont}). This is very useful in applications of Proposition \ref{prop::condconvres}.

\begin{lemma}
\label{lem::condmapcont}
Let $X$ and $Y$ be independent $\Xmc$ and $\Ymc$-random elements respectively, and let $I$ be a finite index set. Suppose there exist $|I|$ $(X,Y)$-a.s. continuous functions $f_i: \Xmc\times \Ymc\to [0,1]$, $i \in I$ and a $I$-valued random variable $D$ such that $\P(D = i|X,Y) = f_i(X,Y)$ for each $i \in I$. Then, the following random elements depend continuously on $X$:
\begin{enumerate}[label = (\alph*)]
\item $\eta_{XYD}\defeq \law(X,Y,D|X)$.
\item $\ex{g(X,Y,D)|X}$ where $g \in C_b\left(\Xmc\times\Ymc\times I\right)$.
\item $\ex{g(X,Y,D)|X}$ where $g:\Xmc\times\Ymc\times I\to \R$ is continuous and there exists an $h:\Ymc\to\R_+$ such that $\sup_{x,d}|g(x,Y,d)|\leq h(Y)$ a.s. and $\ex{h(Y)}<\infty$.
\end{enumerate}
\end{lemma}
\begin{proof}
For each $x \in \Xmc$ let $D_x$ be a $I$-valued random variable such that $\PP(D_x = i|Y) =  f_i(x,Y))$ for each $i \in I$. Then define $\phi: \Xmc\to \P(\Xmc\times \Ymc\times I)$ by
\[\phi(x) = \law(x,Y,D_x) = \delta_x\times \law(Y,D_x).\]
Fix any $g:\Xmc\times\Ymc\times I\to\R$ satisfying the conditions of part (c) of the lemma. Then, for any sequence $x_n \to x$, the Lebesgue-dominated convergence theorem and a.s. continuity of $f$ and $g$ imply
\begin{align}
\lim_{n\to\infty} \langle \phi(x_n),g\rangle &= \lim_{n\to\infty}\ex{g(x_n,Y,D_{x_n})}\nonumber\\
&=\lim_{n\to\infty}\ex{\ex{g(x_n,Y,D_{x_n})\middle|Y}}\nonumber\\
&= \lim_{n\to\infty}\ex{\sum_{i \in I}g(x_n,Y,i)f_i(x_n,Y)}\nonumber\\
&= \ex{\lim_{n\to\infty}\sum_{i \in I}g(x_n,Y,i)f_i(x_n,Y)}\nonumber\\
&= \ex{\sum_{i \in I}g(x,Y,i)f_i(x,Y)}\nonumber\\
&= \ex{g(x,Y,D_x)}\nonumber\\
&=\langle \phi(x),g\rangle.\label{phixngconv}
\end{align}
\ind This proves that $x \mapsto \langle\phi(x),g\rangle$ is continuous for all $g$ satisfying the conditions of part (c). By definition of $\phi$, $\phi(X) = \eta_{XYD}$ a.s.. Because all functions $g \in C_b(\Xmc\times\Ymc\times I)$ satisfy the conditions of part (c), it follows that $\phi$ is continuous so $\eta_{XYD}$ depends continuously on $X$. Moreover, $\langle \phi(X),g\rangle = \langle \eta_{XYD},g\rangle = \ex{g(X,Y,D)|X}$ a.s., so $\ex{g(X,Y,D)|X}$ depends continuously on $X$. That concludes the proofs of parts (a) and (c). Part (b) is a special case of (c).
\end{proof}

\subsection{Convergence Lemmas}
\subsubsection{Conditional Slutzky's Lemma}

The next lemma is a kind of conditional Slutzky's lemma. Slutzky's theorem states that the joint distribution of $(X^n,Y^n)$ converges if the marginals converge and $Y^n \Rightarrow c$, where $c$ is deterministic. We extend this to a similar result except that we now assume the \emph{conditional} distribution of $Y^n$ given $X^n$ approaches a random Dirac delta measure in some sense.

\begin{lemma}[Conditional Slutzky's Lemma]
\label{Slutzkycond}
Suppose that for some $k \in \N$ and all $i \in [1:k]$, a sequence of $\Xmc\times \Ymc$ random elements $(X_i^n,Y^n_i)$ converge weakly to some $(X_i,Y_i)$. Suppose also that for each $i$ there exists a measurable function $\phi_i: \Xmc\to\Ymc$ such that $Y_i = \phi_i(X_i)$ a.s.. If $(X^n_i)_{i\in [1:k]} \Rightarrow (X_i)_{i\in [1:k]}$, then
\[(X^n_i,Y^n_i)_{i\in [1:k]} \Rightarrow (X_i,\phi_i(X_i))_{i\in [1:k]}.\]
\end{lemma}
\begin{proof}
Because its marginals are weakly convergent, the sequence $\{(X^n_i,Y^n_i)_{i\in[1:k]}\}_{n\in\N}$ is tight, so there exists a subsequence $\{n_{\ell}\}_{\ell\in\N}$ such that
\[(X^{n_{\ell}}_i,Y^{n_{\ell}}_i)_{i\in [1:k]} \Rightarrow (\alt{X}_i,\alt{Y}_i)_{i\in [1:k]}.\]
For each $i$, $(X^{n_{\ell}}_i,Y^{n_{\ell}}_i) \Rightarrow (X_i,Y_i)$, so $(\alt{X}_i,\alt{Y}_i) \deq (X_i,Y_i)$ and $\alt{Y}_i = \phi_i(\alt{X}_i)$ a.s.. Lastly, since $(X^{n_{\ell}}_i)_{i \in [1:k]} \rightarrow (X_i)_{i\in [1:k]}$ it follows that $(X_i)_{i \in [1:k]}\deq (\alt{X}_i)_{i\in[1:k]}$. Thus, 
\[(\alt{X}_i,\alt{Y}_i)_{i \in [1:k]} \deq (X_i,\phi_i(X_i))_{i \in [1:k]}.\]
Since this limiting quantity is unique in distribution, it follows, as desired, that
\[(X^n_i,Y^n_i)_{i \in [1:k]} \Rightarrow (\alt{X}_i,\alt{Y}_i)_{i \in [1:k]} = (\alt{X}_i,\phi_i(\alt{X}_i))_{i \in [1:k]} \deq (X_i,\phi_i(X_i))_{i\in [1:k]}.\]
\end{proof}

\subsubsection{Convergence of Conditionally Independent Bernoulli Random Variables}
\label{cibr}

\begin{lemma}
\label{lem::cibrlem}
Let $X$ be any $\Xmc$-random element, and let $I$ be a finite index set. Fix $k \in \N$, let $\{\phi_j^a\}_{j\in [1:k],a\in I}$ be a sequence of bounded, $X$-a.s. continuous functions from $\Xmc$ to $[0,1]$. In addition, let $(X^n,D^n)\defeq (X^n,D^n_{[1:k]})$, $n \in \N$ be a sequence of $\Xmc\times  I^k$- random elements satisfying the following conditions:
\begin{enumerate}[label = (\alph*)]
\item $X^n \Rightarrow X$ for some $\Xmc$-random element $X$;
\item for each $n\in \N$, $D^n_1,\dots,D^n_k$ are mutually conditionally independent given $X^n$;
\item for all $n \in \N$, $j \in [1:k]$, and $a \in I$,
\[\PP(D^n_j = a|X^n) = \phi_j^a(X^n).\]
\end{enumerate}
Then, 
\[(X^n,D^n) \Rightarrow (X,D) \defeq (X,D\pcar{1:k}),\]
where
\begin{enumerate}[label = (\Alph*)]
\item $D\pcar{1},\dots,D\pcar{k}$ are mutually conditionally independent given $X$;
\item for each $j \in [1:k]$ and $a \in I$, 
\[\PP(D\pcar{j} = a|X) = \phi_j^a(X).\]
\end{enumerate}
\end{lemma}
\begin{proof}
Fix any $f \in C_b(\Xmc\times I^k)$. For each $\bm{a} \in I^k$, define $f_{\bm{a}} \in C_b(\Xmc)$ by $f_{\bm{a}}(x) = f(x,\bm{a})$. Define the bounded, $X$-a.s. continuous function $g_{\bm{a}}:\Xmc\to\R$ by
\[g_{\bm{a}}(x) = f_{\bm{a}}(x)\prod_{j = 1}^k \phi_j^{a_j}(x),\]
and define the bounded, $X$-a.s. continuous function $g: \Xmc\to\R$ by
\[g(x) = \sum_{\bm{a} \in I^k} g_{\bm{a}}(x).\]

\ind Suppose that $(X,D)$ satisfies conditions (A) and (B) above. Then,
\begin{align*}
\ex{f(X,D)} &= \sum_{\bm{a} \in I^k}\ex{f_{\bm{a}}(X)\indic{D = \bm{a}}}\\
&= \sum_{\bm{a} \in I^k}\ex{f_{\bm{a}}(X)\PP(D=\bm{a}|X)}\\
&\overset{(A),(B)}{=}\sum_{\bm{a} \in I^k}\ex{f_{\bm{a}}(X)\prod_{j=1}^k \phi_j^{a_j}(X)}\\
&= \ex{\sum_{\bm{a} \in I^k}g_{\bm{a}}(X)}\\
&= \ex{g(X)}.
\end{align*}
By conditions (b) and (c), note that $(X^n,D^n)$ also satisfies conditions (A) and (B) for all $n \in \N$, so 
\[\ex{f(X^n,D^n)} = \ex{g(X^n)},\]
for all $n \in \N$. Then, by condition (a) and the fact that $g$ is bounded and $X$-a.s. continuous,
\[\lim_{n\to\infty} \ex{f(X^n,D^n)} = \lim_{n\to\infty} \ex{g(X^n)} = \ex{g(X)} = \ex{f(X,D)},\]
completing the proof that $(X^n,D^n) \Rightarrow (X,D)$ where $(X,D)$ satisfies conditions (A) and (B).
\end{proof}

\subsubsection{Mutual Conditional Independence of Bernoulli Random Vectors}
\label{CIBRV}

We start with a couple of simple lemmas. The first one reduces the task of establishing mutual conditional independence to the task of establishing simple conditional independence.

\begin{lemma}[Sufficient Condition for Mutual Conditional Independence]
\label{CIsuff}
A collection of $\Ymc$-random elements $(Y_i)_{i=1}^k$ is mutually conditionally independent given a $\sigma$-algebra $\Fmc$ if and only if for any $j \in [k]$,
\begin{equation}
\label{CIsuffjs}
Y_j \indp (Y_i)_{i\neq j}|\Fmc.
\end{equation}

\end{lemma}
\begin{proof}
The ``only if" direction is an immediate consequence of the definition of mutual conditional independence.

\ind We now prove the ``if" direction. For each $i \in [k]$, let $f_i: \Ymc \to \R$ be some bounded, measurable function. For each $i \in [k]$, define $g_i: \Ymc^{k-i+1} \to \R$ by
\[g_i(y_i,y_{i+1},\dots,y_k) = \prod_{j=i}^k f_j(y_j).\]
Then applying \eqref{CIsuffjs} sequentially for $j = 1$, then $j = 2$, etc.
\begin{align*}
\ex{\prod_{i=1}^n f_i(Y_i)\middle|\Fmc} &= \ex{f_1(Y_1)g_2(Y_2,\dots,Y_k)\middle|\Fmc}\\
&= \ex{f_1(Y_1)\middle|\Fmc}\ex{\prod_{i=2}^n f_i(Y_i)\middle|\Fmc}\\
&= \ex{f_1(Y_1)\middle|\Fmc}\ex{f_2(Y_2)g_3(Y_3,\dots,Y_k)\middle|\Fmc}\\
&= \cdots\\
&= \prod_{i=1}^n \ex{f_i(Y_i)\middle|\Fmc}.
\end{align*}
That concludes the proof.
\end{proof}

\begin{lemma}[Mutual Conditional Independence of Finite-State Random Variables]
\label{lem:Bercondindsubfilt}
Let $\Gmc \subseteq \Fmc$ be two sigma algebras and let $I$ be a finite state space. Let $\{D_i\}_{j \in J}$ be a collection of $I$-valued random variables indexed by some finite set $J$ that are conditionally independent given $\Fmc$. If $P^a_j \defeq \PP(D_j = a|\Fmc)$ is $\Gmc$-measurable for all $a \in I$ and $j\in J$, then $\{D_j\}_{j\in J}$ are also conditionally independent given $\Gmc$.
\end{lemma}
\begin{proof}
Notice that for any $j \in J$ and $a \in I$,
\[\PP\left(D_j = a\middle|\Gmc\right) = \ex{\PP\left(D_j = a\middle|\Fmc\right)\middle|\Gmc} = \ex{P_j^a|\Gmc} = P_j^a.\]
Let $\bm{a}\defeq \{a_j\}_{j \in J} \in I^J$ be any constant. Then
\begin{align*}
\PP\left(D_j = a_j\te{ for all }j \in J\middle|\Gmc\right) &= \ex{\PP\left(D_j = a_j\te{ for all }j \in J\middle|\Fmc\right)\middle|\Gmc}\\
&= \ex{\prod_{j \in J} \PP\left(D_j = a_j\middle|\Fmc\right)\middle|\Gmc}\\
&= \ex{\prod_{j \in J}P_j^{a_j}\middle|\Gmc}\\
&= \prod_{j \in J}P_j^{a_j}\\
&= \prod_{j \in J} \PP\left(D_j = a_j\middle|\Gmc\right).
\end{align*}
\end{proof}
\subsubsection{Convergence of Random Measures}

We begin with a sufficient condition under which a sequence of random measures converges in probability. We believe the following lemma is known. However, we were unable to find a suitable reference.

\begin{lemma}[Convergence in Probability of Random Measures]
\label{convinprobsuff}
Assume that $\Xmc$ is locally compact. Let $\{\eta_n\}_{n \in \N}$ be a sequence of random probability measures in $\P(\Xmc)$. Suppose there exists a random $\P(\Xmc)$-element $\eta$ such that for all $f \in C_b(\Xmc)$, $\langle \eta_n, f\rangle \to \langle \eta, f\rangle$ in probability. Then $\eta_n \to \eta$ in probability.
\end{lemma}
\begin{proof}
Ideally we would consider a sequence $\{f_k\}_{k \in \N}$ that is dense in $C_b(\Xmc)$. However, $C_b(\Xmc)$ is not necessarily separable in the topology of uniform convergence. Instead, let $\{f_k\}_{k \in \N}$ be a collection of functions that are dense in the space $C_0(\Xmc)\subseteq C_b(\Xmc)$ (which is separable when $\Xmc$ is locally compact).

\ind Let $\{n_m\}_{m \in \N}\subseteq \N$ be an arbitrary, strictly increasing subsequence of $\N$ such that $n_m \nearrow \infty$. Then by a standard diagonalization argument, there exists a subsubsequence (also strictly increasing to infinity)  $\{n_{m_{\ell}}\}_{m \in \N}\subseteq \{n_m\}_{m \in \N}$ such that 
\[\left\langle \eta_{n_{m_{\ell}}}, f_k\right\rangle \to \langle \eta,f_k\rangle \te{ a.s. for all }k \in \N.\]

Let $f \in C_0(\Xmc)$ and fix a sequence $\{k_i\}_{i\in \N}\subseteq \N$ such that $f_{k_{i}} \to f$ uniformly. Then 
\[\lim_{i\to\infty} \sup_{\ell \in \N} \left|\left\langle \eta_{n_{m_{\ell}}}, f_{k_{i}}\right\rangle - \left\langle \eta_{n_{m_{\ell}}}, f\right \rangle\right| \leq  \lim_{i\to\infty} \left\|f_{k_i} - f\right\|_{\infty} = 0.\]
For any $i \in \N$,
\[\lim_{\ell \to \infty} \left\langle \eta_{n_{m_{\ell}}}, f_{k_i}\right\rangle = \left\langle \eta, f_{k_i}\right\rangle \te{ a.s.}.\]
This shows that the assumptions of the Moore-Osgood double limit theorem (see \cite[Theorem 7.11]{Rud76} or \cite[Chapter 5, Theorem 3]{Hof75} for example) are satisfied by the doubly indexed sequence $\left\{\left\langle \eta_{n_{m_{\ell}}},f_{k_i}\right\rangle\right\}_{i,\ell \in \N}$, so we apply the Moore-Osgood theorem to interchange limits below:
\begin{align*}
\lim_{\ell\to\infty} \left\langle \eta_{n_{m_{\ell}}},f\right\rangle &= \lim_{\ell\to\infty} \lim_{i\to\infty}\left\langle \eta_{n_{m_{\ell}}},f_{k_i}\right\rangle = \lim_{i\to\infty} \lim_{\ell\to\infty} \left\langle \eta_{n_{m_{\ell}}},f_{k_i}\right\rangle = \lim_{i\to\infty} \left\langle \eta, f_{k_i}\right\rangle = \langle \eta,f\rangle \te{ a.s..}
\end{align*}
Thus, $\eta_{n_{m_{\ell}}} \to \eta$ vaguely a.s.. Because $\eta_{n_{m_{\ell}}}$ is a probability measure for all $\ell \in \N$, $\Xmc$ is a locally compact Polish space and $\eta$ is also a probability measure, $\eta_{n_{m_{\ell}}} \to \eta$ a.s. in $\P(\Xmc)$ \cite[Exercise 26, Chapter 7]{Fol99}. Therefore, $\eta_n \to \eta$ in $\P(\Xmc)$ in probability. 
\end{proof}

\subsubsection{Convergence of Integrals of Random Measures}

If a sequence of random measures $\eta_X^n \to \eta_X$ in probability, then for any bounded, continuous $f$, $\langle \eta_X^n,f\rangle \to \langle \eta_X,f\rangle$ in probability. However, if $f$ is continuous but unbounded, then the map $\eta \mapsto \langle \eta,f\rangle$ is no longer continuous. We establish sufficient conditions under which $\langle \eta_X^n,f\rangle$ still converges to $\langle \eta_X,f\rangle$ in probability.

\begin{lemma}
\label{unifintlem}
Let $\{X^n_i\}_{n\in \N,i\in [1:n]}$ be a uniformly integrable triangular array of $\R^d$-random vectors such that 
\[\eta^n_X \defeq \frac{1}{n}\sum_{i=1}^n \delta_{X^n_i}\to \eta_X \in \P(\R^d) \te{ in probability,}\]
where $\eta_X$ is a possibly random probability measure. Then,
\[\frac{1}{n}\sum_{i=1}^n X^n_i \to \int_{\R^d}x\,\eta_X(dx)\te{ in probability.}\]
\end{lemma}
\begin{proof}
The following argument is adapted from the proof of \cite[Theorem 3.4(b)]{BouDupEll00}. Let $\phi: \R^d\to\R^d$ be the identity map. For any $C < \infty$, define $\phi_C: \R^d\to\R^d$ by
\[\phi_C(x) = x\indic{|x|> C}.\]
By uniform integrability of $\{X^n_i\}_{n \in \N,i \in [1:n]}$ there exists for each $C \in (0,\infty)$ a constant $M_C$ converging to $0$ as $C\to \infty$ such that
\[\sup_{n\in \N,i\in[1:n]} \ex{|\phi_C(x^n_i)|} \leq M_C.\]
It follows that for any $n\in \N$,
\begin{align*}
\ex{|\langle \eta^n_X,\phi_C\rangle|} \leq \ex{\frac{1}{n}\sum_{i=1}^n |\phi_C(X^n_i)|} =\frac{1}{n}\sum_{i=1}^n \ex{|\phi_C(x^n_i)|} \leq M_C.
\end{align*}

Thus,

\begin{align*}
\lim_{C\to\infty}\sup_n\ex{|\langle \eta^n_X,\phi_C\rangle|} \leq \lim_{C\to\infty} M_C = 0.
\end{align*}

If we also define

\[\phi^C(x) \defeq \begin{cases}
x &\te{ if } |x| < C,\\
C\frac{x}{|x|} &\te{ if } |x| \geq C.
\end{cases}\]

Note that $|\phi_C+\phi^C| \geq |\phi|$. Moreover, there exists a sequence $M'_C \to 0$ as $C\to\infty$ such that $\ex{|\langle \eta_X,\phi_C\rangle|} \leq M'_C$ for all $C \in (0,\infty)$. Lastly, because $\phi^C$ is a bounded, continuous function, $\{\langle\eta^n_X,\phi^C\rangle\}_{n \in \N}$ is a sequence of random vectors whose magnitudes are uniformly bounded by $C$ and that converge to $\langle \eta_X,\phi^C\rangle$ in probability. Thus, for any $\ep >0$,

\begin{align*}
\lim_{n\to\infty} \PP(|\langle \eta^n_X,\phi\rangle &- \langle \eta_X,\phi\rangle| > \ep) \leq \lim_{n\to\infty} \frac{1}{\ep}\ex{|\langle \eta^n_X,\phi\rangle - \langle \eta_X,\phi\rangle|}\\
&\leq \inf_{C \in (0,\infty)} \lim_{n\to\infty} \frac{1}{\ep}\left(\ex{|\langle \eta^n_X,\phi^C\rangle - \langle \eta_X,\phi^C\rangle|} + \ex{|\langle \eta^n_X,\phi_C\rangle|} + \ex{|\langle \eta_X,\phi_C\rangle|}\right)\\
&\leq \inf_{C \in (0,\infty)} \lim_{n\to\infty} \frac{1}{\ep}\left(\ex{|\langle \eta^n_X,\phi^C\rangle - \langle \eta_X,\phi^C\rangle|} + M_C + M'_C\right)\\
&= \inf_{C \in (0,\infty)} \frac{1}{\ep}\left(M_C + M'_C\right)\\
&=0.
\end{align*}

Thus,

\[\frac{1}{n}\sum_{i=1}^n X^n_i = \langle \eta^n_X,\phi\rangle \to \langle \eta_X,\phi\rangle = \int_{\R^d} x\,\eta_X(dx) \te{ in probability}\]
completing the proof.
\end{proof}

This immediately implies the following simple corollary.

\begin{corollary}
\label{unifintflem}
Suppose that $\Xmc$ is locally compact. Let $\{Y^n_i\}_{n\in \N,i\in [1:n]}$ be a triangular array of $\Xmc$-random elements such that 
\[\eta^n_Y \defeq \frac{1}{n}\sum_{i=1}^n \delta_{Y^n_i}\to \eta_Y \in \P(\Xmc) \te{ in probability,}\]
where $\eta_Y$ is a possibly random probability measure. Let $f: \Xmc \to \R^d$ be a continuous function such that $\{f(Y^n_i)\}_{n \in \N, i \in [1:n]}$ is uniformly integrable. Then,
\[\frac{1}{n}\sum_{i=1}^n f(Y^n_i) \to \int_{\R^d}f(y)\,\eta_Y(dy)\te{ in probability.}\]
\end{corollary}
\begin{proof}
Define 
\[\eta^n_f \defeq \frac{1}{n}\sum_{i=1}^n \delta_{f(Y^n_i)}\te{ and } \eta_f \defeq f_*\eta_Y,\]
so for any $A \in \borel(\R^d)$, $\eta_f(A) = \eta_Y\left(\{y: f(y) \in A\}\right)$. Fix any $g \in C_b(\R^d)$. Then $g\circ f \in C_b(\Xmc)$, so
\[\langle \eta^n_f,g\rangle = \langle \eta^n_Y,g\circ f\rangle \to \langle \eta_Y, g\circ f\rangle = \langle \eta_f,g\rangle\te{ in probability,}\]
so by Lemma \ref{convinprobsuff}, $\eta^n_f \to \eta_f$ in probability. The result now holds by Lemma \ref{unifintlem} setting $X^n_i = f(Y^n_i)$, $\eta^n_X = \eta^n_f$ and $\eta_X = \eta_f$.
\end{proof}

\subsection{Establishing Exchangeability}

\begin{lemma}
\label{ZLexch}
Suppose Property A(a) holds at time $t$. Then the collection $((Z^n[t],L^n(t)),A^n(t))$ is jointly exchangeable.
\end{lemma}
\begin{proof}
Fix any $\sigma \in S_n$. Then,
\begin{align*}
((Z^n[t],L^n(t)),A^n(t)) &= \left(\left(Z^n_i[t], \frac{\sum_{k=1}^n Z^n_k(t)A^n_{ik}(t)}{\sum_{k=1}^n A^n_{ik}(t)}\right)_{i \in [1:n]}, (A^n_{ij}(t))_{i,j \in [1:n]}\right)\\
&\deq\left(\left(Z^n_{\sigma(i)}[t], \frac{\sum_{k=1}^n Z^n_{\sigma(k)}(t)A^n_{\sigma(i)\sigma(k)}(t)}{\sum_{k=1}^n A^n_{\sigma(i)\sigma(k)}(t)}\right)_{i \in [1:n]}, (A^n_{\sigma(i)\sigma(j)}(t))_{i,j \in [1:n]}\right)\\
&=\left(\left(Z^n_{\sigma(i)}[t], \frac{\sum_{k=1}^n Z^n_k(t)A^n_{\sigma(i)k}(t)}{\sum_{k=1}^n A^n_{\sigma(i)k}(t)}\right)_{i \in [1:n]}, (A^n_{\sigma(i)\sigma(j)}(t))_{i,j \in [1:n]}\right)\\
&= \left(\left(Z^n_{\sigma(i)}[t],L^n_{\sigma(i)}(t)\right)_{i \in [1:n]},(A^n_{\sigma(i)\sigma(j)}(t))_{i,j\in [1:n]}\right).
\end{align*}
Thus, $(Z^n[t],L^n(t),A^n(t))$ is jointly exchangeable.
\end{proof}

\begin{lemma}
\label{Addbernexch}
Let $X \defeq (X_i)_{i \in [1:n]}$ and $Y \defeq (Y_{ij})_{i,j \in [1:n]}$ be $\Xmc^n$ and $\Ymc^{n\times n}$ random elements respectively such that $(X,Y)$ is jointly exchangeable. Let $\{B_{ij}\}_{(i,j)\in\incs_n}$ be a conditionally mutually independent sequence of Bernoulli random variables (given $(X,Y)$) and let $B_{ij} = B_{ji}$ for all $i,j \in [1:n]$. If for each $i,j$, $P_{ij} = \PP(B_{ij}=1|X,Y)$ and $(X,(Y,P))$ is jointly exchangeable, then $(X_{1:k},(Y_{ij},B_{ij})_{i,j \in [1:n]})$ is jointly exchangeable.
\end{lemma}
\begin{proof}
Fix any $\sigma \in S_n$ and any bounded, measurable $g: \Xmc^n \times (\Ymc\times \{0,1\}\times [0,1])^{n\times n}\to \R$. Then,
\begin{align*}
&\ex{g\left((X_i)_{i\in [1:n]}, (Y_{ij},B_{ij},P_{ij})_{i,j\in [1:n]}\right)}\\
&=\ex{\ex{g\left((X_i)_{i\in [1:n]}, (Y_{ij},B_{ij},P_{ij})_{i,j\in [1:n]}\right)\middle|X,Y,P}}\\
&= \sum_{\substack{b \in \{0,1\}^{n\times n}\\ b\te{ is symmetric}}} \ex{g\left((X_i)_{i\in [1:n]}, (Y_{ij},b_{ij},P_{ij})_{i,j\in [1:n]}\right)\prod_{\substack{1\leq i\leq j\leq n\\ b_{ij} = 1}} P_{ij}\prod_{\substack{1\leq i\leq j\leq n\\ b_{ij} = 0}} (1-P_{ij})}\\
&= \sum_{\substack{b \in \{0,1\}^{n\times n}\\ b\te{ is symmetric}}} \mathbb{E}\Bigg[g\left((X_{\sigma(i)})_{i\in [1:n]}, (Y_{\sigma(i)\sigma(j)},b_{ij},P_{\sigma(i)\sigma(j)})_{i,j\in [1:n]}\right)\prod_{\substack{1\leq i\leq j\leq n\\ b_{ij} = 1}} P_{\sigma(i)\sigma(j)}\\
&\ind \prod_{\substack{1\leq i\leq j\leq n\\ b_{ij} = 0}} (1-P_{\sigma(i)\sigma(j)})\Bigg]\\
&= \sum_{\substack{b' \in \{0,1\}^{n\times n}\\ b'\te{ is symmetric}}} \mathbb{E}\Bigg[g\left((X_{\sigma(i)})_{i\in [1:n]}, (Y_{\sigma(i)\sigma(j)},b'_{\sigma(i)\sigma(j)},P_{\sigma(i)\sigma(j)})_{i,j\in [1:n]}\right)\\
&\ind \prod_{\substack{1\leq i\leq j\leq n\\ b'_{\sigma(i)\sigma(j)} = 1}} P_{\sigma(i)\sigma(j)}\prod_{\substack{1\leq i\leq j\leq n\\ b'_{\sigma(i)\sigma(j)} = 0}} (1-P_{\sigma(i)\sigma(j)})\Bigg]\\
&=\ex{g\left((X_{\sigma(i)})_{i\in [1:n]}, (Y_{\sigma(i)\sigma(j)},B_{\sigma(i)\sigma(j)},P_{\sigma(i)\sigma(j)})_{i,j\in [1:n]}\right)}.
\end{align*}
Thus, $(X,(Y,P,B))$ is jointly exchangeable, so $(X,(Y,B))$ is also jointly exchangeable.
\end{proof}

\begin{lemma}
\label{lem::Addberexchexcl}
Let $(X,Y)$ be an $\Xmc\times \Ymc^n$-random element that is exchangeable excluding $1$. Let $C: \Xmc\times\Ymc \to [0,1]$ be a measurable function and let $\{B_i\}_{i \in [1:n]}$ be a collection of conditionally independent (given $(X,Y)$) Bernoulli random variables with respective parameters $P_i \defeq C(X,Y_i)$. Then $(X,(Y,B))$ is also exchangeable excluding 1.
\end{lemma}
\begin{proof}
The proof is very similar to the proof of Lemma \ref{Addbernexch}. Let $\sigma \in S_n$ be such that $\sigma(1) = 1$. Then 
\[(X,Y_i,P_i)_{i \in [1:n]} = (X,Y_i,C(X,Y_i))_{i \in [1:n]}\deq (X,Y_{\sigma(i)},C(X,Y_{\sigma(i)}))_{i \in [1:n]} = (X,Y_{\sigma(i)},P_{\sigma(i)}),\]
so $(X,Y,P)$ is exchangeable excluding 1. Then for any bounded, measurable $f: \Xmc\times \Ymc^n\times [0,1]^n\times \{0,1\}^n \to \R$,
\begin{align*}
\ex{f(X,(Y_i,P_i,B_i)_{i\in [1:n]}} &= \sum_{b \in \{0,1\}^n}\ex{f(X,(Y_i,P_i,b_i)_{i\in [1:n]}\prod_{\substack{i \in [1:n]\\ b_i = 1}} P_i\prod_{\substack{i \in [1:n]\\ b_i = 0}}(1-P_i)}\\
&= \sum_{b \in \{0,1\}^n}\ex{f(X,(Y_{\sigma(i)},P_{\sigma(i)},b_i)_{i\in [1:n]}\prod_{\substack{i \in [1:n]\\ b_i = 1}} P_{\sigma(i)}\prod_{\substack{i \in [1:n]\\ b_i = 0}}(1-P_{\sigma(i)})}\\
&= \sum_{b' \in \{0,1\}^n}\ex{f(X,(Y_{\sigma(i)},P_{\sigma(i)},b'_{\sigma(i)})_{i\in [1:n]}\prod_{\substack{i \in [1:n]\\ b'_{\sigma(i)} = 1}} P_{\sigma(i)}\prod_{\substack{i \in [1:n]\\ b'_{\sigma(i)} = 0}}(1-P_{\sigma(i)})}\\
&= \ex{f(X,(Y_{\sigma(i)},P_{\sigma(i)},B_{\sigma(i)})_{i \in[1:n]})},
\end{align*}
where in the third equality, we use the tranformation $(b'_i)_{i\in [1:n]} = (b_{\sigma^{-1}(i)})_{i\in[1:n]}$. Thus, $(X,(Y,P,B))$ is exchangeable excluding 1, so $(X,(Y,B))$ is exchangeable excluding 1.
\end{proof}

\subsection{Other General Lemmas}
\label{usefulother}

Another useful lemma concerns a representation of the product of conditional probability measures.

\begin{lemma}
\label{lem::CIgivenprod}
Let $\eta$ be a $\P(\Xmc\times \Ymc)$-random element defined by $\eta = \law(X,Y|X)$ for respective $\Xmc$ and $\Ymc$-random elements $X$ and $Y$. Now suppose that for some $k\geq 3$, $(X,Y_2,\dots,Y_k)$ satisfies the following conditions:
\begin{itemize}
\item for $j =2,\dots,k$, $\law(X,Y_j|X) = \eta$,
\item $(Y_j)_{j=2}^k$ are conditionally independent given $X$.
\end{itemize}
Then
\[(\eta)^{k-1} = \law\left((X,Y_j)_{j=2}^k\middle|X\right).\]
\end{lemma}
\begin{proof}
Because $(\eta)^{k-1}$ is $\sigma(X)$-measurable, it suffices to show that
\[\ex{\langle (\eta)^{k-1}, f\rangle h(X)} = \ex{f\left((X,Y_j)_{j=2}^k\right)h(X)}\]
for all bounded, measurable functions $f: (\Xmc\times \Ymc)^{k-1}\to \R$ and $h: \Xmc\to \R$.

\ind Let us write $\times_{j=2}^k A_{j-1}$ for the Cartesian product of the sets $A_{1},\cdots,A_{k-1}$. Suppose that $f\left((x_{j-1},y_j)_{j=2}^k\right) = \prod_{j=2}^k\indic{A_{j-1}}(x_{j-1},y_j)$ for the borel sets $A_1,\dots,A_{k-1} \in \borel(\Xmc\times\Ymc)$. Then, 
\begin{align*}
\ex{\langle(\eta)^{k-1},f\rangle h(X)} &= \ex{\prod_{j=1}^{k-1} \eta(A_j) h(X)}\\
&= \ex{\prod_{j=1}^{k-1}\PP\left((X,Y)\in A_j\middle|X\right)h(X)}\\
&= \ex{\PP\left((X,Y_j)_{j=2}^k \in \times_{j=2}^k A_{j-1}\middle|X\right)h(X)}\\
&= \ex{\indic{\times_{j=2}^k A_{j-1}}\left((X,Y_j)_{j=2}^k\right)h(X)}\\
&= \ex{f\left((X,Y_j)_{j=2}^k\right)h(X)}.
\end{align*}
It is easily seen that the collection of bounded, measurable functions $f: (\Xmc\times\Ymc)^{k-1}\to\R$ such that $\ex{\langle (\eta)^{k-1},f\rangle h(X)} = \ex{f\left((X,Y_j)_{j=2}^k\right)h(X)}$ forms a monotone class, and the display above shows that this monotone class includes all functions of the form $f\left((x_{j-1},y_j)_{j=2}^k\right) = \indic{\times_{j=2}^kA_{j-1}}$ for $A_1,\dots,A_{k-1} \in \borel(\Xmc\times\Ymc)$. Since the set of sets $\left\{\times_{j=2}^k A_{j-1}: A_1,\dots,A_{k-1} \in \borel(\Xmc\times\Ymc)\right\}$ is a $\pi$-system, we may use the monotone class theorem \cite[Theorem 5.2.2]{Dur19} to show that
\[\ex{\langle (\eta)^{k-1},f\rangle h(X)} = \ex{f\left((X,Y_j)_{j=2}^k\right)h(X)}\]
for all bounded, measurable $f: (\Xmc\times\Ymc)^{k-1}\to\R$. This completes the proof.
\end{proof}

\bibliographystyle{plain}
\bibliography{reference}
\end{document}